\documentclass[11pt, oneside]{amsart}

\usepackage[pdfauthor={Ziyang GAO}, 
        pdftitle={Betti rank}%
      dvips,colorlinks=true]{hyperref}
\hypersetup{
  bookmarksnumbered=true,
  linkcolor=black,
  citecolor=black,
  pagecolor=black, 
  urlcolor=black,  
}

\usepackage{amsmath,amssymb,cite,mathrsfs,tikz-cd}
\usepackage[all]{xy}
\usepackage{typearea} 
\usepackage{hyperref} 
\usepackage{color} 

\usepackage{epstopdf} 
\usepackage{booktabs}

\newtheorem{teo}{Theorem}[section]
\newtheorem{thm}[teo]{Theorem}
\newtheorem{prop}[teo]{Proposition}
\newtheorem{lemma}[teo]{Lemma}
\newtheorem{cor}[teo]{Corollary}
\newtheorem{conj}[teo]{Conjecture}
\newtheorem{eg}[teo]{Example}
\newtheorem{defn}[teo]{Definition}

\newtheorem{rmk}[teo]{Remark}

\newtheorem{def-prop}[teo]{Definition-Proposition}

\newtheoremstyle{named}{}{}{\itshape}{}{\bfseries}{.}{.5em}{\thmnote{#3 }#1}
\theoremstyle{named}
\newtheorem*{namedques}{Question}
\newtheorem*{namedconj}{Conjecture}

\makeatletter
\newcommand{\neutralize}[1]{\expandafter\let\csname c@#1\endcsname\count@}
\makeatother

\newenvironment{thmbis}[1]
  {\renewcommand{\theteo}{\ref*{#1}$'$}%
   \neutralize{teo}\phantomsection
   \begin{thm}}
  {\end{thm}}

\numberwithin{equation}{section}

  \newcommand{\A}{\mathbb{A}}
  \newcommand{\C}{\mathbb{C}}

  \newcommand{\Q}{\mathbb{Q}}
  \newcommand{\R}{\mathbb{R}}

  \newcommand{\Z}{\mathbb{Z}}

  \newcommand{\bu}{\mathbf{u}}

  \renewcommand{\cong}{\simeq}
  \renewcommand{\bar}{\overline}
  \renewcommand{\tilde}{\widetilde}
  \renewcommand{\hat}{\widehat}

  \providecommand{\frac}[1]{\operatorname{Frac}(#1)}

  \renewcommand{\ker}{\operatorname{Ker}}

  \newcommand{\codim}{\operatorname{codim}}

  \renewcommand{\lim}{\operatorname{lim}}

  \newcommand{\Zar}{\mathrm{Zar}}
  \newcommand{\biZar}{\mathrm{biZar}}
  \newcommand{\GSp}{\mathrm{GSp}}
  \newcommand{\Sp}{\mathrm{Sp}}

\newcommand{\cA}{\mathcal{A}}
\newcommand{\cB}{\mathcal{B}}
\newcommand{\cC}{\mathcal{C}}

\newcommand{\cX}{\mathcal{X}}
\newcommand{\cY}{\mathcal{Y}}
\newcommand{\cZ}{\mathcal{Z}}

\newcommand{\pullbackcorner}[1][dr]{\save*!/#1-1.7pc/#1:(-1.5,1.5)@^{|-}\restore}

\newcommand\supervisor[1]{\def\@supervisor{#1}}

\newcounter{elno}

\renewcommand{\cong}{\simeq}

\begin{document}
\title[]{Generic rank of Betti map and unlikely intersections}
\author{Ziyang Gao}

\address{CNRS, IMJ-PRG, 4 place Jussieu, 75005 Paris, France}
\email{ziyang.gao@imj-prg.fr}


\maketitle

\begin{abstract} Let $\cA \rightarrow S$ be an abelian scheme over an irreducible variety over $\C$ of relative dimension $g$. For any simply-connected subset $\Delta$ of $S^{\mathrm{an}}$ one can define the \textit{Betti map} from $\cA_{\Delta}$ to $\mathbb{T}^{2g}$, the real torus of dimension $2g$, by identifying each closed fiber of $\cA_{\Delta} \rightarrow \Delta$ with $\mathbb{T}^{2g}$ via the Betti homology. Computing the generic rank of the Betti map restricted to a subvariety $X$ of $\cA$ is useful to study Diophantine problems, \textit{e.g.} proving the Geometric Bogomolov Conjecture over char $0$ and studying the relative Manin--Mumford conjecture. In this paper we give a~geometric criterion to detect this rank. As an application we show that it is maximal after taking a large fibered power (if $X$ satisfies some conditions); it is an important step to prove the bound for the number of rational points on curves \cite{DGHMazur}. Another application is to answer a question of Andr\'{e}--Corvaja--Zannier and improve a result of Voisin. We also systematically study its link with the relative Manin-Mumford conjecture, reducing the latter to a simpler conjecture. Our tools are functional transcendence and unlikely intersections for mixed Shimura varieties.
\end{abstract}

\tableofcontents

\section{Introduction}
Let $S$ be an irreducible variety over $\mathbb{C}$, and let $\pi_S \colon \cA \rightarrow S$ be an abelian scheme of relative dimension $g$, namely a proper smooth group scheme whose fibers are abelian varieties.

For any $s \in S(\C)$, there exists an open neighborhood $\Delta \subseteq S^{\mathrm{an}}$ of $s$ with a real-analytic map, which we call the \textbf{Betti map},
\[
b_{\Delta} \colon \cA_{\Delta} = \pi_S^{-1}(\Delta) \rightarrow \mathbb{T}^{2g},
\]
where $\mathbb{T}^{2g}$ is the real torus of dimension $2g$. The precise definition will be given in \eqref{EqBettiMapActualScheme}, but we give a brief explanation here: Up to shrinking $\Delta$ we may assume that it is simply-connected. Then one can define a basis $\omega_1(s),\ldots,\omega_{2g}(s)$ of the period lattice of each fiber $s \in \Delta$ as holomorphic functions of $s$. Now each fiber $\cA_s = \pi_S^{-1}(s)$ can be identified with the complex torus $\C^g/\Z \omega_1(s)\oplus \cdots \oplus \Z\omega_{2g}(s)$, and each point $x \in \cA_s(\C)$ can be expressed as the class of $\sum_{i=1}^{2g}b_i(x) \omega_i(s)$ for real numbers $b_1(x),\ldots,b_{2g}(x)$. Then $b_{\Delta}(x)$ is defined to be the class of the $2g$-tuple $(b_1(x),\ldots,b_{2g}(x)) \in \R^{2g}$ modulo $\Z^{2g}$.


Through the introduction, let $X$ be a closed irreducible subvariety of $\cA$ such that $\pi_S(X) = S$.

The \textbf{\textit{goal}} of this paper is to compute the generic rank of $b_{\Delta}|_{X \cap \cA_{\Delta}}$. We also systematically study the link between this rank and the relative Manin--Mumford conjecture, reducing the latter 
to a simpler conjecture on unlikely intersections.

First of all, $\mathrm{rank}_{\R} (\mathrm{d}b_{\Delta}|_X)_x$ must be an even number for each $x \in X^{\mathrm{sm}}(\C)$, as each $\omega_i(s)$ is holomorphic. So the question becomes giving a characterization for
\begin{equation}\label{EqBettiRankMaxIneqIntro}
\mathrm{rank}_{\R} (\mathrm{d}b_{\Delta}|_X) := \max_{x \in X^{\mathrm{sm}}(\C) \cap \cA_{\Delta}}(\mathrm{rank}_{\R}(\mathrm{d}b_{\Delta}|_{X\cap \cA_{\Delta}})_x) = 2l.
\end{equation}
Also it can be shown that $\mathrm{rank}_{\R} (\mathrm{d}b_{\Delta}|_X)$ does not depend on the choice of $\Delta$ by Sard's theorem; see the end of $\mathsection$\ref{SectionBettiRevisited} for a detailed explanation.

There is a naive bound for $\mathrm{rank}_{\R} (\mathrm{d}b_{\Delta}|_X)$ as follows. Let $\mathfrak{A}_g$ be the universal abelian variety over the moduli space $\A_g$ of abelian varieties. We have the modular map
\begin{equation}\label{EqModMapIntro}
\xymatrix{
\cA \ar[r]^{\iota} \ar[d]_{\pi_S} \pullbackcorner & \mathfrak{A}_g \ar[d]^{\pi} \\
S \ar[r]^{\iota_S} & \A_g.
}
\end{equation}
It is not hard to show that $b_{\Delta}$ factors through $\iota$. So $\mathrm{rank}_{\R} (\mathrm{d}b_{\Delta}|_X) \le 2 \min(\dim \iota(X), g)$.

The question to characterize \eqref{EqBettiRankMaxIneqIntro} was explicitly asked and systematically studied for the first time in \cite{ACZBetti}. The submersivity problem, \textit{i.e.} $l=g$, was solved in \textit{loc.cit} for all $\cA/S$ with $\mathrm{End}(\cA/S) = \Z$ and $\dim \iota_S(S) \ge g$. A clear connection to functional transcendence (Andr\'{e}'s independence of abelian logarithms \cite{AndreMumford-Tate-gr} and pure Ax-Schanuel \cite{MokAx-Schanuel-for}) was presented.

In our paper we completely solve the \textit{generic rank} problem by giving a criterion (equivalent condition) to \eqref{EqBettiRankMaxIneqIntro} for each $l$ \textit{in simple terms of the geometry of $X$}. Our approach is independent of \cite{ACZBetti}. It uses functional transcendence (in a different vein, mixed Ax-Schanuel \cite{GaoAxSchanuel}), combined with a finiteness result \cite[Theorem~1.4]{GaoAxSchanuel}. 
Here let us state the criterion in an equivalent (but simpler formulated) way: give a criterion to $\mathrm{rank}_{\R} (\mathrm{d}b_{\Delta}|_X) < 2l$ for each $l$.

The following notation will be used throughout the paper: Denote by $\cA_X$ the translate of an abelian subscheme of $\cA \rightarrow S$ by a torsion section (up to a finite covering of $S$) which contains $X$, minimal for this property. Then $\cA_X \rightarrow S$ itself is an abelian scheme (up to taking a finite covering of $S$), whose relative dimension we denote by $g_X$.

\begin{thm}[Criterion to \eqref{EqBettiRankMaxIneqIntro}]\label{ThmCriterionDegIntro}
For each $l$, we have
$\mathrm{rank}_{\R} (\mathrm{d}b_{\Delta}|_X) < 2 l$ if and only if the following condition holds: 
there exists an abelian subscheme $\cB$ of $\cA_X \rightarrow S$ (of relative dimension $g_{\cB}$) such that for the quotient abelian scheme $p_{\cB} \colon \cA_X \rightarrow \cA_X/\cB$ and the modular map $\iota_{/\cB} \colon \cA_X/\cB \rightarrow \mathfrak{A}_{g_X - g_{\cB}}$, we have $\dim (\iota_{/\cB} \circ p_{\cB})(X) < l - g_{\cB}$.
\begin{equation}
\xymatrix{
\cA_X \ar[r]^-{p_{\cB}} \ar[d]_{\pi_S} & \cA_X/\cB \ar[r]^-{\iota_{/\cB}} \ar[d] \pullbackcorner & \mathfrak{A}_{g_X - g_{\cB}} \ar[d] \\
S \ar[r]^{\mathrm{id}_S} & S \ar[r]^-{\iota_{/\cB,S}} & \A_{g_X - g_{\cB}}
}
\end{equation}
\end{thm}


This geometric criterion to detect the Betti rank can be simplified in some cases, \textit{e.g.}
\begin{equation}\label{EqCriterionDegIntroSimpleCases}
\begin{array}{l}
\mathrm{rank}_{\R}(\mathrm{d}b_{\Delta}|_X) = 2 \min(\dim \iota(X), \dim \langle X \rangle_{\mathrm{gen-sp}} - \dim S) \\
 \qquad \text{if }\dim \iota_S(S) = 1\text{ or }\iota_S(S)\text{ has simple connected algebraic monodromy group}.
\end{array}
\end{equation}
See Definition~\ref{DefnGenSp} for the definition of $\langle X \rangle_{\mathrm{gen-sp}}$ and Corollary~\ref{CorBettiRankSimpleBase} for the statement.

Betti map was first studied and used in \cite{ZannierUnlikelyIntersectons}. Then it was used to study the relative Manin-Mumford conjecture by Bertrand, Corvaja, Masser, Pillay and Zannier in a series of works \cite{MasserZannierTorsionPointOnSqEC, MASSER2014116, MasserZannierRelMMSimpleSur, BMPZRelativeMM, CorvajaMasserZannier2018, MasserZannierRMMoverCurve}, to prove the geometric Bogomolov conjecture over char $0$ by Gao-Habegger \cite{GHGBC} and Cantat-Gao-Habegger-Xie \cite{CGHXGBC} with \eqref{EqBettiRankMaxIneqIntro} for $l = \dim X - \dim S$, and to prove the denseness of torsion points on sections of Lagrangian fibrations by Voisin \cite{VoisinBetti} using Andr\'{e}-Corvaja-Zannier's result \cite{ACZBetti} on \eqref{EqBettiRankMaxIneqIntro} for $l = g \le 4$.

\subsection{Two applications}
We see two applications of Theorem~\ref{ThmCriterionDegIntro} in this subsection.

As shown by \cite[Theorem~1.6]{DGHMazur}, in some number theory and algebro-geometric applications, it is particularly important to understand when \eqref{EqBettiRankMaxIneqIntro} holds for $l = \dim X$, namely
\begin{equation}\label{EqBettiRankMaxDIneqIntro}
\mathrm{rank}_{\R} (\mathrm{d}b_{\Delta}|_X) = 2 \dim X.
\end{equation}

As $\mathrm{rank}_{\R} (\mathrm{d}b_{\Delta}|_X) \le 2 \min(\dim \iota(X), g)$ always holds, \eqref{EqBettiRankMaxDIneqIntro} requires $\iota|_X$ to be generically finite and that $\dim X  \le g$. But in applications we often need to handle $X$ with $\dim X >g$.


Because of this problem, our main result towards \eqref{EqBettiRankMaxDIneqIntro} is in the following philosophy. Instead of proving  \eqref{EqBettiRankMaxDIneqIntro} for $X$, we raise $X$ to a large enough fibered power so that  \eqref{EqBettiRankMaxDIneqIntro} holds for this fibered power. In this process we need to put some extra assumptions on $X$.

Let us explain this in details. For any integer $m \ge 1$, denote by $\cA^{[m]} = \cA\times_S \ldots \times_S \cA$ ($m$-copies), $X^{[m]} = X \times_S \ldots \times_S X$ ($m$-copies) and by $b^{[m]}_{\Delta} = (b_{\Delta}, \ldots, b_{\Delta}) \colon \cA^{[m]}_{\Delta} \rightarrow \mathbb{T}^{2mg}$. Let $\mathscr{D}_m^{\cA} \colon \cA^{[m+1]} \rightarrow \cA^{[m]}$ be the $m$-th Faltings-Zhang map fiberwise defined by $(P_0, P_1, \ldots, P_m) \mapsto (P_1 - P_0, \ldots, P_m - P_0)$.

We start with the following example. Let $\mathbb{M}_g$ be the moduli space of curves of genus $g$. Up to taking a finite covering, we have a universal curve $\mathfrak{C}_g \rightarrow \mathbb{M}_g$. Let $\A_g$ be the moduli space of principally polarized abelian varieties of dimension $g$. Up to taking a finite covering, we have a universal abelian variety $\pi \colon \mathfrak{A}_g \rightarrow \A_g$. Then we have
\[
\xymatrix{
\mathrm{Jac}(\mathfrak{C}_g) \ar[r]^-{\iota} \ar[d] \pullbackcorner & \mathfrak{A}_g \ar[d]^{\pi} \\
\mathbb{M}_g \ar[r] & \A_g.
}
\]
By abuse of notation denote by $\mathbb{M}_g$ the image of the bottom arrow.

Let $S$ be an irreducible variety with a generically finite (not necessarily dominant) morphism $S \rightarrow \mathbb{M}_g$ such that $\mathfrak{C}_g \times_{\mathbb{M}_g} S \rightarrow S$ admits a section $\epsilon$. Take $\cA = \mathfrak{A}_g \times_{\A_g} S$, and let $\mathfrak{C}_S = \mathfrak{C}_g \times_{\mathcal{M}_g}S$. We have the fiberwise Abel-Jacobi embedding $j_{\epsilon} \colon \mathfrak{C}_S \hookrightarrow \mathrm{Jac}(\mathfrak{C}_S)$ via $\epsilon$, and by abuse of notation denote by $\mathfrak{C}_S$ the image of $\mathfrak{C}_S$ under $\iota\circ j_{\epsilon}$. Denote by $\mathfrak{C}_S-\mathfrak{C}_S = \mathscr{D}_1^{\cA}(\mathfrak{C}_S)$.

\begin{thm}\label{ThmUnivCurveBettiRankIntro}
Let $S$, $\cA$ and $\mathfrak{C}_S$ as above.
For any $m \ge 1$, we have
\begin{enumerate}
\item[(i)] $\mathrm{rank}_{\R}(\mathrm{d}b^{[m]}_{\Delta}|_{\mathfrak{C}_S^{[m]}}) = 2 \dim (\mathfrak{C}_S^{[m]})$ for all $m \ge \dim S$ if $g \ge 2$.
\item[(ii)] $\mathrm{rank}_{\R}(\mathrm{d}b^{[m]}_{\Delta}|_{(\mathfrak{C}_S - \mathfrak{C}_S)^{[m]}}) = 2 \dim (\mathfrak{C}_S - \mathfrak{C}_S)^{[m]}$ for all $m \ge \dim S$ if  $g \ge 3$.
\end{enumerate}
\end{thm}
Theorem~\ref{ThmUnivCurveBettiRankIntro} is a particular case of part (i) of the following Theorem~\ref{ThmFaltingsZhangBettiMapIntro} (applied to $\iota = \mathrm{id}$, $X = \mathfrak{C}_S$ and $X = \mathfrak{C}_S - \mathfrak{C}_S$), which is the general result towards \eqref{EqBettiRankMaxDIneqIntro}.

\begin{thm}\label{ThmFaltingsZhangBettiMapIntro}
Assume that $\iota|_X$ is generically finite. Assume furthermore that $X$ satisfies:
\begin{enumerate}
\item[(a)] We have $\dim X > \dim S$.
\item[(b)] 
For each $s \in S(\C)$, $X_s$ generates $\cA_s$.
\item[(c)] 
We have $X + \cA' \not\subseteq X$ for any non-isotrivial abelian subscheme $\cA'$ of $\cA \rightarrow S$.
\end{enumerate}
Then we have
\begin{enumerate}
\item[(i)] $\mathrm{rank}_{\R}(\mathrm{d}b^{[m]}_{\Delta}|_{X^{[m]}}) = 2 \dim X^{[m]}$ for all $m \ge \dim S$.
\item[(ii)] 
$\mathrm{rank}_{\R}(\mathrm{d}b^{[m]}_{\Delta}|_{\mathscr{D}_m^{\cA}(X^{[m+1]})}) = 2 \dim \mathscr{D}_m^{\cA}(X^{[m+1]})$  for all $m \ge \dim X$ if $\iota$ is quasi-finite.
\end{enumerate}
\end{thm}
This theorem follows directly from Theorem~\ref{ThmFaltingsZhang}, applied to $t=0$. In practice, the bound for $m$ can often be improved; see Remark~\ref{RmkDegSimpleBase} for some cases for Theorem~\ref{ThmUnivCurveBettiRankIntro} and Theorem~\ref{ThmUnivCurveBettiRankIntroBis}. Hypothesis (a) is crucial: if $X$ is the image of a multi-section of $\cA \rightarrow S$, then $X^{[m]}$ is contained in the diagonal of $\cA \rightarrow \cA^{[m]}$, so essentially no new objects are constructed with the operations.


We close this subsection with the following result, which is a direct corollary of part (ii) of Theorem~\ref{ThmFaltingsZhangBettiMapIntro} applied to $\iota = \mathrm{id}$ and $X = \mathfrak{C}_S$.
\begin{thmbis}{ThmUnivCurveBettiRankIntro}\label{ThmUnivCurveBettiRankIntroBis}
Under the notation of Theorem~\ref{ThmUnivCurveBettiRankIntro}. Let $\mathscr{D}_m := \mathscr{D}_m^{\mathfrak{A}_g}$, namely
\[
\mathscr{D}_m \colon \underbrace{\mathfrak{A}_{g} \times_{\mathbb{A}_{g}} \mathfrak{A}_{g} \times_{\mathbb{A}_{g}} \ldots \times_{\mathbb{A}_{g}} \mathfrak{A}_{g}}_{(m+1)\text{-copies}} \rightarrow \underbrace{\mathfrak{A}_{g} \times_{\mathbb{A}_{g}} \ldots \times_{\mathbb{A}_{g}} \mathfrak{A}_{g}}_{m\text{-coplies}}
\]
fiberwise defined by $(P_0, P_1, \ldots, P_m) \mapsto (P_1 - P_0, \ldots, P_m - P_0)$. Assume $g \ge 2$. Then we have $\mathrm{rank}_{\R}(\mathrm{d}b^{[m]}_{\Delta}|_{\mathscr{D}_m(\mathfrak{C}_S^{[m+1]})}) = 2 \dim \mathscr{D}_m(\mathfrak{C}_S^{[m+1]})$ for all $m \ge \dim (\mathfrak{C}_S) = 1 + \dim S$.
\end{thmbis}

Theorem~\ref{ThmUnivCurveBettiRankIntroBis} will be applied in \cite{DGHMazur}, as an important step, to prove: For a smooth projective curve $C$ of genus $g \ge 2$ defined over a number field $K$, $\#C(K)$ is bounded only in terms of $g$, $[K:\Q]$ and the Mordell-Weil rank.

Our second application is to answer a question of Andr\'{e}--Corvaja--Zannier \cite{ACZBetti}.

\begin{namedques}[ACZ]\label{ConjACZ}
Assume that $\cA/S$ has no fixed part over any finite covering of $S$ and that $\Z X = \cup_{N \in \Z}\{[N]x : x\in X(\C)\}$ is Zariski dense in $\cA$. Does \eqref{EqBettiRankMaxIneqIntro} hold for $l = \min(\dim \iota(X), g)$?
\end{namedques}

Many cases of this question were proved to be true when $\dim \iota_S(S) \ge g$ in \cite{ACZBetti}, \textit{e.g.} when $g\le 3$ or any $\cA/S$ with $\mathrm{End}(\cA/S) = \Z$.

We hereby answer the ACZ~Question: it has positive answer in many cases, but may be false in general.
\begin{thm}\label{ThmACZIntro}
We have:
\begin{enumerate}
\item[(i)] The ACZ~Question has a positive answer if:
\begin{enumerate}
\item[(a)] Either $\cA \rightarrow S$ is geometrically simple;
\item[(b)] Or each Hodge generic curve $C \subseteq \iota_S(S)$ satisfies the following property: $\iota(\cA)|_C := \pi^{-1}(C) \rightarrow C$ has no fixed part over any finite covering of $C$.
\end{enumerate}
\item[(ii)] There exist a closed irreducible subvariety $S \subseteq \A_4$ of dimension $4$ and a section $\xi$ of $\mathfrak{A}_4|_S \rightarrow S$ such that $\mathfrak{A}_4|_S \rightarrow S$ has no fixed part over any finite covering of $S$, $\Z\xi$ is Zariski dense in $\mathfrak{A}_4|_S$, and $\mathrm{rank}_{\R}(\mathrm{d}b_{\Delta}|_{\xi(S)\cap \cA_{\Delta}})_x < 8$ for all $x \in \xi(S)$.
\end{enumerate}
\end{thm}
Part (i) is Theorem~\ref{ThmACZ}; see Remark~\ref{RmkACZ} for (1)(b). Part (1)(a) for $l = g$, combined with \cite[Proposition~2.1.1]{ACZBetti}, shows that \cite[Theorem~0.3]{VoisinBetti} holds without the dimension assumption because by Lemma~4.5 of \textit{loc.cit} the abelian scheme in question is geometrically simple. Part (ii) is constructed in Example~\ref{EgCounterexampleACZ}; it is closely related to \cite[Remark~6.2.1]{ACZBetti}. Note that this counterexample is the simplest one: In this example $\cA/S$ has maximal variation and $X$ is the image of a section, and by \cite[Theorem~2.3.1]{ACZBetti} no such examples exist for $g \le 3$.

\subsection{The $t$-th degeneracy locus}
Our method to study the generic rank of the Betti map is to translate the problem into studying the \textit{$t$-th degeneracy locus} defined below. Let us explain it in this subsection. Recall our abelian scheme $\pi_S \colon \cA \rightarrow S$.

\begin{defn}\label{DefnSubvarOfsgType}
A closed irreducible subvariety $Z$ of $\cA$ is called a \textbf{generically special subvariety of sg type} of $\cA$ if there exists a finite covering $S' \rightarrow S$, inducing a morphism $\rho \colon \cA' = \cA \times_S S' \rightarrow \cA$, such that $Z = \rho(\sigma' + \sigma_0' + \cB')$, where $\cB'$ is an abelian subscheme of $\cA'/S'$, $\sigma'$ is a torsion section of $\cA'/S'$, and $\sigma_0'$ is a constant section of $\cA'/S'$.
\end{defn}
We briefly explain the meaning of constant section here. Let $C' \times S'$ be the largest constant abelian subscheme of $\cA'/S'$. We say that a section $\sigma_0' \colon S' \rightarrow \cA'$ is a constant section if there exists $c' \in C'(\mathbb{C})$ such that $\sigma_0$ is the composite of $S' \rightarrow C' \times S'$, $s' \mapsto (c',s')$, and the inclusion $C' \times S' \subseteq \cA'$.

Definition~\ref{DefnSubvarOfsgType} is closely related to the \textit{generically special subvarieties} defined in \cite[Definition~1.2]{GHGBC}. See Appendix~\ref{SectionAppendixParticularBase} for some discussion.

For any locally closed irreducible subvariety $Y$ of $\cA$, denote by $\langle Y \rangle_{\mathrm{sg}}$ the smallest generically special subvariety of sg type of $\cA|_{\pi_S(Y)} = \pi_S^{-1}(\pi_S(Y))$ which contains $Y$.

\begin{defn}\label{DefnDegeneracyLocus}
Let $X$ be a closed irreducible subvariety of $\cA$. For any $t \in \mathbb{Z}$, define the \textbf{$t$-th degeneracy locus} of $X$, denoted by $X^{\mathrm{deg}}(t)$, to be the union of positive dimensional closed irreducible subvarieties $Y \subseteq X$ such that 
$\dim \langle Y \rangle_{\mathrm{sg}} - \dim \pi_S(Y) < \dim Y + t$. When $t = 0$, we abbreviate $X^{\mathrm{deg}}(0)$ as $X^{\mathrm{deg}}$. We say that $X$ is \textbf{degenerate} if $X^{\mathrm{deg}}$ is Zariski dense in $X$.
\end{defn}

Note that $X^{\mathrm{deg}} = X$ clearly holds if $X$ is a multi-section and $g < \dim S$.

The locus on which $\mathrm{rank}_{\R} (\mathrm{d}b_{\Delta}|_X)_x$ is smaller than expected is precisely $X^{\mathrm{deg}}(t)$ for some $t \le 0$. More precisely we have: (recall the modular map $\iota \colon \cA \rightarrow \mathfrak{A}_g$ \eqref{EqModularMapIntro} and the naive bound $\mathrm{rank}_{\R} (\mathrm{d}b_{\Delta}|_X) \le 2 \dim \iota(X)$)
\begin{thm}\label{ThmDegLocusIntro}
For each integer $l \le \dim \iota(X)$, we have
\[
\mathrm{rank}_{\R} (\mathrm{d}b_{\Delta}|_X) < 2l \Leftrightarrow  X^{\mathrm{deg}}(l-\dim X)\text{ is Zariski dense in }X.
\]
\end{thm}

This is not yet satisfactory as the $X^{\mathrm{deg}}(t)$ thus defined is \textit{a priori} a complicated subset of $X$. However we show that they are all Zariski closed in $X$.
\begin{thm}\label{ThmZarClosedXdegIntro}
The set $X^{\mathrm{deg}}(t)$ is Zariski closed in $X$ for each $t \in \Z$.
\end{thm}
Both theorems will be proved in $\mathsection$\ref{SectionAppConjACZ}. Before the treatment of the general case, we will prove, for the case $X \subseteq \mathfrak{A}_g$, Theorem~\ref{ThmDegLocusIntro} in $\mathsection$\ref{SectionSubvarBettiNonMaxRank} and Theorem~\ref{ThmZarClosedXdegIntro}  in $\mathsection$\ref{SectionDegeneracyLocusClosed}.


\subsection{Relation with the relative Manin-Mumford conjecture}
Another application of the $t$-th degeneracy locus is for the relative Manin-Mumford conjecture. In this application we need to take $t = 1$. Let us state the result.

Denote by $\cA_{\mathrm{tor}}$ the set of points $x \in \cA(\C)$ such that $[N]x$ lies in the zero section of $\cA \rightarrow S$ for some integer $N$. Zannier \cite{ZannierUnlikelyIntersectons} proposed the following relative Manin-Mumford conjecture.
\begin{namedconj}[Relative Manin-Mumford]\label{ConjRelativeMMIntro}
Assume that $\Z X:= \bigcup_{N \in \Z}\{[N]x: x \in X(\C)\}$ is Zariski dense in $\cA$. If $(X\cap \cA_{\mathrm{tor}})^{\Zar} = X$, then $\codim_{\cA}(X) \le \dim S$.
\end{namedconj}

In this paper we will reduce this conjecture to another simpler conjecture.
\begin{conj}\label{ConjRelativeMMEasyIntro}
Assume $S$ is a locally closed irreducible subvariety of $\A_g$ defined over $\bar{\Q}$ and $\cA = \mathfrak{A}_g \times_{\mathbb{A}_g} S$. Assume $X$ is defined over $\bar{\Q}$. If $(X\cap \cA_{\mathrm{tor}})^{\Zar} = X$, then $X^{\mathrm{deg}}(1) = X$.
\end{conj}

\begin{prop}\label{RelativeMMEquivIntro}
Conjecture~\ref{ConjRelativeMMEasyIntro} implies the relative Manin-Mumford conjecture.
\end{prop}
A more precise version of this reduction is Proposition~\ref{RelativeMMEquiv}. 

Proposition~\ref{RelativeMMEquivIntro} suggests that there is a strong link between the Betti map and the relative Manin-Mumford conjecture. The existence of such a link already appeared in previous works on relative Manin-Mumford: the \textit{Betti coordinate} played a key role in the proofs of many particular cases of the conjecture by Masser-Zannier \cite{MasserZannierTorsionPointOnSqEC, MASSER2014116, MasserZannierRelMMSimpleSur} and Corvaja-Masser-Zannier \cite{CorvajaMasserZannier2018} (pencils of abelian surfaces, first over $\bar{\Q}$ then over $\C$; passing from $\overline{\Q}$ to $\C$ is highly non-trivial as it enlarges the base), Bertrand-Masser-Pillay-Zannier \cite{BMPZRelativeMM} (semi-abelian surfaces), and Masser-Zannier \cite{MasserZannierRMMoverCurve} (any abelian scheme over a curve). See also \cite{ZannierRMM}.

\subsection{Outline of the paper}
In $\mathsection$\ref{SectionNotation} we set up some convention of the paper. In $\mathsection$\ref{SectionUnivAb} we recall the universal abelian variety and define the Betti map for this case. In $\mathsection$\ref{SectionBettiRevisited} we define the Betti map for a general abelian scheme. These are the basic setting up of the paper.

In $\mathsection$\ref{SectionBiAlgSystem} we explain in details our main tools to study the Betti map. There are two parts. The first part $\mathsection$\ref{SubsectionGenSpSGandBiAlg}-\ref{SubsectionWAS} is to introduce the functional transcendence theorem (\textit{mixed Ax-Schanuel}), and the second part $\mathsection$\ref{SubsectionMSV}-\ref{SubsectionGeomQuotByNormalSubgp} is Deligne-Pink's language of mixed Shimura varieties.

$\mathsection$\ref{SectionSubvarBettiNonMaxRank}--\ref{SectionCriterionDegeneracy} are the core of this paper. In these sections we prove the main results on the Betti rank (Theorem~\ref{ThmDegLocusIntro}, Theorem~\ref{ThmZarClosedXdegIntro} and Theorem~\ref{ThmCriterionDegIntro}) for the case $X \subseteq \mathfrak{A}_g$. In $\mathsection$\ref{SectionSubvarBettiNonMaxRank} we use \textit{weak mixed Ax-Schanuel} to transfer the problem of the generic rank of the Betti map into studying the $t$-th degeneracy locus for some particular $t$'s, and hence prove Theorem~\ref{ThmDegLocusIntro} for $X \subseteq \mathfrak{A}_g$. In $\mathsection$\ref{SectionDegeneracyLocusClosed} we use the \textit{finiteness result \cite[Theorem~1.4]{GaoAxSchanuel}} to  prove the Zariski closedness of the $t$-th degeneracy locus (Theorem~\ref{ThmZarClosedXdegIntro}) for $X \subseteq \mathfrak{A}_g$. The proof in this section will also be used in $\mathsection$\ref{SectionCriterionDegeneracy}, where the criterion to check whether $X$ is degenerate is proved for $X \subseteq \mathfrak{A}_g$. Combined with results in $\mathsection$\ref{SectionSubvarBettiNonMaxRank}-\ref{SectionDegeneracyLocusClosed} this proves Theorem~\ref{ThmCriterionDegIntro} for $X \subseteq \mathfrak{A}_g$.

Then ultimate versions of the main results on the Betti rank (Theorem~\ref{ThmDegLocusIntro}, Theorem~\ref{ThmZarClosedXdegIntro} and Theorem~\ref{ThmCriterionDegIntro}) are proved in $\mathsection$\ref{SectionAppConjACZ}. As we shall see they are not hard to be deduced from the case $X \subseteq \mathfrak{A}_g$. The end of this section sees its application to the ACZ question and proves Theorem~\ref{ThmACZIntro}.

Then in $\mathsection$\ref{SectionAppFaltingsZhang} we prove Theorem~\ref{ThmFaltingsZhangBettiMapIntro}, claiming that the Betti map attains the maximal rank if we raise $X$ to a large enough fibered power, for $X$ satisfying some mild properties.

In $\mathsection$\ref{SectionRelativeMM} we reduce the relative Manin-Mumford conjecture to a simpler conjecture involving the $1$-st degeneracy locus.

In Appendix~\ref{SectionAppendixParticularBase} we further simplify the formula of the Betti rank when the base takes some simple form.

\subsection*{Acknowledgements}
The author would like to thank Umberto Zannier for relavant discussions, especially on the historical notes on the Betti map. The author would like to thank Fabrizio Berroero, Philipp Habegger, and Umberto Zannier on relavant discussions on relative Manin-Mumford. 
The author would like to thank Yves Andr\'{e}, Daniel Bertrand, Bruno Klingler, Emmanuel Ullmo, Xinyi Yuan and Shouwu Zhang for their comments and suggestions to improve the manuscript. 
The author would like to thank Tangli Ge for pointing out a gap in $\mathsection$\ref{SectionAppFaltingsZhang} in a previous version. 
The author would also like to thank the referees for their comments and suggestions. The author would like to thank the Institute for Advanced Studies (NJ, USA) and the Morningside Center of Mathematics (Beijing, China) for their hospitality during the preparation of this work.

\section{Convention and Notation}\label{SectionNotation}
\subsection{}
Let $g \ge 1$ be an integer. 
Let $S$ be an irreducible variety over $\C$ and let $\pi_S \colon \cA \rightarrow S$ be an abelian scheme of relative dimension $g$. Since level structures are not important in this paper, we will use the following abuse of notation through the whole paper.
\begin{enumerate}
\item[(i)] We say that a subvariety $\cB$ of $\cA$ is an \textbf{abelian subscheme} of $\cA \rightarrow S$ if there exists a finite covering $S' \rightarrow S$, inducing a morphism $\rho \colon \cA' = \cA \times_S S' \rightarrow \cA$, such that $\cB = \rho(\cB')$ where $\cB'$ is an abelian subscheme of $\cA'/S'$ in the usual sense.\footnote{Namely $\cB'$ is an irreducible subgroup scheme of $\cA' \rightarrow S'$ which is proper, flat and dominant to $S'$. In particular each fiber of $\cB' \rightarrow S'$ is an abelian subvariety of the corresponding fiber of $\cA' \rightarrow S'$.}
\item[(ii)] We say that $\sigma$ is a \textbf{section} of $\cA \rightarrow S$ if there exists a finite covering $S' \rightarrow S$, inducing a morphism $\rho \colon \cA' = \cA \times_S S' \rightarrow \cA$, such that $\sigma = \rho \circ \sigma'$ where $\sigma' \colon S' \rightarrow \cA'$ is a section of $\cA' \rightarrow S'$ in the usual sense.\footnote{In other words $\sigma$ is a multi-section in the usual sense.} Denote by $\sigma(S) := (\rho\circ\sigma')(S')$.
\item[(iii)] In (ii), call $\sigma$ a \textbf{torsion section} if $\sigma'(s')$ is a torsion point on $\cA'_{s'}$ for each $s' \in S'(\C)$; call $\sigma$ a \textbf{constant section} if $\sigma'$ is the composite of $S' \rightarrow C' \times S'$, $s' \mapsto (c',s')$, and the inclusion $C' \times S' \subseteq \cA'$, where $C' \times S'$ is a constant abelian subscheme of $\cA' \rightarrow S'$.
\end{enumerate}

The following definition is convenient to study constant sections.
\begin{defn}
An abelian scheme $\cC \rightarrow S$ (of relative dimension $g$) is said to be \textbf{isotrivial} if one of the following equivalent conditions holds:
\begin{enumerate}
\item[(i)] The fibers $\cC_s$ are isomorphic to each other for all $s \in S(\C)$, 
\item[(ii)] There exists a finite covering $S' \rightarrow S$ such that $\cC \times_S S'$ is a constant abelian scheme, namely $\cC \times_S S' = C \times S'$ for some abelian variety $C$ over $\C$.
\item[(iii)] The image of the modular map $S \rightarrow \A_g$ induced by $\cC \rightarrow S$ is a point.
\end{enumerate}
\end{defn}
In particular the zero section of any abelian scheme $\cA \rightarrow S$ is an isotrivial abelian subscheme. If $\dim S = 0$, then $\cA \rightarrow S$ is isotrivial.

Since the sum of two isotrivial abelian subschemes of $\cA \rightarrow S$ is isotrivial, we can define the largest isotrivial abelian subscheme of $\cA \rightarrow S$ which we denote by $\cC$. Then any constant section of $\cA \rightarrow S$ (defined above) has image in $\cC$.

\subsection{Modular map}
When $g = 0$, let $\mathfrak{A}_0$ and $\A_0$ be a point. When $g \ge 1$. Let $D = \mathrm{diag}(d_1,\ldots,d_g)$ be a $g \times g$ diagonal matrix with $d_1, \ldots, d_g$ positive integers such that $d_1|\cdots |d_g$.

For any integer $N \ge 3$, let $\mathbb{A}_{g,D}(N)$ denote the moduli space of abelian varieties of dimension $g$ with polarization of type $D$ and with level-$N$-structure. Then $\mathbb{A}_{g,D}(N)$ is a fine moduli space, and hence admits a universal family $\pi \colon \mathfrak{A}_{g,D}(N) \rightarrow \mathbb{A}_{g,D}(N)$. Since level structure is not important for our purpose, we shall drop the ``$(N)$'' in the rest of the paper. When there is no ambiguity about the polarization type, we also drop the ``$D$'' and simply use the notation $\pi \colon \mathfrak{A}_g \rightarrow \mathbb{A}_g$.

It is known that any abelian scheme $\pi_S \colon \cA \rightarrow S$ can be equipped with a polarization of type $D$ for some $D = \mathrm{diag}(d_1,\ldots,d_g)$, with $d_1, \ldots, d_g$ positive integers such that $d_1|\cdots |d_g$; see \cite[$\mathsection$2.1]{GenestierNgo}. Thus up to taking a finite covering of $S$ and the associated base change of $\cA \rightarrow S$, we have the modular map
\begin{equation}\label{EqModularMapIntro}
\xymatrix{
\cA \ar[r]^{\iota} \ar[d] \pullbackcorner & \mathfrak{A}_g \ar[d]^{\pi} \\
 S \ar[r]^-{\iota_S} & \mathbb{A}_g.
}
\end{equation}
Then $\cA \rightarrow S$ is isotrivial if and only if $\iota_S(S)$ is a point. In particular if $S \subseteq \A_g$, then $\mathfrak{A}_g \times_{\A_g}S$ is isotrivial if and only if $S$ is a point.


\section{Universal abelian variety and Betti map}\label{SectionUnivAb}

We recall some facts on the universal abelian variety in this section. Let $D = \mathrm{diag}(d_1,\ldots,d_g)$ be a $g \times g$ diagonal matrix with $d_1,\ldots,d_g$ positive integers such that $d_1|\cdots |d_g$.

\subsection{Uniformizing space of $\mathbb{A}_g$}\label{SubsectionUnifSpaceOfModuliSpace}
Let $\mathfrak{H}_g^+$ be the Siegel upper half space
\[
\{ Z = X + \sqrt{-1}Y \in M_{g \times g}(\mathbb{C}) : Z = Z^{\!^{\intercal}},~ Y > 0\}.
\]
It is well-known that the uniformization of $\mathbb{A}_g := \mathbb{A}_{g,D}$ in the category of complex varieties is given by
\begin{equation}\label{EqUniformizationModuliSpace}
\mathbf{u}_G \colon \mathfrak{H}_g^+ \rightarrow \mathbb{A}_g.
\end{equation}
Let us take a closer look at this uniformization.

Let $\mathrm{Sp}_{2g}$ be the $\mathbb{Q}$-group
\[
\left\{ h \in \mathrm{GL}_{2g} : h \begin{pmatrix} 0 & D \\ -D & 0 \end{pmatrix} h^{\!^{\intercal}} = \begin{pmatrix} 0 & D \\ -D & 0 \end{pmatrix} \right\},
\]
and let $\mathrm{GSp}_{2g}$ be the image of $\mathbb{G}_{\mathrm{m}} \times \mathrm{Sp}_{2g}$ under the central isogeny $\mathbb{G}_{\mathrm{m}} \times \mathrm{SL}_{2g} \rightarrow \mathrm{GL}_{2g}$. 
Then $\mathrm{GSp}_{2g}(\mathbb{R})^+$, the connected component of $\mathrm{GSp}_{2g}(\mathbb{R})$ containing the identity, acts on $\mathfrak{H}_g^+$ by the formula
\[
\begin{pmatrix} A' & B' \\ C' & D' \end{pmatrix} Z = (A'Z+B')(C'Z+D')^{-1}, \quad \forall \begin{pmatrix} A' & B' \\ C' & D' \end{pmatrix} \in \mathrm{GSp}_{2g}(\mathbb{R})^+\text{ and }Z \in \mathfrak{H}_g^+.
\]
It is known that the action of $\mathrm{GSp}_{2g}^{\mathrm{der}}(\mathbb{R}) = \mathrm{Sp}_{2g}(\mathbb{R})$ on $\mathfrak{H}_g^+$ thus defined is transitive, and the uniformization \eqref{EqUniformizationModuliSpace} is obtained by identifying $(\mathbb{A}_g)^{\mathrm{an}}$ with the quotient space $\Gamma_{\mathrm{GSp}_{2g}}\backslash \mathfrak{H}_g^+$ for a suitable congruence group $\Gamma_{\mathrm{Sp}_{2g}}$ of $\mathrm{Sp}_{2g}(\mathbb{Z})$.



\subsection{Uniformizing space of $\mathfrak{A}_g$}\label{SubsectionUnivAbVar}
To obtain the uniformization of $\mathfrak{A}_g$, let us construct the following 
 complex space $\cX_{2g,\mathrm{a}}^+$.
\begin{enumerate}
\item[(i)] As a semi-algebraic space, $\cX_{2g,\mathrm{a}}^+ = \mathbb{R}^{2g} \times \mathfrak{H}_g^+$.
\item[(ii)] 
The complex structure of $\cX_{2g,\mathrm{a}}^+$ is the one given by
\begin{equation}\label{EqComplexStrOfX2g}
\begin{array}{cccc}
\cX_{2g,\mathrm{a}}^+ = & \mathbb{R}^g \times \mathbb{R}^g \times \mathfrak{H}_g^+ & \xrightarrow{\sim} & \mathbb{C}^g \times \mathfrak{H}_g^+, \\
& (a,b,Z) & \mapsto & (Da+Zb, Z)
\end{array}.
\end{equation}
\end{enumerate}

The uniformization of $\mathfrak{A}_g$ in the category of complex varieties is then given by
\begin{equation}\label{EqUnifUniversalAbVar}
\mathbf{u} \colon \cX_{2g,\mathrm{a}}^+ \rightarrow \mathfrak{A}_g.
\end{equation}

Similar to the discussion on $\mathbf{u}_G$, there exists a $\mathbb{Q}$-group which we call $P_{2g,\mathrm{a}}$ such that $P_{2g,\mathrm{a}}^{\mathrm{der}}(\mathbb{R})$ acts transitively on $\cX_{2g,\mathrm{a}}^+$ and $\mathbf{u}$ is obtained by identifying $(\mathfrak{A}_g)^{\mathrm{an}}$ with the quotient space $\Gamma\backslash \cX_{2g,\mathrm{a}}^+$ for a suitable congruence subgroup $\Gamma = \mathbb{Z}^{2g} \rtimes \Gamma_{\mathrm{Sp}_{2g}}$ of $P_{2g,\mathrm{a}}^{\mathrm{der}}(\mathbb{Z})$. Let us briefly explain this.

Use $V_{2g}$ to denote the $\mathbb{Q}$-vector group of dimension $2g$. Then the natural action of $\mathrm{GSp}_{2g}$ on $V_{2g}$ defines a $\mathbb{Q}$-group
\[
P_{2g,\mathrm{a}} = V_{2g} \rtimes \mathrm{GSp}_{2g}.
\]
The action of $P_{2g,\mathrm{a}}(\mathbb{R})^+$ on $\cX_{2g,\mathrm{a}}^+$ is defined as follows: for any $(v,h) \in P_{2g,\mathrm{a}}(\mathbb{R})^+ = V_{2g}(\mathbb{R}) \rtimes \mathrm{GSp}_{2g}(\mathbb{R})^+$ and any $(v',x) \in \cX_{2g,\mathrm{a}}^+$, we have
\begin{equation}\label{EqActionOfP2gOnX2g}
(v,h) \cdot (v',x) = (v+hv',hx)
\end{equation}
where $\mathrm{GSp}_{2g}(\mathbb{R})^+$ acts on $\mathbb{R}^{2g}$ as above \eqref{EqComplexStrOfX2g}.

The natural projection of complex spaces $\widetilde{\pi} \colon \cX_{2g,\mathrm{a}}^+ \rightarrow \mathfrak{H}_g^+$ is equivariant with respect to the natural projection of groups $P_{2g,\mathrm{a}} \rightarrow \mathrm{GSp}_{2g}$. Hence by abuse of notation we also denote by $\widetilde{\pi} \colon P_{2g,\mathrm{a}} \rightarrow \mathrm{GSp}_{2g}$.

\subsection{Betti map}\label{SubsectionBettiMapUnivAb}
We define the Betti map in this section. 
We will start by defining the \textit{universal uniformized Betti map} on $\cX_{2g,\mathrm{a}}^+$, and then descend it to $\cA_{\mathfrak{H}_g^+}$, the pullback of $\mathfrak{A}_g / \mathbb{A}_g$ under $\mathbf{u}_G \colon \mathfrak{H}_g^+ \rightarrow \mathbb{A}_g$. Note that $\cA_{\mathfrak{H}_g^+}$ is a family of abelian varieties over $\mathfrak{H}_g^+$. 

Recall that $\cX_{2g,\mathrm{a}}^+$ is defined to be $\mathbb{R}^{2g} \times \mathfrak{H}_g^+$ with the complex structure determined by \eqref{EqComplexStrOfX2g}. 
The \textit{universal uniformized Betti map} $\tilde{b}$ is defined to be the natural projection
\begin{equation}\label{EqUnivUniformizedBettiMap}
\tilde{b} \colon \cX_{2g,\mathrm{a}}^+ \rightarrow \mathbb{R}^{2g}.
\end{equation}
Then $\tilde{b}$ is semi-algebraic. For the complex structure on $\cX_{2g,\mathrm{a}}^+$ given by \eqref{EqComplexStrOfX2g}, it is clear that $\tilde{b}^{-1}(r)$ is complex analytic for each $r \in \mathbb{R}^{2g}$.

Recall that $(\mathfrak{A}_g)^{\mathrm{an}} \cong \Gamma\backslash \cX_{2g,\mathrm{a}}^+$ as complex spaces for a suitable congruence subgroup $\Gamma = \mathbb{Z}^{2g} \rtimes \Gamma_{\mathrm{Sp}_{2g}}$ of $P_{2g,\mathrm{a}}^{\mathrm{der}}(\mathbb{Z})$.
The family of abelian varieties $\cA_{\mathfrak{H}_g^+}$ defined as at the beginning of this subsection can be identified with the quotient space $(\mathbb{Z}^{2g} \rtimes \{1\}) \backslash \cX_{2g,\mathrm{a}}^+$. Now taking quotient by $\mathbb{Z}^{2g}$ on both sides of \eqref{EqUnivUniformizedBettiMap}, we obtain the \textit{universal Betti map}
\begin{equation}\label{EqUnivBettiMap}
b \colon \cA_{\mathfrak{H}_g^+} \rightarrow \mathbb{T}^{2g}
\end{equation}
where $\mathbb{T}^{2g}$ denotes the real torus of dimension $2g$. By the discussion below \eqref{EqUnivUniformizedBettiMap}, we have the following properties for $\tilde{b}$ and $b$.
\begin{enumerate}
\item[(i)] Both $\tilde{b}$ and $b$ are real-analytic, and $\tilde{b}$ is moreover semi-algebraic.
\item[(ii)] For each $r \in \mathbb{R}^{2g}$, resp. each $t \in \mathbb{T}^{2g}$, we have that $\tilde{b}^{-1}(r)$, resp. $b^{-1}(t)$, is complex analytic.
\item[(iii)] For each $\tau \in \mathfrak{H}_g^+$, the restriction $b|_{(\cA_{\mathfrak{H}_g^+})_{\tau}}$ is a group isomorphism.
\end{enumerate}

We summarize our notations regarding the uniformizations in the following  diagram
\begin{equation}\label{DiagramUnivAbVarAndModuliSpace}
\xymatrix{
\cX_{2g,\mathrm{a}}^+ \ar[rd]_{\mathbf{u}} \ar[r] \ar@/^1pc/[rr]|-{\tilde{\pi}} & \cA_{\mathfrak{H}_g^+} \ar[r] \ar[d] \pullbackcorner & \mathfrak{H}_g^+ \ar[d]^{\mathbf{u}_G} \\
& \mathfrak{A}_g \ar[r]^{\pi} & \mathbb{A}_g
}
\end{equation}
with the uniformization $\mathbf{u}$ from \eqref{EqUnifUniversalAbVar} and the uniformization $\mathbf{u}_G$ from \eqref{EqUniformizationModuliSpace}.

\section{Betti map on arbitrary abelian schemes}\label{SectionBettiRevisited}
The goal of this section is to extend the definition of Betti map to an arbitrary abelian scheme. Moreover we choose to work on the original abelian scheme instead of on the pullback to the universal covering.


Let $S$ be an irreducible quasi-projective variety over $\C$, and let $\pi_S \colon \cA \rightarrow S$ be an abelian scheme of relative dimension $g$. Then up to replacing $S$ by a finite covering and $\cA \rightarrow S$ by the corresponding base change, we have 
\begin{equation}\label{EqEmbedAbSchIntoUnivAbVar}
\xymatrix{
\cA \ar[r]^-{\iota} \ar[d]_{\pi_S} \pullbackcorner & \mathfrak{A}_g \ar[d]^{\pi} \\
S \ar[r]^-{\iota_S} & \mathbb{A}_g.
}
\end{equation}

Let $\Delta_0$ be a simply-connected open subset in $\mathbb{A}_g^{\mathrm{an}}$. Fix a component of $\widetilde{\Delta}_0$ of $\mathbf{u}_G^{-1}(\Delta_0)$ under the uniformization $\mathbf{u}_G \colon \mathfrak{H}_g^+ \rightarrow \mathbb{A}_g$. The fact that $\Delta_0$ is simply-connected implies that $\mathbf{u}_G|_{\widetilde{\Delta}_0}$ is an isomorphism in the category of complex spaces. Thus the universal Betti map \eqref{EqUnivBettiMap} induces a map $b_{\Delta_0} \colon \mathfrak{A}_g|_{\Delta_0} \rightarrow \mathbb{T}^{2g}$ by identifying $\mathfrak{A}_g|_{\Delta_0} = \pi^{-1}(\Delta_0)$ with $\cA_{\mathfrak{H}_g^+}|_{\widetilde{\Delta}_0}$.

For any $s \in S(\C)$, we can find a $\Delta_0$ as above such that $\iota_S(s) \in \Delta_0$. Let $\Delta$ be a component of $\iota_S^{-1}(\Delta_0)$ which contains $s$. Let $\cA_{\Delta} = \pi_S^{-1}(\Delta)$. Define
\begin{equation}\label{EqBettiMapActualScheme}
b_{\Delta} \colon \cA_{\Delta} \rightarrow \mathbb{T}^{2g}
\end{equation}
to be the composite of $\iota$ and $b_{\Delta_0}$. The following properties of $b_{\Delta}$ follows from the properties of the universal Betti map listed below \eqref{EqUnivBettiMap}.
\begin{enumerate}
\item[(i)] The map $b_{\Delta}$ is real-analytic.
\item[(ii)] For each $t \in \mathbb{T}^{2g}$, we have that $b_{\Delta}^{-1}(t)$, is complex analytic.
\item[(iii)] For each $s \in \Delta$, the restriction $b_{\Delta}|_{\cA_s}$ is a group isomorphism.
\end{enumerate}

Note that $b_{\Delta}$ is not unique as we can choose different components of $\bu_G^{-1}(\Delta_0)$. But $b_{\Delta}$ is unique up to $\Sp_{2g}(\Z)$ because $\A_g \cong \Gamma_{\GSp_{2g}} \backslash \mathfrak{H}_g^+$ for some congruence subgroup $\Gamma_{\GSp_{2g}}$ of $\Sp_{2g}(\Z)$.

The Betti map factors through the universal abelian variety by definition. Thus to study the generic rank of the Betti map, it often suffices to consider the subvarieties of $\mathfrak{A}_g$.

Before moving on, let us see another way to define the Betti map. Let $\bu_S \colon \tilde{S} \rightarrow S^{\mathrm{an}}$ be the universal covering, and let $\cA_{\tilde{S}}$ be the pullback of $\cA \rightarrow S$ under $\bu_S$. Then the modular map $\iota \colon \cA \rightarrow \mathfrak{A}_g$ induces a natural morphism $\tilde{\iota} \colon \cA_{\tilde{S}} \rightarrow \cA_{\mathfrak{H}_g^+}$; see \eqref{EqUnivBettiMap} for notation. Then one can define $b_{\tilde{S}} \colon \cA_{\tilde{S}} \rightarrow \mathbb{T}^{2g}$ to be $\tilde{\iota}$ composed with the universal Betti map \eqref{EqUnivBettiMap}. Note that $b_{\tilde{S}}$ is uniquely determined, contrary to $b_{\Delta}$. Now $b_{\Delta}$ \eqref{EqBettiMapActualScheme} can be obtained as follows: Identify $\cA_{\Delta}$ and $\cA_{\tilde{\Delta}}$ by identifying $\Delta$ with a component $\tilde{\Delta}$ of $\bu_S^{-1}(\Delta)$, then $b_{\Delta}$ is $b_{\tilde{S}}$ restricted to $\cA_{\Delta}$.

Now we are able to prove some easy properties of the Betti rank. Let $X$ be a closed irreducible subvariety of $\cA$ with $\pi_S(X) = S$. Then $\mathrm{rank}_{\R}(\mathrm{d}b_{\Delta}|_X) := \max_{x \in X^{\mathrm{sm}}(\C)\cap \cA_{\Delta}}(\mathrm{rank}_{\R}(\mathrm{d}b_{\Delta}|_X)_x)$ satisfies:
\begin{enumerate}
\item It is at most $2\min(\dim \iota(X), g)$ as $b_{\Delta}$ factors through $\iota$.
\item It is even by property (ii) above.
\item It does not depend on the choice of $\Delta$: Take a complex analytic irreducible component $\hat{X}$ of the inverse image of $X^{\mathrm{sm}}$ under $\cA_{\tilde{S}} \rightarrow \cA$, then $\mathrm{rank}_{\R}(\mathrm{d}b_{\Delta}|_X) = \mathrm{rank}_{\R}(\mathrm{d}b_{\tilde{S}}|_{\hat{X} \cap \cA_{\tilde{\Delta}}})$. As $\hat{X} \cap \cA_{\tilde{\Delta}}$ is open (and hence has positive Lebesgue measure) in $\hat{X}$, we have $\mathrm{rank}_{\R}(\mathrm{d}b_{\tilde{S}}|_{\hat{X} \cap \cA_{\tilde{\Delta}}}) = \mathrm{rank}_{\R}(\mathrm{d}b_{\tilde{S}}|_{\hat{X}})$ by Sard's theorem. The conclusion then follows.
\end{enumerate}

\section{Bi-algebraic system associated with $\mathfrak{A}_g$}\label{SectionBiAlgSystem}
The goal of this section is to give some further background knowledge on the universal abelian varieties, which will serve as our main tools to study the Betti map. There are two parts. The first part $\mathsection$\ref{SubsectionGenSpSGandBiAlg}-\ref{SubsectionWAS} is to introduce the functional transcendence theorem (called \textit{weak Ax-Schanuel}), and the second part $\mathsection$\ref{SubsectionMSV}-\ref{SubsectionGeomQuotByNormalSubgp} is Deligne-Pink's language of mixed Shimura varieties.\footnote{For readers not familiar with the language of Shimura varieties but only want to see how to study the generic rank of the Betti map or the relative Manin-Mumford conjecture via $X^{\mathrm{deg}}(t)$, it is probably a better idea to skip $\mathsection$\ref{SubsectionMSV}-\ref{SubsectionGeomQuotByNormalSubgp} as these two complicated subsections  will only be used in $\mathsection$\ref{SectionDegeneracyLocusClosed} and $\mathsection$\ref{SectionCriterionDegeneracy} (whose proofs we also suggest to skip at first).}

\subsection{Generically special subvarieties of sg type and bi-algebraic subvarieties}\label{SubsectionGenSpSGandBiAlg}
The goal of this subsection is to explain the relation between generically special subvarieties of sg type (see Definition~\ref{DefnSubvarOfsgType}) and bi-algebraic subvarieties of $\mathfrak{A}_g$.

Let us start with defining bi-algebraic subvarieties of $\mathfrak{A}_g$. Recall the uniformization $\bu \colon \cX_{2g,\mathrm{a}}^+ \rightarrow \mathfrak{A}_g$. By \cite[$\mathsection$4]{GaoTowards-the-And} $\cX_{2g,\mathrm{a}}^+$ can be embedded as an open, in the usual topology, semi-algebraic subset of a complex flag variety (hence algebraic) $\cX_{2g,\mathrm{a}}^\vee$.

\begin{defn}
\begin{enumerate}
\item[(i)] A subset $\widetilde{Y}$ of $\cX_{2g,\mathrm{a}}^+$ is said to be \textbf{irreducible algebraic} if it is a complex analytic irreducible component of $\cX_{2g,\mathrm{a}}^+ \cap W$ for some algebraic subvariety $W$ of $\cX_{2g,\mathrm{a}}^\vee$.
\item[(ii)] An irreducible subvariety $Y$ of $\mathfrak{A}_g$ is said to be \textbf{bi-algebraic} if one (and hence any) complex analytic irreducible component $\widetilde{Y}$ of $\mathbf{u}^{-1}(Y)$ is algebraic.
\end{enumerate}
\end{defn}

It is not hard to show that the intersection of two bi-algebraic subvarieties of $\mathfrak{A}_g$ is a finite union of irreducible bi-algebraic subvarieties of $\mathfrak{A}_g$. Hence for any subset $Z$ of $\mathfrak{A}_g$, there exists a smallest bi-algebraic subvariety $\mathfrak{A}_g$ which contains $Z$. We use $Z^{\mathrm{biZar}}$ to denote it. Note that $Z^{\mathrm{biZar}} \supseteq Z^{\mathrm{Zar}}$.

\begin{rmk}\label{RemarkBiAlgSystem}
There is a canonical way to endow $\mathfrak{H}_g^+$ with an algebraic structure which is compatible with $\widetilde{\pi} \colon \cX_{2g,\mathrm{a}}^+ \rightarrow \mathfrak{H}_g^+$ and the algebraic structure on $\cX_{2g,\mathrm{a}}^+$ defined above; see \cite[$\mathsection$4]{GaoTowards-the-And}. Then it is clear that for any $F$ bi-algebraic in $\mathfrak{A}_g$, we have that $\pi(F)$ is bi-algebraic in $\mathbb{A}_g$.
\end{rmk}

Bi-algebraic subvarieties of $\mathfrak{A}_g$ are closely related to generically special subvarieties of sg type defined in Definition~\ref{DefnSubvarOfsgType} by the following proposition.
\begin{prop}[\!\!{\cite[Proposition~3.3]{GaoA-special-point}}]\label{PropBiAlgAg}
Let $B$ be an irreducible subvariety of $\mathbb{A}_g$. Denote by $\mathfrak{A}_g|_B = \pi^{-1}(B)$. Then we have
\begin{align*}
& \{\text{generically special subvarieties of sg type of } \mathfrak{A}_g|_B\} \\
= & \{ \text{irreducible components of }(\mathfrak{A}_g|_B) \cap F : F \text{ irreducible bi-algebraic in } \mathfrak{A}_g \text{ with } B \subseteq \pi(F)\}.
\end{align*}
\end{prop}

The following equivalent form of Proposition~\ref{PropBiAlgAg} is more practical for our use.
\begin{cor}\label{CorsgBiZar}
Let $Y$ be an irreducible subvariety of $\mathfrak{A}_g$. Then $\langle Y \rangle_{\mathrm{sg}}$ is an irreducible component of $(\mathfrak{A}_g|_{\pi(Y)}) \cap Y^{\mathrm{biZar}}$. In particular $\dim \langle Y \rangle_{\mathrm{sg}} - \dim \pi(Y) = \dim Y^{\mathrm{biZar}} - \dim \pi(Y)^{\mathrm{biZar}}$.
\end{cor}
\begin{proof}
First note that $(\mathfrak{A}_g|_{\pi(Y)}) \cap Y^{\mathrm{biZar}}$ is equidimensional. So the ``In particular'' part follows from the main part and the fact that $\pi(Y)^{\mathrm{biZar}} = \pi(Y^{\mathrm{biZar}})$ (Remark~\ref{RemarkBiAlgSystem}).

Denote by $B = \pi(Y)$ for simplicity. Proposition~\ref{PropBiAlgAg} implies that each irreducible component of $(\mathfrak{A}_g|_B) \cap Y^{\mathrm{biZar}}$ is a generically special subvariety of sg type of $\mathfrak{A}_g|_B$. Hence $\langle Y \rangle_{\mathrm{sg}} \subseteq W$ where $W$ is an irreducible component of $(\mathfrak{A}_g|_B) \cap Y^{\mathrm{biZar}}$.

For the other inclusion, since $\langle Y \rangle_{\mathrm{sg}}$ is a generically special subvariety of sg type of $\mathfrak{A}_g|_B$, we have by Proposition~\ref{PropBiAlgAg} that
$\langle Y \rangle_{\mathrm{sg}}$ is an irreducible component of $(\mathfrak{A}_g|_B) \cap F$
 for some irreducible bi-algebraic subvariety $F$ of $\mathfrak{A}_g$. Then $Y \subseteq \langle Y \rangle_{\mathrm{sg}} \subseteq F$. Hence $Y^{\biZar} \subseteq F$. Thus $ (\mathfrak{A}_g|_B) \cap Y^{\biZar} \subseteq (\mathfrak{A}_g|_B) \cap F$.
 
In summary, we have $\langle Y \rangle_{\mathrm{sg}} \subseteq W \subseteq (\mathfrak{A}_g|_B) \cap Y^{\biZar}  \subseteq (\mathfrak{A}_g|_B) \cap F$, and that $\langle Y \rangle_{\mathrm{sg}}$ is an irreducible component of $(\mathfrak{A}_g|_B) \cap F$. So $W = \langle Y \rangle_{\mathrm{sg}}$ and we are done.
\end{proof}

\subsection{Weak Ax-Schanuel for $\mathfrak{A}_g$}\label{SubsectionWAS}
One of the most important tools we use to study the Betti map is the following weak Ax-Schanuel theorem for $\mathfrak{A}_g$ \cite[Theorem~1.1 or Theorem~3.5]{GaoAxSchanuel}.

\begin{thm}\label{ThmWAS}
Let $\widetilde{Z}$ be a complex analytic irreducible subset of $\cX_{2g,\mathrm{a}}^+$. Then
\[
\dim \widetilde{Z}^{\mathrm{Zar}} + \dim (\mathbf{u}(\widetilde{Z}))^{\mathrm{Zar}} \ge \dim \widetilde{Z} + \dim (\mathbf{u}(\widetilde{Z}))^{\mathrm{biZar}},
\]
where $\widetilde{Z}^{\mathrm{Zar}}$ means the smallest irreducible algebraic subset of $\cX_{2g,\mathrm{a}}^+$ which contains $\widetilde{Z}$.
\end{thm}



\subsection{A quick introduction to Shimura varieties}\label{SubsectionMSV}
We gather some notation and facts on Shimura varieties. We will need the knowledge (only) for the proofs in $\mathsection$\ref{SectionDegeneracyLocusClosed} and $\mathsection$\ref{SectionCriterionDegeneracy}.

Recall that in $\mathsection$\ref{SubsectionUnifSpaceOfModuliSpace}, we have associated a reductive $\mathbb{Q}$-group $\mathrm{GSp}_{2g}$ and a complex space $\mathfrak{H}_g^+$ to the moduli space $\mathbb{A}_g$. In other words we have associated a pair $(\mathrm{GSp}_{2g}, \mathfrak{H}_g^+)$ to $\mathbb{A}_g$, where 
\begin{itemize}
\item $\mathrm{GSp}_{2g}$ is a reductive $\mathbb{Q}$-group;
\item $\mathfrak{H}_g^+$ is a complex space on which $\mathrm{GSp}_{2g}(\mathbb{R})^+$ acts transitively;
\item As a complex space, $\mathbb{A}_g$ is the quotient of $\mathfrak{H}_g^+$ by a congruence subgroup of $\mathrm{GSp}_{2g}^{\mathrm{der}}(\mathbb{Z})$.
\end{itemize}
This pair $(\mathrm{GSp}_{2g},\mathfrak{H}_g^+)$ is a special case of \textit{pure Shimura datum}, and the third bullet point makes $\mathbb{A}_g$ a \text{pure Shimura variety}. In general, a pure Shimura datum is a pair $(G,\cX_G^+)$ such that $G$ is a reductive $\mathbb{Q}$-group and $\cX_G^+$ is a complex space on which $G^{\mathrm{der}}(\mathbb{R})^+$ acts transitively (along with some other properties). A pure Shimura variety is a quotient space $\Gamma_G \backslash \cX_G^+$ for some congruence subgroup $\Gamma_G$ of $G^{\mathrm{der}}(\mathbb{Z})$.


Next we turn to the universal abelian variety $\mathfrak{A}_g$. In $\mathsection$\ref{SubsectionUnivAbVar} we have associated with it a $\mathbb{Q}$-group $P_{2g,\mathrm{a}}$ and a complex space $\cX_{2g,\mathrm{a}}^+$. Note that $P_{2g,\mathrm{a}}$ is not a reductive group. Nevertheless we have the following properties for the pair $(P_{2g,\mathrm{a}}, \cX_{2g,\mathrm{a}}^+)$: 
\begin{itemize}
\item $P_{2g,\mathrm{a}}$ is a $\mathbb{Q}$-group, whose unipotent radical is a vector group;
\item $\cX_{2g,\mathrm{a}}^+$ is a complex space on which $P_{2g,\mathrm{a}}(\mathbb{R})^+$ acts transitively;
\item As a complex space, $\mathfrak{A}_g$ is the quotient of $\cX_{2g,\mathrm{a}}^+$ by a congruence subgroup of $P_{2g,\mathrm{a}}^{\mathrm{der}}(\mathbb{Z})$.
\end{itemize}
This makes the pair $(P_{2g,\mathrm{a}}, \cX_{2g,\mathrm{a}}^+)$ a \textit{mixed Shimura datum of Kuga type}, and the third bullet point makes $\mathfrak{A}_g$ a \textit{mixed Shimura variety of Kuga type}. In general a mixed Shimura datum of Kuga type is a pair $(P,\cX^+)$ such that $P$ is a $\mathbb{Q}$-group whose unipotent radical is a vector group, and $\cX^+$ is a complex space on which $P(\mathbb{R})^+$ acts transitively (along with some other properties). A mixed Shimura variety of Kuga type is a quotient space $\Gamma \backslash \cX^+$ for some congruence subgroup $\Gamma$ of $P^{\mathrm{der}}(\mathbb{Z})$. 

It is worth pointing out that any pure Shimura datum is a mixed Shimura datum of Kuga type (such that the unipotent radical of the underlying group is trivial). Given two mixed Shimura data of Kuga type $(Q,\cY^+)$ and $(P,\cX^+)$, a map $f \colon (Q,\cY^+) \rightarrow (P,\cX^+)$ is called a \textit{Shimura morphism} if $f$ is a group homomorphism on the underlying groups and is a holomorphic morphism on the underlying complex spaces, and that $f(q \cdot \widetilde{y}) = f(q) \cdot f(\widetilde{y})$ for any $q \in Q(\mathbb{R})^+$ and any $\widetilde{y} \in \cY^+$. For a Shimura morphism $f$ of this form, we say that $(f(Q), f(\cY^+))$ is a \textit{mixed Shimura subdatum} of $(P,\cX^+)$. Applying the discussion to $(P,\cX^+) = (P_{2g,\mathrm{a}} , \cX_{2g,\mathrm{a}}^+)$, we get the definition of mixed Shimura subdata of $(P_{2g,\mathrm{a}}, \cX_{2g,\mathrm{a}}^+)$. Applying the discussion to $(G,\cX_G^+) = (\mathrm{GSp}_{2g}, \mathfrak{H}_g^+)$, we get the definition of Shimura subdata of $(\mathrm{GSp}_{2g}, \mathfrak{H}_g^+)$.

It is known that for each mixed Shimura subdatum $(Q,\cY^+)$ of $(P_{2g,\mathrm{a}}, \cX_{2g,\mathrm{a}}^+)$, the unipotent radical of $Q$ is $V_{2g} \cap Q$ by weight reasons; see \cite[Proposition~2.9]{GaoTowards-the-And}.

Define the \textbf{special subvarieties} of $\mathfrak{A}_g$ to be the subvarieties of the form $\mathbf{u}(\cY^+)$ where $\cY^+$ is the underlying space of some mixed Shimura subdatum $(Q,\cY^+)$ of $(P_{2g,\mathrm{a}}, \cX_{2g,\mathrm{a}}^+)$. Define the \textit{special subvarieties} of $\mathbb{A}_g$ to be the subvarieties of the form $\mathbf{u}_G(\cY_G^+)$ where $\cY_G^+$ is the underlying space of some Shimura subdaum $(H, \cY_G^+)$ of $(\mathrm{GSp}_{2g}, \mathfrak{H}_g^+)$.

Let us end this subsection by the following proposition on the geometric meaning of special subvarieties of $\mathfrak{A}_g$ and beyond. Recall the notations $\widetilde{\pi} \colon P_{2g,\mathrm{a}} \rightarrow \mathrm{GSp}_{2g}$ at the end of $\mathsection$\ref{SubsectionUnivAbVar}  and the uniformizations (see \eqref{DiagramUnivAbVarAndModuliSpace})
\begin{equation}\label{DiagramUnivAbVarAndModuliSpace2}
\xymatrix{
\cX_{2g,\mathrm{a}}^+ \ar[r]^-{\widetilde{\pi}} \ar[d]_{\mathbf{u}} & \mathfrak{H}_g^+ \ar[d]^{\mathbf{u}_G} \\
\mathfrak{A}_g \ar[r]^{\pi} & \mathbb{A}_g.
}
\end{equation}

\begin{prop}\label{PropSpecialSubvarUnivAb}
Let $M$ be a special subvariety of $\mathfrak{A}_g$. Then $M_G := \pi(M)$ is a special subvariety of $\mathbb{A}_g$, and $M$ is the translate of an abelian subscheme of $\mathfrak{A}_g|_{M_G} \rightarrow M_G$ by a torsion section. Conversely all special subvarieties of $\mathfrak{A}_g$ are obtained in this way.

More precisely, if $M$ is associated with the mixed Shimura subdatum of Kuga type $(Q,\cY^+)$ of $(P_{2g,\mathrm{a}}, \cX_{2g,\mathrm{a}}^+)$ (namely $M = \mathbf{u}(\cY^+)$), then the relative dimension of $M \rightarrow M_G$ is $g_Q := \frac{1}{2}\dim(V_{2g} \cap Q)$.
\end{prop}

\subsection{Quotient by a normal group}\label{SubsectionGeomQuotByNormalSubgp}
In $\mathsection$\ref{SectionDegeneracyLocusClosed} and $\mathsection$\ref{SectionCriterionDegeneracy}, we need the geometric interpretation of the operation of taking the \textit{quotient mixed Shimura varieties of Kuga type}. We explain this in the current subsection.



The setting is as follows. Let $(Q,\cY^+)$ be a mixed Shimura subdatum of $(P_{2g,\mathrm{a}}, \cX_{2g,\mathrm{a}}^+)$, and let $M = \mathbf{u}(\cY^+)$ be the associated special subvariety of $\mathfrak{A}_g$. Take a normal subgroup $N$ of $Q$. Pink \cite[2.9]{PinkThesis} constructed the \textit{quotient mixed Shimura datum} $(Q,\cY^+)/N$ whose underlying group is $Q/N$.

By Proposition~\ref{PropSpecialSubvarUnivAb}, $M_G := \pi(M)$ is a special subvariety of $\mathbb{A}_g$. Let $(G_Q, \cY_{G_Q}^+)$ be the pure Shimura subdatum of $(\mathrm{GSp}_{2g}, \mathfrak{H}_g^+)$ in Proposition~\ref{PropSpecialSubvarUnivAb}, then $M_G = \mathbf{u}_G(\cY_{G_Q}^+)$.

Denote by $V_Q = V_{2g} \cap Q$ and $V_N = V_{2g} \cap N$. We explained above \eqref{DiagramUnivAbVarAndModuliSpace2} that $V_Q$ is the unipotent radical of $Q$. Since $N \lhd Q$, group theory implies that $V_N$ is the unipotent radical of $N$. Denote by $G_N = N/V_N$, then $G_N$ is a normal subgroup of $G_Q$.

Proposition~\ref{PropSpecialSubvarUnivAb} says that 
$\pi|_M \colon M \rightarrow M_G$ itself is an abelian scheme of relative dimension $g_Q = \frac{1}{2}\dim (V_{2g} \cap Q)$. Let $\epsilon \colon M_G \rightarrow M$ be the zero section.\footnote{It is a torsion section of $\mathfrak{A}_g|_{M_G} \rightarrow M_G$.} It induces a Levi decomposition $Q = V_Q \rtimes G_Q$ and a semi-algebraic isomorphism 
\[
\cY^+ \cong V_Q(\mathbb{R}) \times \cY_{G_Q}^+
\]
such that $\mathbf{u}(\{0\} \times \cY_{G_Q}^+) = \epsilon (M_G)$. In the rest of this subsection, we shall use this identification of $\cY^+$ with $V_Q(\mathbb{R}) \times \cY_{G_Q}^+$.

Deligne \cite[Rappel~4.4.3]{DeligneHodgeII} proved that $\cY^+ \rightarrow \cY^+_{G_Q}$ is a variation of Hodge structure of type $(-1,0)+(0,-1)$.

We have $V_N \lhd Q$ since $N\lhd Q$ and $V_N$ is the unipotent radical of $N$. The reductive group $G_Q$, as a subgroup of $\mathrm{GSp}_{2g}$, acts on $V_{2g}$. Now $V_N \lhd Q$ implies that $V_N$ is a $G_Q$-submodule of $V_{2g}$, so $V_N$ is a sub-Hodge structure of $V_Q$. Thus $V_N(\R) \times \cY_{G_Q}^+ \rightarrow \cY^+_{G_Q}$ is a sub-variation of Hodge structure of $\cY^+ \rightarrow \cY^+_{G_Q}$. Thus $\mathbf{u}(V_N(\mathbb{R}) \times \cY_{G_Q}^+)$ is an abelian subscheme of $M \rightarrow M_G$ by \cite[Rappel~4.4.3]{DeligneHodgeII}.

The quotient $\widetilde{p}_N \colon (Q,\cY^+) \rightarrow (Q,\cY^+)/N$ can be constructed in two steps. First we take the quotient $(Q_0,\cY_0^+) := (Q,\cY^+)/V_N$, and then we take $(Q', \cY^{\prime+}) := (Q_0,\cY_0^+)/G_N$.

The geometric meaning of the quotient $\widetilde{p}_0 \colon (Q,\cY^+) \rightarrow (Q_0,\cY_0^+) = (Q,\cY^+)/V_N$ is as follows. We have the following commutative diagram (where $\Gamma_0 = \widetilde{p}_0(\Gamma \cap Q(\mathbb{R}))$)
\begin{equation}\label{EqFirstStepQuotMSV}
\xymatrix{
\cY^+ \ar[r]^-{\widetilde{p}_0} \ar[d]_{\mathbf{u}|_{\cY^+}} & \cY_0^+ \ar[d] \\
M \ar[r]^-{p_0} \ar[d]_{\pi|_M} & M_0 := \Gamma_0 \backslash \cY_0^+ \ar[d]^{\pi_0} \\
M_G \ar[r]^-{\mathrm{id}_{M_G}} & M_G
}
\end{equation}
such that $M_0 \rightarrow M_G$ is the abelian scheme obtained by taking the quotient of $M \rightarrow M_G$ by $\mathbf{u}(V_N(\mathbb{R}) \times \cY_{G_Q}^+)$. Denote by $g_N = \frac{1}{2} \dim V_N$, then the relative dimension of $M_0 \rightarrow M_G$ is $g_Q - g_N$.

To explain the quotient $\widetilde{p}' \colon (Q_0, \cY_0^+) \rightarrow (Q', \cY^{\prime+}) = (Q_0,\cY_0^+)/G_N$, we need the following preliminary. We know that $G_N$ is a normal subgroup of $G_Q$. We take for granted that we can do the operation $(G_Q, \cY_{G_Q}^+)/G_N$ for pure Shimura datum, and this quotient gives a morphism $p_{G_N} \colon M_G \rightarrow M'_G$.\footnote{See \cite[2.9]{PinkThesis} or \cite[Definition~2.1]{UllmoA-characterisat}. In this paper we mostly only need the notion, so we choose not go into more details on this in the preliminary part.}

For each $y'_G \in M'_G$, the inverse image $(p_{G_N})^{-1}(y'_G)$ equals $\mathbf{u}_G(G_N(\mathbb{R})^+\widetilde{y}_G)$ for some $\widetilde{y}_G \in \cY_{G_Q}^+$. So the connected algebraic monodromy group of $(p_{G_N})^{-1}(y'_G)$ is $G_N^{\mathrm{der}}$. 
Since $N \lhd Q$, we have $G_N = N/V_N \lhd Q/V_N$. Hence $G_N$ acts trivially on $V_Q/V_N$. So $M_0|_{(p_{G_N})^{-1}(y'_G)} \rightarrow (p_{G_N})^{-1}(y'_G)$ is an isotrivial abelian scheme by Deligne's Theorem of the Fixed Part \cite[Corollaire~4.1.2]{DeligneHodgeII}. We have the following commutative diagram
\begin{equation}\label{EqSecondStepQuotMSV}
\xymatrix{
\cY_0^+ \ar[r]^-{\widetilde{p}'} \ar[d] & \cY_0' \ar[d] \\
M_0 = \Gamma_0 \backslash \cY_0^+ \ar[r]^-{p'} \ar[d]_{\pi_0} & M' \ar[d]^{\pi'} \\
M_G \ar[r]^-{p_{G_N}} & M'_G
}
\end{equation}
where $\pi'$ is an abelian scheme (of relative dimension $g_Q - g_N$) such that each closed fiber of $M_0|_{(p_{G_N})^{-1}(y'_G)} \rightarrow (p_{G_N})^{-1}(y'_G)$ is the abelian variety $(\pi')^{-1}(y'_G)$. In other words, the lower box is an intermediate step of the modular map
\[
\xymatrix{
M_0 \ar[r] \ar[d]_{\pi_0} \pullbackcorner & \mathfrak{A}_{g_Q-g_N} \ar[d] \\
M_G \ar[r] & \mathbb{A}_{g_Q-g_N}.
}
\]

We end this subsection by summarizing the quotient $\widetilde{p}_N \colon (Q,\cY^+) \rightarrow (Q,\cY^+)/N$ (after taking the uniformizations) in the following commutative diagram:
\begin{equation}\label{EqGeomMeaningQuotShimura}
\xymatrix{
M \ar[r]|-{p_0} \ar[d]_{\pi|_M} \ar@/^1pc/[rr]|-{p_N} & M_0 := \Gamma_0 \backslash \cY_0^+ \ar[d]_{\pi_0} \ar[r]|-{p'} \pullbackcorner & M' \ar[d]_{\pi'} \ar[r] \pullbackcorner & \mathfrak{A}_{g_Q-g_N} \ar[d] \\
M_G \ar[r]^-{\mathrm{id}_{M_G}} & M_G \ar[r]^-{p_{G_N}} & M'_G \ar[r] & \A_{g_Q-g_N}
}
\end{equation}
where each vertical arrow is an abelian scheme, with the left one of relative dimension $g_Q$ and the other three of relative dimension $g_Q - g_N$. Moreover $M_0|_{p_{G_N}^{-1}(b')}$ is isotrivial for any $b' \in M'_G$ by the discussion above \eqref{EqSecondStepQuotMSV}. And for any $x' \in M'$, we have that $p_N^{-1}(x')$ is the translate of an abelian subscheme of $M|_{\pi(p_N^{-1}(x'))} \rightarrow \pi(p_N^{-1}(x'))$ of relative dimension $g_Q - (g_Q - g_N) = g_N$ by a constant section. 


\section{From Betti map to the $t$-th degeneracy locus}\label{SectionSubvarBettiNonMaxRank}
In $\mathsection$\ref{SectionSubvarBettiNonMaxRank}-\ref{SectionCriterionDegeneracy} we will prove Theorem~\ref{ThmCriterionDegIntro} for $X \subseteq \mathfrak{A}_g$. It is arranged as follows: $\mathsection$\ref{SectionSubvarBettiNonMaxRank} transfers the study of the generic rank of the Betti map to the $t$-th degenerate locus of $X$ for some particular $t$'s, $\mathsection$\ref{SectionDegeneracyLocusClosed} proves the Zariski closedness of the $t$-th degenerate locus $X^{\mathrm{deg}}(t)$, and $\mathsection$\ref{SectionCriterionDegeneracy} gives the criterion to $X^{\mathrm{deg}}(t) = X$.

Let us fix the notations. Recall the uniformizations \eqref{DiagramUnivAbVarAndModuliSpace2}
\[
\xymatrix{
\cX_{2g,\mathrm{a}}^+ \ar[r]^-{\widetilde{\pi}} \ar[d]_{\mathbf{u}} & \mathfrak{H}_g^+ \ar[d]^{\mathbf{u}_G} \\
\mathfrak{A}_g \ar[r]^{\pi} & \mathbb{A}_g
}
\]
and the uniformized universal Betti map $\tilde{b} \colon \cX_{2g,\mathrm{a}}^+ \rightarrow \mathbb{R}^{2g}$ defined in \eqref{EqUnivUniformizedBettiMap}. Let $X$ be an irreducible subvariety of $\mathfrak{A}_g$. Fix a complex analytic irreducible component $\widetilde{X}$ of $\mathbf{u}^{-1}(X)$.


\begin{prop}\label{PropDegLocus}
Denote for simplicity by $d = \dim X$. For any integer $l \in \{1,\ldots,d\}$, let
\[
\widetilde{X}_{<2l} = \{ \widetilde{x} \in \widetilde{X} : \mathrm{rank}_{\mathbb{R}}(\tilde{b}|_{\widetilde{X}})_{\widetilde{x}} < 2l, ~ \mathbf{u}(\widetilde{x}) \in X^{\mathrm{sm}}(\mathbb{C}) \}.
\]
Then $X^{\mathrm{deg}}(l-d) \cap X^{\mathrm{sm}}(\mathbb{C}) \subseteq \mathbf{u}(\widetilde{X}_{<2l})
$, where $X^{\mathrm{deg}}(l-d)$ is defined in Definition~\ref{DefnDegeneracyLocus}. 

Conversely, if $\mathrm{rank}_{\mathbb{R}}(\tilde{b}|_{\widetilde{X}})_{\tilde{x}} < 2l$ for all $\widetilde{x} \in \widetilde{X}$ with $\mathbf{u}(\widetilde{x}) \in X^{\mathrm{sm}}(\mathbb{C})$, then $X^{\mathrm{deg}}(l-d)$ is Zariski dense in $X$.
\end{prop}

\begin{proof}

Let us prove $X^{\mathrm{deg}}(l-d) \cap X^{\mathrm{sm}}(\mathbb{C}) \subseteq \mathbf{u}(\widetilde{X}_{<2l})$. Suppose not, then there exists some $\widetilde{x} \in \widetilde{X}$ such that $\mathbf{u}(\widetilde{x}) \in X^{\mathrm{deg}} \cap X^{\mathrm{sm}}(\mathbb{C})$ and $\mathrm{rank}_{\mathbb{R}}(\tilde{b}|_{\widetilde{X}})_{\widetilde{x}} = 2 l$. By definition of $X^{\mathrm{deg}}(l-d)$, the point $\mathbf{u}(\widetilde{x})$ lies in a subvariety $Y$ of $X$ such that $\dim \langle Y \rangle_{\mathrm{sg}} - \dim \pi(Y) < \dim Y + (l-d)$, 
where $\langle Y \rangle_{\mathrm{sg}}$ is defined above Definition~\ref{DefnDegeneracyLocus}.

Let $\widetilde{Y}$ be a complex analytic irreducible component of $\mathbf{u}^{-1}(Y)$ with $\widetilde{x} \in \widetilde{Y} \subseteq \widetilde{X}$. Observe that $\mathrm{rank}_{\mathbb{R}}(\tilde{b}|_{\widetilde{X}})_{\widetilde{x}}  \le \mathrm{rank}_{\mathbb{R}}(\tilde{b}|_{\widetilde{Y}})_{\widetilde{x}} + 2 (\dim X - \dim Y)$. So $2l - 2(d-\dim Y) = \mathrm{rank}_{\mathbb{R}}(\tilde{b}|_{\widetilde{X}})_{\widetilde{x}} - 2(\dim X - \dim Y) \le \mathrm{rank}_{\mathbb{R}}(\tilde{b}|_{\widetilde{Y}})_{\widetilde{x}}$. So for $\widetilde{\langle Y \rangle_{\mathrm{sg}}}$ a complex analytic irreducible component of $\mathbf{u}^{-1}(\langle Y \rangle_{\mathrm{sg}})$ which contains $\widetilde{Y}$, we have 
\begin{align*}
\mathrm{rank}_{\mathbb{R}}(\tilde{b}|_{\widetilde{\langle Y \rangle_{\mathrm{sg}}}})_{\tilde{x}} & = 2(\dim \langle Y \rangle_{\mathrm{sg}} - \dim \pi(Y)) \\
& < 2l - 2(d-\dim Y) \qquad \text{ by choice of }Y \\
& \le \mathrm{rank}_{\mathbb{R}}(\tilde{b}|_{\widetilde{Y}})_{\widetilde{x}}
\end{align*}
But this cannot happen as $\tilde{Y} \subseteq \tilde{\langle Y \rangle_{\mathrm{sg}}}$.


Conversely, assume $2l':=\max_{\widetilde{x} \in \widetilde{X},~ \mathbf{u}(\widetilde{x}) \in X^{\mathrm{sm}}(\C)}\mathrm{rank}_{\mathbb{R}}(\tilde{b}|_{\widetilde{X}})_{\tilde{x}} < 2l$. There exists a non-empty open (in the usual topology) subset $\widetilde{U}$ of $\widetilde{X}$ on which $\mathrm{rank}_{\mathbb{R}}(\tilde{b}|_{\widetilde{X}})_{\tilde{x}} = 2l'$ for all $\tilde{x} \in \widetilde{U}$.
By \cite[Appendix~II, Corollary~7F]{Whitney}, each fiber of $\widetilde{b}|_{\widetilde{U}}$ has (real-)dimension $2d - 2l' > 2(d-l)$. 

Let $\widetilde{x} \in \widetilde{U}$ and set $r = \widetilde{b}(\widetilde{x})$, then
\begin{equation}\label{EqBettiRankGivesSubset}
(\dim_{\mathbb{R}})_{\widetilde{x}} (\tilde{b}^{-1}(r) \cap \widetilde{X})  > 2(d - l).
\end{equation}

In the rest of the proof, we identify $\cX_{2g,\mathrm{a}}^+$ as the semi-algebraic space $\mathbb{R}^{2g} \times \mathfrak{H}_g^+$ with the complex structure defined by \eqref{EqComplexStrOfX2g}. In particular $\tilde{b}^{-1}(r) = \{r\} \times \mathfrak{H}_g^+$. 
Property (ii) of the Betti map (below \eqref{EqUnivBettiMap}) implies that $\tilde{b}^{-1}(r) \cap \widetilde{X}$ is complex analytic. Then by \eqref{EqBettiRankGivesSubset}, there exists a complex analytic irreducible subset $\widetilde{W}$ in $\mathfrak{H}_g^+$ of dimension $\ge d-l + 1$ such that
\[
\widetilde{x} \in \{r\} \times \widetilde{W} \subseteq \widetilde{X}.
\]
Now $\mathbf{u}(\widetilde{U})$ is Zariski dense in $X$ because it contains a non-empty open subset (in the usual topology) of $X^{\mathrm{sm,an}}$. So it suffices to prove the following assertion: $Y = \left(\mathbf{u}(\{r\} \times \widetilde{W})\right)^{\mathrm{Zar}}$ satisfies
\begin{equation}\label{EqGoalBettiDegenerate}
\dim \langle Y \rangle_{\mathrm{sg}} - \dim \pi(Y) < \dim Y + (l-d).
\end{equation}

Apply weak Ax-Schanuel for $\mathfrak{A}_g$, namely Theorem~\ref{ThmWAS}, to $\{r\} \times \widetilde{W}$. Then we get
\begin{equation}\label{EqWASBetti}
\dim (\{r\} \times \widetilde{W})^{\mathrm{Zar}} + \dim Y \ge \dim (\{r\} \times \widetilde{W}) + \dim Y^{\mathrm{biZar}}.
\end{equation}

On the other hand we have
\begin{equation}\label{EqLongInclusionWY}
\dim \tilde{W}^{\mathrm{Zar}} \le \dim \tilde{W}^{\mathrm{biZar}} 
= \dim (\bu_G(\tilde{W}))^{\mathrm{biZar}} \le \dim \pi(Y)^{\mathrm{biZar}},
\end{equation}
where the last inequality holds because $\bu_G(\tilde{W})\subseteq \pi(Y)$ by definition of $Y$.

Let us temporarily assume $(\{r\} \times \widetilde{W})^{\mathrm{Zar}} = \{r\} \times \widetilde{W}^{\mathrm{Zar}}$ and finish the proof. Then from \eqref{EqWASBetti} and \eqref{EqLongInclusionWY} we get
\[
\dim \pi(Y)^{\mathrm{biZar}} + \dim Y \ge \dim \tilde{W} + \dim Y^{\mathrm{biZar}}.
\]
Thus
\[
\dim \langle Y \rangle_{\mathrm{sg}} - \dim \pi(Y) = \dim Y^{\mathrm{biZar}} - \dim \pi(Y)^{\mathrm{biZar}} \le \dim Y - \dim \tilde{W},
\]
where the first equality follows from Corollary~\ref{CorsgBiZar}. Therefore \eqref{EqGoalBettiDegenerate} holds since $\dim \tilde{W} \ge d-l+1$.

It remains to prove $(\{r\} \times \widetilde{W})^{\mathrm{Zar}} = \{r\} \times \widetilde{W}^{\mathrm{Zar}}$. First note that $\{r\} \times \widetilde{W}^{\mathrm{Zar}}$ is semi-algebraic and complex analytic. Then it is a general fact that $\{r\} \times \widetilde{W}^{\mathrm{Zar}}$ is algebraic in $\cX_{2g,\mathrm{a}}^+$; see \cite[(the proof of) Lemma~B.1]{KlinglerThe-Hyperbolic-}. Thus $(\{r\} \times \widetilde{W})^{\mathrm{Zar}} \subseteq \{r\} \times \widetilde{W}^{\mathrm{Zar}}$. Remark~\ref{RemarkBiAlgSystem} says that the algebraic structures on $\cX_{2g,\mathrm{a}}^+$ and $\mathfrak{H}_g^+$ are compatible under $\widetilde{\pi}$, so $\widetilde{\pi}|_{(\{r\} \times \widetilde{W})^{\mathrm{Zar}}} \colon (\{r\} \times \widetilde{W})^{\mathrm{Zar}} \rightarrow \widetilde{W}^{\mathrm{Zar}}$ is dominant. Now we can conclude.
\end{proof}

\section{Zariski closeness of the $t$-th degeneracy locus}\label{SectionDegeneracyLocusClosed}
The goal of this section is to prove Theorem~\ref{ThmZarClosedXdegIntro} for $X \subseteq \mathfrak{A}_g$. Let $X$ be an irreducible subvariety of $\mathfrak{A}_g$. 
Through the whole section we fix a $t \in \Z$. Let $X^{\mathrm{deg}}(t)$ be the $t$-th degeneracy locus of $X$ defined in Definition~\ref{DefnDegeneracyLocus}.
\begin{thm}\label{ThmDegeneracyLocusZarClosed}
The subset $X^{\mathrm{deg}}(t)$ is Zariski closed in $X$.
\end{thm}


Our proof of Theorem~\ref{ThmDegeneracyLocusZarClosed} is inspired by Daw-Ren's work \cite[$\mathsection$7]{DawRenAppOfAS} on the anomalous subvarieties in a pure Shimura variety.

\subsection{Weakly optimal subvariety}
To prove Theorem~\ref{ThmDegeneracyLocusZarClosed}, we need the following definition of \textit{weakly optimal} subvarieties. The notion was first introduced by Habegger-Pila \cite{HabeggerO-minimality-an} to study the Zilber-Pink conjecture for abelian varieties and product of modular curves.
\begin{defn}
\begin{itemize}
\item[(i)] For any irreducible subvariety $Z$ of $\mathfrak{A}_g$, define the weakly defect to be $\delta_{\mathrm{ws}}(Z) = \dim Z^{\mathrm{biZar}} - \dim Z$.
\item[(ii)] A closed irreducible subvariety $Z$ of $X$ is said to be \textbf{weakly optimal} if the following condition holds: $Z \subsetneq Z' \subseteq X$ with $Z'$ irreducible closed in $X$ $\Rightarrow$ $\delta_{\mathrm{ws}}(Z') > \delta_{\mathrm{ws}}(Z)$.
\end{itemize}
\end{defn}

Weakly optimal subvarieties of $X$ are closely related to $X^{\mathrm{deg}}(t)$ by the following lemmas.
\begin{lemma}\label{LemmaEquivOfTwoInequality}
$\dim \langle Z \rangle_{\mathrm{sg}} - \dim \pi(Z) < \dim Z + t \Leftrightarrow \delta_{\mathrm{ws}}(Z) <  \dim \pi(Z)^{\mathrm{biZar}} + t$.
\end{lemma}
\begin{proof}
The condition $\dim \langle Z \rangle_{\mathrm{sg}} - \dim \pi(Z) < \dim Z + t$ can be rewritten to be
\[
((\dim \langle Z \rangle_{\mathrm{sg}} - \dim \pi(Z)) + \dim \pi(Z)^{\mathrm{biZar}} ) - \dim Z <  \dim \pi(Z)^{\mathrm{biZar}} + t.
\]
By Corollary~\ref{CorsgBiZar}, we have $\dim Z^{\mathrm{biZar}} = (\dim \langle Z \rangle_{\mathrm{sg}} - \dim \pi(Z)) + \dim \pi(Z)^{\mathrm{biZar}}$. Hence the inequality above becomes $\delta_{\mathrm{ws}}(Z) <  \dim \pi(Z)^{\mathrm{biZar}} + t$.
\end{proof}

\begin{lemma}\label{LemmaBadLocusWeaklyOptimal}
Assume a closed irreducible subvariety $Z \subseteq X$ satisfies $\dim \langle Z \rangle_{\mathrm{sg}} - \dim \pi(Z) < \dim Z + t$ and is maximal for this property. Then $Z$ is weakly optimal.
\end{lemma}
\begin{proof}
For any $Z \subseteq Z' \subseteq X$ with $Z'$ irreducible closed in $X$, if $\delta_{\mathrm{ws}}(Z') \le \delta_{\mathrm{ws}}(Z)$, then we have
\[
\delta_{\mathrm{ws}}(Z') \le \delta_{\mathrm{ws}}(Z) < \dim \pi(Z)^{\mathrm{biZar}} + t \le \dim \pi(Z')^{\mathrm{biZar}} + t
\]
where the second inequality follows from Lemma~\ref{LemmaEquivOfTwoInequality} (applied to $Z$). Applying Lemma~\ref{LemmaEquivOfTwoInequality} to $Z'$, we get $\dim \langle Z' \rangle_{\mathrm{sg}} - \dim \pi(Z') < \dim Z' + t$. The maximality of $Z$ then implies $Z = Z'$. So $Z$ is weakly optimal.
\end{proof}

Thus $X^{\mathrm{deg}}(t)$ is the union of some weakly optimal subvarieties. We will show that this union is finite. The key point is the following finiteness theorem concerning weakly optimal subvarieties \cite[Theorem~1.4]{GaoAxSchanuel}. See $\mathsection$\ref{SubsectionMSV} for notation.
\begin{thm}\label{ThmFinitenessALaBogomolov}
There exists a finite set $\Sigma$ consisting of elements of the form $((Q,\cY^+),N)$, where $(Q,\cY^+)$ is a mixed Shimura subdatum of Kuga type of $(P_{2g,\mathrm{a}}, \cX_{2g,\mathrm{a}}^+)$ and $N$ is a connected normal subgroup of $Q$ whose reductive part is semi-simple, such that the following property holds. If a closed irreducible subvariety $Z$ of $X$ is weakly optimal, then there exists $((Q,\cY^+),N) \in \Sigma$ such that $Z^{\mathrm{biZar}} = \mathbf{u}(N(\mathbb{R})^+\tilde{y})$ for some $\tilde{y} \in \cY^+$.
\end{thm}

\subsection{An auxiliary proposition}
Let $X$ be as in Theorem~\ref{ThmDegeneracyLocusZarClosed}.

Let $Z$ be a positive dimensional closed irreducible subvariety $X$ such that $\dim \langle Z \rangle_{\mathrm{sg}} - \dim \pi(Z) < \dim Z + t$ and is maximal for this property. Then $Z$ is weakly optimal by Lemma~\ref{LemmaBadLocusWeaklyOptimal}. So Theorem~\ref{ThmFinitenessALaBogomolov} gives a finite set $\Sigma$ and a $((Q,\cY^+),N) \in \Sigma$ such that $Z^{\mathrm{biZar}} = \mathbf{u}(N(\mathbb{R})^+\tilde{y})$ for some $\tilde{y} \in \cY^+$. Recall that $N$ is a connected normal subgroup of $Q$ whose reductive part is semi-simple. Note that $N \not= 1$ since $Z$ has positive dimension.


Consider the operation of taking quotient mixed Shimura datum $\widetilde{p}_N \colon (Q,\cY^+) \rightarrow (Q,\cY^+)/N =: (Q', \cY^{\prime+})$ discussed in $\mathsection$\ref{SubsectionGeomQuotByNormalSubgp} and the induced morphism on the corresponding mixed Shimura varieties of Kuga type
\begin{equation}\label{EqQuotXdegZarClosed}
\xymatrix{
\cY^+ \ar[r]^-{\widetilde{p}_N} \ar[d]_{\mathbf{u}|_{\cY^+}} & \cY^{\prime+} \ar[d] \\
M \ar[r]^-{p_N} & M' 
}
\end{equation}
We refer to \eqref{EqGeomMeaningQuotShimura} for the geometric meaning of $p_N \colon M \rightarrow M'$ and notation.

For any integer $h$, define
\begin{equation}\label{EqDefEh}
E_h = \{ x \in X : \dim_x (p_N|_{X\cap M})^{-1}(p_N(x)) > h\}.
\end{equation}
Then $E_h$ is Zariski closed in $X$.

\begin{prop}\label{PropAuxForZarClosedness}
For $g_N = \frac{1}{2}\dim (V_{2g} \cap N)$ as in \eqref{EqGeomMeaningQuotShimura}, we have
\begin{enumerate}
\item[(i)] We have $Z \subseteq E_{g_N - t}$.
\item[(ii)] We have $E_{g_N - t} \subseteq X^{\mathrm{deg}}(t)$.
\end{enumerate}
\end{prop}
\begin{proof}
\begin{enumerate}
\item[(i)] Recall that $Z$ satisfies $\dim \langle Z \rangle_{\mathrm{sg}} - \dim \pi(Z) < \dim Z + t$. Therefore $\delta_{\mathrm{ws}}(Z) <  \dim \pi(Z)^{\mathrm{biZar}} + t$ by Lemma~\ref{LemmaEquivOfTwoInequality}. Hence
\[
\dim Z > \dim  Z^{\mathrm{biZar}} - \dim \pi(Z)^{\mathrm{biZar}} - t.
\]
By construction we know that $Z^{\mathrm{biZar}}$ is a fiber of $p_N$. Hence as a morphism, $Z^{\mathrm{biZar}} \rightarrow \pi(Z)^{\mathrm{biZar}}$ is an abelian scheme of relative dimension $g_N$ by the discussion below \eqref{EqGeomMeaningQuotShimura}. So the inequality above becomes $\dim Z > g_N - t$. As $Z \subseteq X \cap M$, we have that $Z \subseteq E_{g_N - t}$ by definition of $E_{g_N - t}$.

\item[(ii)] For any $x \in E_{g_N - t}$, there exists a component $Y$ of $(p_N|_{X\cap M})^{-1}(p_N(x))$ containing $x$ such that $\dim Y > g_N - t$. Denote by $x' = p_N(x)$.

By \eqref{EqGeomMeaningQuotShimura} and the discussion below, $p_N^{-1}(x') \rightarrow \pi(p_N^{-1}(x'))$ is an abelian scheme of relative dimension $g_N$.

Proposition~\ref{PropBiAlgAg} says that as a morphism, $Y^{\mathrm{biZar}} \rightarrow \pi(Y)^{\mathrm{biZar}}$ is an abelian scheme. Since $p_N^{-1}(x')$ is bi-algebraic (\!\cite[Corollary~8.3]{GaoTowards-the-And}) and $Y \subseteq p_N^{-1}(x')$, we have $Y^{\mathrm{biZar}} \subseteq p_N^{-1}(x')$. Thus the relative dimension of the abelian scheme $Y^{\mathrm{biZar}} \rightarrow \pi(Y)^{\mathrm{biZar}}$ is at most $g_N$. So
\[
\dim Y > g_N-t \ge \dim Y^{\mathrm{biZar}} - \dim \pi(Y)^{\mathrm{biZar}} -t.
\]
Hence $\delta_{\mathrm{ws}}(Y) < \dim \pi(Y)^{\mathrm{biZar}} + t$. 
Thus $\dim \langle Y \rangle_{\mathrm{sg}} - \dim \pi(Y) < \dim Y + t$ by Lemma~\ref{LemmaEquivOfTwoInequality}, and hence $Y \subseteq X^{\mathrm{deg}}(t)$ by definition. By varying $x \in E_h$, we get the conclusion. \qedhere
\end{enumerate}
\end{proof}

\subsection{Proof of Theorem~\ref{ThmDegeneracyLocusZarClosed}}\label{SubsectionEndOfProofOfZarClosedOfDegLocus}
Now we are ready to prove Theorem~\ref{ThmDegeneracyLocusZarClosed}. Let $X$ be as in Theorem~\ref{ThmDegeneracyLocusZarClosed}. Let $\Sigma$ be the finite set in Theorem~\ref{ThmFinitenessALaBogomolov}.

For any positive dimensional  closed irreducible subvariety $Z$ of $X$ such that $\dim \langle Z \rangle_{\mathrm{sg}} - \dim \pi(Z) < \dim Z + t$ and is maximal for this property, we obtain some $((Q,\cY^+),N) \in \Sigma$ from which we can construct a Zariski closed subset $E_{g_N - t}$ of $X$ such that $Z \subseteq E_{g_N - t}$ by (i) of Proposition~\ref{PropAuxForZarClosedness}. Since $\Sigma$ is a finite set, we have finitely many such $E_{g_N - t}$'s. By definition of $X^{\mathrm{deg}}(t)$, we then have that $X^{\mathrm{deg}}(t)$ is contained in the union of these $E_{g_N - t}$'s, which is a Zariski closed subset of $X$.

Conversely by (ii) of Proposition~\ref{PropAuxForZarClosedness}, each such $E_{g_N - t}$ is contained in $X^{\mathrm{deg}}(t)$.

Hence $X^{\mathrm{deg}}(t)$ is the union of these $E_{g_N - t}$'s. This is a finite union with each member being a closed subset of $X$. Hence $X^{\mathrm{deg}}(t)$ is Zariski closed in $X$.

\section{Criterion of degenerate subvarieties}\label{SectionCriterionDegeneracy}
Let $X$ be an irreducible subvariety of $\mathfrak{A}_g$. Through the whole section we fix a $t \in \Z$. Let $X^{\mathrm{deg}}(t)$ be the $t$-th degeneracy locus of $X$ defined in Definition~\ref{DefnDegeneracyLocus}.

The goal of this section is to prove the criterion for $X = X^{\mathrm{deg}}(t)$, and hence part (i) of Theorem~\ref{ThmCriterionDegIntro} for $X \subseteq \mathfrak{A}_g$ in view of Proposition~\ref{PropDegLocus} and Theorem~\ref{ThmDegeneracyLocusZarClosed}.

For notation let $B = \pi(X)$, let $\cA_X$ be the translate of an abelian subscheme of $\mathfrak{A}_g|_B \rightarrow B$ by a torsion section which contains $X$, minimal for this property. Then $\cA_X \rightarrow B$ itself is an abelian scheme (up to taking a finite covering of $B$), whose relative dimension we denote by $g_X$.

For any abelian subscheme $\cB$ of $\cA_X \rightarrow B$ whose relative dimension we denote by $g_{\cB}$, we obtain the following diagram
\begin{equation}\label{EqGeomMeaningQuotShimuraCrit}
\xymatrix{
\cA_X \ar[r]^-{p_{\cB}} \ar[d]_{\pi|_{\cA_X}} & \cA_X/\cB \ar[d]_{\pi_{/\cB}} \ar[r]^-{\iota_{/\cB}} \pullbackcorner & \mathfrak{A}_{g_X - g_{\cB}} \ar[d] \\
B \ar[r]^-{\mathrm{id}_B} & B \ar[r]^-{\iota_{/\cB,G}} & \A_{g_X-g_{\cB}},
}
\end{equation}
where $p_{\cB}$ is taking the quotient abelian scheme, and the right box is the modular map.

\begin{thm}\label{ThmCriterionXDegenerateGeom}
Assume either $t \le 0$, or $t = 1$ and $\cA_X = \mathfrak{A}_g|_B$. Then $X = X^{\mathrm{deg}}(t)$ if and only if the following condition holds: There exists an abelian subscheme $\cB$ of $\cA_X \rightarrow B$ (whose relative dimension we denote by $g_{\cB}$) such that for the map $\iota_{/\cB} \circ p_{\cB}$ constructed above, we have $\dim (\iota_{/\cB} \circ p_{\cB})(X) < \dim X -  g_{\cB} + t$ and that $\iota_{/\cB} \circ p_{\cB}$ is not generically finite.
\end{thm}

Note that when $t \le 0$, the condition \textit{$\iota_{/\cB} \circ p_{\cB}$ is not generically finite} is redundant. This is because in this case, $\dim (\iota_{/\cB} \circ p_{\cB})(X) < \dim X -  g_{\cB} + t$ implies $\dim (\iota_{/\cB}\circ p_{\cB})(X) < \dim X$, and hence $\iota_{/\cB}\circ p_{\cB}$ is not generically finite.

\subsection{Auxiliary proposition}
We start by proving the following auxiliary proposition.
\begin{prop}\label{PropCriterionXDegenerate}
We have $X = X^{\mathrm{deg}}(t)$ if and only if there exist
\begin{itemize}
\item a special subvariety $M_*$ of $\mathfrak{A}_g$, associated with $(Q_*,\cY^+_*)$, which contains $X$;
\item  a non-trivial connected normal subgroup $N_*$ of $Q_*$ whose reductive part is semi-simple;
\end{itemize}
such that the following condition holds: For the operation of taking quotient mixed Shimura datum $\tilde{p}_{N_*} \colon (Q_*,\cY_*^+) \rightarrow (Q_*,\cY_*^+)/N_*$ discussed in $\mathsection$\ref{SubsectionGeomQuotByNormalSubgp} and the induced morphism on the corresponding mixed Shimura varieties of Kuga type
\begin{equation}\label{EqQuotXdegCriterionPrem}
\xymatrix{
\cY_*^+ \ar[r]^-{\tilde{p}_{N_*}} \ar[d]_{\mathbf{u}|_{\cY_*^+}} & \cY_*^{\prime+} \ar[d] \\
M_* \ar[r]^-{p_{N_*}} & M_*' ,
}
\end{equation}
we have $\dim X - \dim p_{N_*}(X) > g_{N_*} - t$, where $g_{N_*} = \frac{1}{2} \dim (V_{2g} \cap N_*)$.
\end{prop}

\begin{proof}
We prove $\Leftarrow$.  Let $E_{g_{N_*} - t} = \{ x \in X : \dim_x (p_{N_*}|_{X})^{-1}(p_{N_*}(x)) > g_{N_*} - t\}$ as defined in \eqref{EqDefEh}. Then $X = E_{g_{N_*} - t}$ as $\dim X - \dim p_{N_*}(X) > g_{N_*} - t$. 
Note that \eqref{EqQuotXdegZarClosed} and \eqref{EqQuotXdegCriterionPrem} have the same shape. 
Hence $X = X^{\mathrm{deg}}(t)$ by (ii) of Proposition~\ref{PropAuxForZarClosedness}.

Now let us prove $\Rightarrow$. By the proof of Theorem~\ref{ThmDegeneracyLocusZarClosed} ($\mathsection$\ref{SubsectionEndOfProofOfZarClosedOfDegLocus}), we have that $X^{\mathrm{deg}}(t)$ is a finite union of some Zariski closed subsets $E_{g_{N_*} - t}$'s of $X$. Now $X = E_{g_{N_*} - t}$ for some $E_{g_{N_*} - t}$ since $X= X^{\mathrm{deg}}(t)$.

Recall the definition of $E_{g_{N_*} - t}$ as in \eqref{EqDefEh}: $E_{g_{N_*} - t} = \{ x \in X : \dim_x (p_{N_*}|_{X \cap M_*})^{-1}(p_{N_*}(x)) > g_{N_*} - t\}$ for some mixed Shimura subdatum of Kuga type $(Q_*,\cY_*^+)$ of $(P_{2g,\mathrm{a}},\cX_{2g,\mathrm{a}}^+)$ and a non-trivial connected normal subgroup $N_*$ of $Q_*$ whose reductive part is semi-simple, with $g_{N_*} = \frac{1}{2}\dim(V_{2g} \cap N_*)$.

By definition of $E_{g_{N_*} - t}$, it is contained in $X \cap M_*$. Hence $X = E_{g_{N_*} - t} \subseteq M_*$.

Now that each fiber of $p_{N_*}|_X$ has dimension $> g_{N_*}-t$, we have
 $\dim X - \dim p_{N_*}(X) > g_{N_*} - t$.
\end{proof}

\subsection{Theorem~\ref{ThmCriterionXDegenerateGeom} in terms of mixed Shimura variety}
In this subsection, we prove Theorem~\ref{ThmCriterionXDegenerateGeom} in terms of mixed Shimura variety. Then we translate it into the desired geometric description in the next subsection.

Let $M$ be the smallest special subvariety of $\mathfrak{A}_g$ which contains $X$. Assume that $M$ is associated with the mixed Shimura subdatum of Kuga type $(Q,\cY^+)$ of $(P_{2g,\mathrm{a}},\cX_{2g,\mathrm{a}}^+)$.

\begin{prop}\label{ThmCriterionXDegenerate}
Assume either $t \le 0$, or $t = 1$ and $M = \mathfrak{A}_g|_{\pi(M)}$. Then $X = X^{\mathrm{deg}}(t)$ if and only if there exists a non-trivial connected normal subgroup $N$ of $Q$ whose reductive part is semi-simple, such that the following condition holds: For the operation of taking quotient Shimura datum $\widetilde{p}_N \colon (Q,\cY^+) \rightarrow (Q,\cY^+)/N =: (Q', \cY^{\prime+})$ discussed in $\mathsection$\ref{SubsectionGeomQuotByNormalSubgp} and the induced morphism on the corresponding mixed Shimura varieties of Kuga type
\begin{equation}\label{EqQuotXdegCriterion}
\xymatrix{
\cY^+ \ar[r]^-{\widetilde{p}_N} \ar[d]_{\mathbf{u}|_{\cY^+}} & \cY^{\prime+} \ar[d] \\
M \ar[r]^-{p_N} & M' ,
}
\end{equation}
we have $\dim X - \dim p_N(X) > g_N - t$, where $g_N = \frac{1}{2} \dim (V_{2g} \cap N)$.
\end{prop}


\begin{proof}[Proof of Proposition~\ref{ThmCriterionXDegenerate}]
We use Proposition~\ref{PropCriterionXDegenerate}. First $\Leftarrow$ of Proposition~\ref{ThmCriterionXDegenerate} follows directly from $\Leftarrow$ of Proposition~\ref{PropCriterionXDegenerate}.

Let us prove $\Rightarrow$ of Proposition~\ref{ThmCriterionXDegenerate}. Let $M_*$, $(Q_*,\cY^+_*)$, $N_*$ and $g_{N_*}$ be as in $\Rightarrow $ of Proposition~\ref{PropCriterionXDegenerate}.

Set $N$ to be the identity component of $Q \cap N_*$. Since $X \subseteq M_*$ and $M$ is the smallest special subvariety of $\mathfrak{A}_g$ which contains $X$, we have $M \subseteq M_*$. Hence we may assume $(Q,\cY^+) \subseteq (Q_*,\cY_*^+)$. Let $N = Q \cap N_*$, then $N \lhd Q$. Replacing $N$ by $N \cap \tilde{\pi}^{-1}(\tilde{\pi}(N)^{\mathrm{der}})$, we may and do assume that the reductive part of $N$ is semi-simple; here $\tilde{\pi} \colon P_{2g,\mathrm{a}} \rightarrow \mathrm{GSp}_{2g}$.

Let us temporarily assume
\begin{equation}\label{EqCritInequalityAfterQuot}
\dim p_{N_*}(X) = \dim p_N(X)
\end{equation}
and finish the proof by showing that this $N$ can be taken as the desired connected normal subgroup of $Q$. We need to prove:
\begin{enumerate}
\item[(i)] $\dim X - \dim p_N(X) > g_N - t$ for $g_N = \frac{1}{2} \dim (V_{2g} \cap N)$;
\item[(ii)] $N$ is a non-trivial group.
\end{enumerate}

For (i): We have $g_N \le g_{N_*}$ since $N < N_*$. So $\dim X - \dim p_N(X) = \dim X - \dim p_{N_*}(X) > g_{N_*} - t \ge g_N - t$.

For (ii): Suppose $N$ is trivial. Then $p_N = \mathrm{id}_M$ and thus $\dim p_N(X) = \dim X$. So $\dim p_{N_*}(X) = \dim X$. Hence $p_{N_*}|_X$ is generically finite. So $X= E_{g_{N_*}-t}$ implies that $g_{N_*} - t < 0$. But then $g_{N_*} = 0$ and $t = 1$ as $t \le 1$. Hence $N_*$ is reductive. It can be viewed as a subgroup of $\Sp_{2g}$ via $\tilde{\pi} \colon P_{2g,\mathrm{a}} \rightarrow \GSp_{2g}$. Now $N_* \lhd Q_*$ implies that the subgroup $N_*$ of $\Sp_{2g}$ acts trivially on $V_{2g} \cap Q_*$. By our hypothesis on $X$ (note that $t = 1$ now) and Proposition~\ref{PropSpecialSubvarUnivAb}, we have $M = \mathfrak{A}_g|_{\pi(M)}$. Hence $V_{2g} \cap Q = V_{2g}$. Thus $V_{2g} \cap Q_* = V_{2g}$ since $Q \subseteq Q_*$. But the only connected subgroup of $\Sp_{2g}$ acting trivially on $V_{2g}$ is $1$. Hence $N_*$ is trivial, contradicting to the choice of $N_*$. Thus $N$ is non-trivial. Hence this $N$ can be taken as the desired connected normal subgroup of $Q$.

Now it remains to prove \eqref{EqCritInequalityAfterQuot}. Each fiber of $\tilde{p}_N$ is of the form $N(\R)^+\tilde{y}$ for some $\tilde{y} \in \cY^+$, and each fiber of $\tilde{p}_{N_*}$ is of the form $N_*(\R)^+\tilde{y}_*$ for some $\tilde{y}_* \in \cY^+_*$. As $\cY^+ \subseteq \cY^+_*$, we can take $\tilde{y}_* = \tilde{y} \in \cY^+$. By definition of $N$, we have $N(\R)^+\tilde{y} = Q(\R)^+ \tilde{y} \cap N_*(\R)^+ \tilde{y} = \cY^+ \cap N_*(\R)^+\tilde{y}$.

Denote by $\tilde{X}$ a complex analytic irreducible component of $\bu^{-1}(X)$ which is contained in $\cY^+$. Then by the previous paragraph, each fiber of $\tilde{p}_N|_{\tilde{X}}$ is of the form $N(\R)^+\tilde{y} \cap \tilde{X}$, and each fiber of $\tilde{p}_{N_*}|_{\tilde{X}}$ is of the form $N_*(\R)^+\tilde{y} \cap \tilde{X}$.  The last sentence of last paragraph says $N(\R)^+\tilde{y} \cap \tilde{X} = \cY^+ \cap N_*(\R)^+\tilde{y} \cap \tilde{X}$, which furthermore equals $N_*(\R)^+\tilde{y} \cap \tilde{X}$ since $\tilde{X} \subseteq \cY^+$. Thus the fibers of $\tilde{p}_N|_{\tilde{X}}$ and $\tilde{p}_{N_*}|_{\tilde{X}}$ have the same dimensions. Hence the fibers of $p_N|_X$ and $p_{N_*}|_X$ have the same dimensions. Thus \eqref{EqCritInequalityAfterQuot} holds. Now we are done.
\end{proof}

\subsection{Proof of Theorem~\ref{ThmCriterionXDegenerateGeom}}
Let $M$ be the smallest special subvariety of $\mathfrak{A}_g$ which contains $X$. Assume that $M$ is associated with the mixed Shimura subdatum of Kuga type $(Q,\cY^+)$ of $(P_{2g,\mathrm{a}},\cX_{2g,\mathrm{a}}^+)$. Denote by $V_Q = V_{2g} \cap Q$ and $G_Q = \tilde{\pi}(Q)$ the reductive part of $Q$. Denote by $g_Q = \frac{1}{2}\dim V_Q$.

Special subvarieties of $\mathfrak{A}_g$ are described in Proposition~\ref{PropSpecialSubvarUnivAb}. Hence the smallest special subvariety $M$ of $\mathfrak{A}_g$ which contains $X$ can be described as follows:
\begin{itemize}
\item Let $M_G$ be the smallest special subvariety of $\A_g$ which contains $\pi(X)$;
\item Let $M$ be the translate of an abelian subscheme of $\mathfrak{A}_g|_{M_G} \rightarrow M_G$ by a torsion section which contains $X$, minimal for this property.
\end{itemize}
Then $M \rightarrow M_G$ itself is an abelian scheme of relative dimension $g_Q$.

\noindent \framebox{$\Rightarrow$ of Theorem~\ref{ThmCriterionXDegenerateGeom}} Apply $\Rightarrow$ of Proposition~\ref{ThmCriterionXDegenerate}. Then we obtain a non-trivial connected normal subgroup $N$ of $Q$ whose reductive part is semi-simple, such that for the quotient Shimura morphism $p_N \colon M \rightarrow M'$, we have ($g_N = \frac{1}{2}\dim(V_{2g}\cap N)$)
\begin{equation}\label{EqCritToGeom}
\dim X - \dim p_N(X) > g_N - t.
\end{equation}
Recall the geometric meaning of $p_N$ \eqref{EqGeomMeaningQuotShimura}
\[
\xymatrix{
M \ar[r]|-{p_0} \ar[d]_{\pi|_M} \ar@/^1pc/[rr]|-{p_N} & M_0 := \Gamma_0 \backslash \cY_0^+ \ar[d]_{\pi_0} \ar[r]|-{p'} \pullbackcorner & M' \ar[d]_{\pi'} \ar[r]^-{i} \pullbackcorner & \mathfrak{A}_{g_Q-g_N} \ar[d] \\
M_G \ar[r]^-{\mathrm{id}_{M_G}} & M_G \ar[r]^-{p_{G_N}} & M'_G \ar[r]^-{i_G} & \A_{g_Q-g_N} .
}
\]
The morphism $p_0$ is taking the quotient of $M \rightarrow M_G$ by the abelian subscheme $\ker(p_0)^{\circ}$ (which has relative dimension $g_N$). For each $b' \in M_G'$, the abelian scheme $M_0|_{p_{G_N}^{-1}(b')} \rightarrow p_{G_N}^{-1}(b')$ is isotrivial.

Take $\cB = \ker(p_0)^{\circ} \cap \cA_X$. Then $\cB$ is an abelian subscheme of $\cA_X \rightarrow B$ of relative dimension $g_N$, namely $g_{\cB} = g_N$. Now we can construct the maps in \eqref{EqGeomMeaningQuotShimuraCrit}. We have $\cA_X = M|_B$, $p_{\cB} = p_0|_{\cA_X}$, $\pi_{/\cB} = \pi_0|_{M_0|_B}$, $\iota_{/\cB} = (i \circ p')|_{M_0|_B}$ and $\iota_{/\cB,G} = (i_G \circ p_{G_N})|_B$. Thus
\[
\dim (\iota_{/\cB} \circ p_{\cB})(X) \le \dim p_N(X) < \dim X - g_N + t = \dim X - g_{\cB} +t
\]
where the first inequality follows from $X \subseteq M|_B$ and the second inequality follows from \eqref{EqCritToGeom}. Thus it suffices to prove that $\iota_{/\cB} \circ p_{\cB}$ is not generically finite.

Suppose $\iota_{/\cB} \circ p_{\cB}$ is generically finite. Then $g_N = 0$, $p_{\cB} = \mathrm{id}_{\cA_X}$ and $\iota_{/\cB}$ is generically finite. Hence $\dim (\iota_{/\cB} \circ p_{\cB})(X) = \dim X$. But then $\dim X < \dim X - g_{\cB} + t = \dim X + t$. When $t \le 0$ this cannot hold. Hence $t = 1$ and $\cA_X = \mathfrak{A}_g|_B$ by our hypothesis. Hence $M = \mathfrak{A}_g|_{M_G}$.

Since $g_N = 0$, we have $p_0 = \mathrm{id}_M$ and $M_0 = M = \mathfrak{A}_g|_{M_G}$. For each $b' \in M_G'$, the abelian scheme $\mathfrak{A}_g|_{p_{G_N}^{-1}(b')} \rightarrow p_{G_N}^{-1}(b')$ is isotrivial. So $\dim p_{G_N}^{-1}(b') = 0$; see the end of $\mathsection$\ref{SectionNotation}. Hence $p_{G_N}$ is generically finite, and so is $p'$. Thus $\dim M = \dim M'$. This contradicts our choice of $N$ (non-trivial connected).



\noindent \framebox{$\Leftarrow$ of Theorem~\ref{ThmCriterionXDegenerateGeom}}
Before moving on, let me point out that if $t \le 0$, then this implication follows rather easily from Proposition~\ref{PropDegLocus} and Theorem~\ref{ThmDegeneracyLocusZarClosed} because we can translate it into studying the generic rank of the Betti map. However the argument below works also for $t =1$.

Hodge theory says that: (1) every abelian subscheme of $\cA_X \rightarrow B$ is the intersection of an abelian subscheme of $M \rightarrow M_G$ with $\mathfrak{A}_g|_B$ (in particular $\cA_X = M \cap \mathfrak{A}_g|_B = M|_B$ and $g_X = g_Q$); (2) the abelian subschemes of $M \rightarrow M_G$ are in 1-to-1 correspondence to $G_Q$-submodules of $V_Q$. See \cite[4.4.1-4.4.3]{DeligneHodgeII}.

Assume that $\cB = \mathfrak{B} \cap \mathfrak{A}_g|_B$ where $\mathfrak{B}$ is an abelian subscheme of $M \rightarrow M_G$, and that $\mathfrak{B}$ corresponds to the $G_Q$-submodule $V_N$ of $V_Q$. Let $g_N = \frac{1}{2}\dim V_N$, then $\mathfrak{B} \rightarrow M_G$ has relative dimension $g_N$. Hence $g_{\cB} = g_N$.

Taking the quotient, we get
\[
\xymatrix{
M \ar[r]^-{p_{\mathfrak{B}}} \ar[d]_{\pi|_M}  & M/\mathfrak{B} \ar[d]^{\pi_{/\mathfrak{B}}}  \\
M_G \ar[r]^-{\mathrm{id}_{M_G}} & M_G  .
}
\]
In particular $\pi_{/\mathfrak{B}}$ is an abelian scheme of relative dimension $g_Q-g_N$. It induces a modular map
\[
\xymatrix{
M/\mathfrak{B} \ar[r]^-{\iota_{/\mathfrak{B}}} \ar[d]_{\pi_{/\mathfrak{B}}} \pullbackcorner & \mathfrak{A}_{g_Q-g_N} \ar[d] \\
M_G \ar[r]^-{\iota_{/\mathfrak{B},G}} & \A_{g_Q-g_N} .
}
\]
By definition of $\iota_{/\mathfrak{B},G}$, $(M/\mathfrak{B})|_{\iota_{/\mathfrak{B},G}^{-1}(a)}$ is an isotrivial abelian scheme for each $a \in \A_{g_Q-g_N}$.

Modular interpretation of Shimura varieties implies that $\iota_{/\mathfrak{B},G}$ is a Shimura morphism, namely $\iota_{/\mathfrak{B},G}$ is induced by some $\tilde{\iota}_{/\mathfrak{B},G} \colon (G_Q,\tilde{\pi}(\cY^+)) \rightarrow (\GSp_{2(g_Q-g_N)}, \mathfrak{H}_{g_Q-g_N}^+)$. Denote by $H$ the kernel of $\tilde{\iota}_{/\mathfrak{B},G}$ on the underlying groups, then $H \lhd G_Q$. Replace $H$ by $H^{\mathrm{der}}$, then $H \lhd G_Q^{\mathrm{der}}$.

Each fiber of $\tilde{\iota}_{/\mathfrak{B},G}$ on the underlying spaces is of the form $H(\R)^+\tilde{y}_G$ for some $\tilde{y}_G \in \tilde{\pi}(\cY^+)$. Hence each fiber of $\iota_{/\cB,G}$ is of the form $\bu_G(H(\R)^+\tilde{y}_G)$ for some $\tilde{y}_G \in \tilde{\pi}(\cY^+)$, where $\bu_G \colon \mathfrak{H}_g^+ \rightarrow \A_g$ is the uniformization. Thus the connected algebraic monodromy group of $\iota_{/\mathfrak{B},G}^{-1}(a)$ is $H$ for each $a \in \A_{g_Q-g_N}$. By Deligne's Theorem of the Fixed Part \cite[Corollaire~4.1.2]{DeligneHodgeII}, $H$ acts trivially on $V_Q/V_N \cong V_{2(g_Q-g_N)}$ because $(M/\mathfrak{B})|_{\iota_{/\mathfrak{B},G}^{-1}(a)}$ is isotrivial.

The zero section of the abelian scheme $M \rightarrow M_G$ gives rise to a Levi decomposition $Q = V_Q \rtimes G_Q$. 
Let $N = V_N \rtimes H$. Then $N$ is a connected normal subgroup of $Q$ whose reductive part is semi-simple. Moreover the quotient morphism $p_N$ is precisely $\iota_{/\cB} \circ p_{\mathfrak{B}}$, namely the geometric interpretation \eqref{EqGeomMeaningQuotShimura} for this $N$ becomes
\begin{equation}\label{EqGeomMeaningQuotShimuraCritPf}
\xymatrix{
M \ar[r]|-{p_{\mathfrak{B}}} \ar[d]_{\pi|_M} \ar@/^1pc/[rr]|-{p_N} & M/\mathfrak{B} \ar[d]_{\pi_{/\cB}} \ar[r]_-{\iota_{/\mathfrak{B}}} \pullbackcorner & \iota(M/\mathfrak{B}) \ar[d]^{\pi'} \\
M_G \ar[r]^-{\mathrm{id}_{M_G}} & M_G \ar[r]^-{\iota_{/\mathfrak{B},G}} & \iota_{/\mathfrak{B},G}(M_G) .
}
\end{equation}
Note that \eqref{EqGeomMeaningQuotShimuraCrit} is precisely \eqref{EqGeomMeaningQuotShimuraCritPf} restricted to $B \subseteq M_G$.

If $N = 1$, then $p_N = \mathrm{id}_M$. This contradicts $\iota_{/\cB} \circ p_{\cB}$ being not generically finite. Hence $N$ is non-trivial.

Note that $X \subseteq M|_B$. Thus the inequality in $\Leftarrow$ of Theorem~\ref{ThmCriterionXDegenerateGeom} becomes $\dim p_N(X) < \dim X - g_N + t$. Now it suffices to apply $\Leftarrow$ of Proposition~\ref{ThmCriterionXDegenerate} to this $N$.

\section{Summary on the generic rank of the Betti map}\label{SectionAppConjACZ}
In this section, we go back to the general setting and prove the main results concerning the Betti rank. First of all let us recall the setting up.

Let $S$ be an irreducible quasi-projective variety over $\C$, and let $\pi_S \colon \cA \rightarrow S$ be an abelian scheme of relative dimension $g\ge 1$. Recall the modular map \eqref{EqEmbedAbSchIntoUnivAbVar}
\[
\xymatrix{
\cA \ar[r]^{\iota} \ar[d]_{\pi_S} \pullbackcorner & \mathfrak{A}_g \ar[d]^{\pi} \\
S \ar[r]^{\iota_S} & \A_g .
}
\]
Let $b_{\Delta} \colon \cA_{\Delta} \rightarrow \mathbb{T}^{2g}$ be as defined in \eqref{EqBettiMapActualScheme} where $\Delta$ is some open subset of $S^{\mathrm{an}}$.

Let $X$ be a closed irreducible subvariety of $\cA$ dominant to $S$. Assume that $\iota|_X \colon X \rightarrow \iota(X)$ has relative dimension $r \ge 0$. Then for each $x \in X(\C)$, we have
\begin{equation}\label{EqRelDim}
\dim_x \iota|_X^{-1}(\iota(x)) \ge r.
\end{equation}

\subsection{Proof of Theorem~\ref{ThmZarClosedXdegIntro}} Fix $t \in \Z$. 
We proved that $\iota(X)^{\mathrm{deg}}(t)$ is Zariski closed in $\iota(X)$ for each $t$ in Theorem~\ref{ThmDegeneracyLocusZarClosed}. Thus Theorem~\ref{ThmZarClosedXdegIntro} follows immediately from:
\begin{lemma}\label{LemmaDegLocusRelUnderMod}
$X^\mathrm{deg}(t) = 
\begin{cases}
 \iota^{-1}\left( \iota(X)^{\mathrm{deg}}(t+r) \right) & \text{if}\quad t+r \le 0 \text{ or } r=0 \\
 X & \text{otherwise} 
\end{cases}$.
\end{lemma}
\begin{proof}
First let us make the following observation: For any $Y \subseteq \cA$ irreducible, we have that $\langle Y \rangle_{\mathrm{sg}}$ is an irreducible component of $\iota^{-1}(\langle \iota(Y) \rangle_{\mathrm{sg}})$. So
\begin{equation}\label{EqSGMod}
\dim \langle \iota(Y) \rangle_{\mathrm{sg}}  - \dim \pi(\iota(Y)) = \dim \langle \iota(Y) \rangle_{\mathrm{sg}} - \dim \iota_S(\pi_S(Y)) = \dim \langle Y \rangle_{\mathrm{sg}} - \dim \pi_S(Y)
\end{equation}

\noindent\boxed{\text{Case } t+r \le 0\text{ or }r=0}
For $\subseteq$: Take $Y \subseteq X^{\mathrm{deg}}(t)$ irreducible such that $\dim Y > 0$ and 
\[
\dim \langle Y \rangle_{\mathrm{sg}} - \dim \pi_S(Y) < \dim Y + t.
\]
As $\iota|_X$ has relative dimension $r$, we have $\dim \iota(Y) \ge \dim Y - r$. 
So \eqref{EqSGMod} implies
\[
\dim \langle \iota(Y) \rangle_{\mathrm{sg}}  - \dim \pi(\iota(Y)) = \dim \langle Y \rangle_{\mathrm{sg}} - \dim \pi_S(Y) < \dim Y + t \le \dim \iota(Y) + (t+r).
\]
If $t+r \le 0$, then the inequality above further implies $0 < \dim \iota(Y)$. If $r=0$, then $\iota$ is generically finite and hence $\dim \iota(Y) = \dim Y >0$. Hence in either case, we have $\iota(Y) \subseteq \iota(X)^{\mathrm{deg}}(t+r)$. This proves $\subseteq$.


For $\supseteq$: Conversely take $Y' \subseteq \iota(X)^{\mathrm{deg}}(t+r)$ irreducible such that $\dim Y' >0$ and
\[
\dim \langle Y' \rangle_{\mathrm{sg}} - \dim \pi(Y') < \dim Y' + t+r.
\]
Let $Y$ be an irreducible component of $\iota|_X^{-1}(Y')$. Then $Y' = \iota(Y)$. So $\dim Y \ge \dim Y' + r >0$ by \eqref{EqRelDim}. By \eqref{EqSGMod} we have then
\[
 \dim \langle Y \rangle_{\mathrm{sg}} - \dim \pi_S(Y) = \dim \langle Y' \rangle_{\mathrm{sg}} - \dim \pi(Y') < \dim Y' + t+r = \dim Y + t.
\]
Hence $Y \subseteq X^{\mathrm{deg}}(t)$. This proves $\supseteq$.

\noindent\boxed{\text{Case }t+r > 0\text{ and }r > 0} Let $x \in X(\C)$. Let $Y$ be the irreducible component of $\iota|_X^{-1}(\iota(x))$ which contains $x$. Then $\dim Y \ge r > 0$.

For $B:= \pi_S(\iota^{-1}(\iota(x))) = \iota_S^{-1}(\pi(\iota(x)))$, we have that $\iota^{-1}(\iota(x))$ is a constant section of $\cA|_B \rightarrow B$. Thus $\dim \pi_S(Y) = \dim Y \ge r$, and $\langle Y \rangle_{\mathrm{sg}} = Y$. Hence
\[
\dim \langle Y \rangle_{\mathrm{sg}} - \dim \pi_S(Y) = \dim Y - \dim Y = 0 < r -r+1 \le \dim Y + (-r+1).
\]
So $Y \subseteq X^{\mathrm{deg}}(-r+1)$. Thus $X \subseteq X^{\mathrm{deg}}(-r+1)$.

But $t \ge -r+1$ since $t+r>0$. So $X^{\mathrm{deg}}(-r+1) \subseteq X^{\mathrm{deg}}(t)$ by definition. So we are done.
\end{proof}

\subsection{Proof of Theorem~\ref{ThmDegLocusIntro}}
Fix $x \in X^{\mathrm{sm}}(\C)$. Then
\[
\mathrm{rank}_{\R} (\mathrm{d}b_{\Delta}|_X)_x < 2l \Leftrightarrow \mathrm{rank}_{\R} (\mathrm{d}b_{\Delta}|_{\iota(X)})_{\iota(x)} < 2l.
\]
Apply Proposition~\ref{PropDegLocus} to $\iota(X)$, then we have
\[
\mathrm{rank}_{\R} (\mathrm{d}b_{\Delta}|_{\iota(X)}) < 2l \Leftrightarrow \iota(X)^{\mathrm{deg}}(l- \dim \iota(X))\text{ is Zariski dense in }\iota(X).
\]
Since $l \le \dim \iota(X)$, we have
\[
\iota(x) \in \iota(X)^{\mathrm{deg}}(l- \dim \iota(X)) \Leftrightarrow x \in X^{\mathrm{deg}}(l - \dim X).
\]
by Lemma~\ref{LemmaDegLocusRelUnderMod}, because $\dim X = \dim \iota(X) + r$. We are done by the three equivalences above.

\subsection{Proof of Theorem~\ref{ThmCriterionDegIntro}}

The implication $\Leftarrow$ is clear: The generic rank of the Betti map on $(\iota_{/\cB} \circ p_{\cB})(X)$ has the trivial upper bound $2 \dim (\iota_{/\cB} \circ p_{\cB})(X)$, thus  $\mathrm{rank}_{\R} (\mathrm{d}b_{\Delta}|_X) \le 2g_{\cB} + 2 \dim (\iota_{/\cB} \circ p_{\cB})(X) < 2l$.\footnote{In other words, $\Leftarrow$ of Theorem~\ref{ThmCriterionXDegenerateGeom} is clearly true when $t \le 0$ if one translates the condition $X = X^{\mathrm{deg}}(t)$ into studying the generic rank of the Betti map. But for $t = 1$, we still need to go into our original proof.}

Let us prove $\Rightarrow$. If $l > \dim \iota(X)$, then $\mathrm{rank}_{\R} (\mathrm{d}b_{\Delta}|_X) < 2l$ always holds. Take $\cB$ to be the zero section of $\cA_X \rightarrow S$. Then $g_{\cB} =0$. Thus
\[
\dim (\iota_{/\cB} \circ p_{\cB})(X) = \dim \iota_{/\cB}(X) \le \dim \iota(X) < l = l-g_{\cB}.
\]
It suffices then to consider the case $l \le \dim \iota(X)$. 
By Theorem~\ref{ThmDegLocusIntro} and Theorem~\ref{ThmZarClosedXdegIntro}, we have
\[
\mathrm{rank}_{\R} (\mathrm{d}b_{\Delta}|_X) < 2l \Leftrightarrow \iota(X) = \iota(X)^{\mathrm{deg}}(l- \dim \iota(X)).
\]
Denote by $B = \iota_S(S) = \pi(\iota(X))$. Apply Theorem~\ref{ThmCriterionXDegenerateGeom} (and the discussion below) to $\iota(X)$ and $t = l - \dim \iota(X) \le 0$. We thus obtain an abelian subscheme $\cB'$ of $\cA_{\iota(X)} \rightarrow B$ (of relative dimension $g_{\cB'}$), such that $\dim (\iota_{/\cB'}\circ p_{\cB'})(\iota(X)) < \dim \iota(X) - g_{\cB'}$ for 
\[
\xymatrix{
\cA_{\iota(X)} \ar[r]^-{p_{\cB'}} \ar[d]_{\pi|_{\cA_{\iota(X)}}} & \cA_{\iota(X)}/\cB' \ar[d]_{\pi_{/\cB'}} \ar[r]^-{\iota_{/\cB'}} \pullbackcorner & \mathfrak{A}_{g_{\iota(X)} - g_{\cB'}} \ar[d] \\
B \ar[r]^-{\mathrm{id}_B} & B \ar[r]^-{\iota_{/\cB,G}} & \A_{g_{\iota(X)}-g_{\cB'}} .
}
\]
Observe $\cA_X=\iota^{-1}(\cA_{\iota(X)})$; so $g_X = g_{\iota(X)}$. Set $\cB = \iota^{-1}(\cB')$. Then $\cB$ is an abelian subscheme of $\cA_X \rightarrow S$; its relative dimension $g_{\cB}$ equals $g_{\cB'}$. The following diagram commutes
\[
\xymatrix{
\cA_X \ar[r]^-{p_{\cB}} \ar[d]_{\iota} & \cA_X/\cB \ar[d]_{\bar{\iota}} \ar[r]^-{\iota_{/\cB}}  & \mathfrak{A}_{g_X - g_{\cB}} \ar[d]^{=} \\
\cA_{\iota(X)} \ar[r]^-{p_{\cB'}}  & \cA_{\iota(X)}/\cB'  \ar[r]^-{\iota_{/\cB'}}  & \mathfrak{A}_{g_{\iota(X)} - g_{\cB'}} .
}
\]
Thus $(\iota_{/\cB} \circ p_{\cB})(X) = (\iota_{/\cB'}\circ p_{\cB'})(\iota(X))$. 
So $\dim (\iota_{/\cB} \circ p_{\cB})(X) = \dim (\iota_{/\cB'}\circ p_{\cB'})(\iota(X)) < \dim \iota(X) - g_{\cB'} \le \dim X - g_{\cB}$.

\subsection{A question of Andr\'{e}-Corvaja-Zannier}
In this subsection we apply the previous results to study a conjecture of Andr\'{e}-Corvaja-Zannier.

Let $\xi$ be a section of $\cA \rightarrow S$. Denote by $b_{\Delta}|_{\xi}$ the composite $b_{\Delta} \circ \xi \colon \Delta \rightarrow \cA_{\Delta} \rightarrow \mathbb{T}^{2g}$.

\begin{thm}\label{ThmACZ}
Assume that $\Z \xi$ is Zariski dense in $\cA$. Assume furthermore:
\begin{enumerate}
\item[(i)] Either $\cA/S$ is geometrically simple;
\item[(ii)] Or each Hodge generic curve $C \subseteq \iota_S(S)$ satisfies the following property: $\iota(\cA) \times_{\iota_S(S)} C = \pi^{-1}(C) \rightarrow C$ has no fixed part over any finite covering of $C$.\footnote{We say that a curve $C \subseteq \iota_S(S)$ is Hodge generic if the generic Mumford-Tate group of $C$ coincides with the generic Mumford-Tate group of $\iota_S(S)$.}
\end{enumerate}
Then $\max_{s \in \Delta} \mathrm{rank}_{\R} (\mathrm{d}b_{\Delta}|_{\xi})_s = 2 \min(\dim \iota(\xi(S)), g)$.
\end{thm}
\begin{rmk}\label{RmkACZ}
If $\cA/S$ is isotrivial, then hypothesis (ii) holds automatically since $\iota_S(S)$ is a point. In general hypothesis (ii) can be checked in the following way. Let $G_B$ be the connected algebraic monodromy group of $\iota_S(S)$. Then hypothesis (ii) is equivalent to: for any non-trivial connected normal subgroup $H$ of $G_B$, the only element of $V_{2g}$ stable under $H$ is $0$. It holds for example when $\cA \rightarrow S$ has no fixed part and $G_B$ is a simple group.
\end{rmk}

\begin{proof}
Denote by $X = \xi(S)$ and $d = \dim \iota(X)$. Then $\max_{s \in \Delta} \mathrm{rank}_{\R} (\mathrm{d}b_{\Delta}|_{\xi})_s = \mathrm{rank}_{\R} (\mathrm{d}b_{\Delta}|_X)$.

Since $\Z\xi$ is Zariski dense in $\cA$, we have $\cA_X = \cA$ where $\cA_X$ is defined above Theorem~\ref{ThmCriterionDegIntro}.

Assume $\mathrm{rank}_{\R} (\mathrm{d}b_{\Delta}|_X) < 2 \min (d,g)$. Applying Theorem~\ref{ThmCriterionDegIntro}.(1) to $l = \min (d,g)$, we get an abelian subscheme $\cB$ of $\cA \rightarrow S$ (whose relative dimension we denote by $g_{\cB}$) such that 
\begin{equation}\label{EqACZ}
\dim (\iota_{/\cB}\circ p_{\cB})(X) <  \min (d,g) - g_{\cB}
\end{equation}
for
\begin{equation}\label{EqDiagImporACZ}
\xymatrix{
\cA \ar[r]^-{p_{\cB}} \ar[d]_{\pi_S} & \cA/\cB \ar[r]^-{\iota_{/\cB}} \ar[d] \pullbackcorner & \mathfrak{A}_{g - g_{\cB}} \ar[d] \\
S \ar[r]^{\mathrm{id}_S} & S \ar[r]^-{\iota_{/\cB,S}} & \A_{g - g_{\cB}}.
}
\end{equation}

\noindent\framebox{Case (i)} By assumption $\cB$ is either the whole $\cA$ or the zero section of $\cA \rightarrow S$. In the former case, $g_{\cB} = g$. But then \eqref{EqACZ} cannot hold. In the latter case, $g_{\cB} = 0$, $p_{\cB} = \mathrm{id}_{\cA}$ and $\iota_{/\cB} = \iota$. So $\dim (\iota_{/\cB}\circ p_{\cB})(X) = \dim \iota(X) = d$. But then \eqref{EqACZ} cannot hold.

\noindent\framebox{Case (ii)} If $\cB = \cA$, then $g = g_{\cB}$, and hence \eqref{EqACZ} cannot hold. So $\cA/\cB \rightarrow S$ is a non-trivial abelian scheme. Applying the modular map $\iota \colon \cA \rightarrow \mathfrak{A}_g$ (and $\iota_S \colon S \rightarrow \mathbb{A}_g$) to \eqref{EqDiagImporACZ}, we obtain
\[
\xymatrix{
\iota(\cA) \ar[r]^-{p'_{\cB}} \ar[d]_{\pi|_{\iota(\cA)}} & \iota(\cA)/\iota(\cB) \ar[r]^-{\iota'_{/\cB}} \ar[d] \pullbackcorner & \mathfrak{A}_{g - g_{\cB}} \ar[d] \\
\iota_S(S) \ar[r]^{\mathrm{id}} & \iota_S(S)  \ar[r]^-{\iota'_{/\cB,S}} & \A_{g - g_{\cB}} .
}
\]
It is not hard to check $(\iota'_{/\cB}\circ p'_{\cB})(\iota(X)) = (\iota_{/\cB} \circ p_{\cB})(X)$.

By construction of $\iota'_{/\cB,S}$, $(\iota(\cA)/\iota(\cB))|_{\iota_{/\cB,S}^{'-1}(a)} \rightarrow \iota_{/\cB,S}^{'-1}(a)$ is an isotrivial abelian scheme. Taking a splitting of $p'_{\cB}$ yields an isotrivial abelian subscheme of $\iota(\cA)|_{\iota_{/\cB,S}^{'-1}(a)} \rightarrow \iota_{/\cB,S}^{'-1}(a)$ which is non-trivial. For each $a$ Hodge generic in $\iota'_{/\cB,S}(\iota_(S))$, we have that $\iota_{/\cB,S}^{'-1}(a)$ is Hodge generic in $\iota_S(S)$. But then our hypothesis forces $\dim \iota_{/\cB,S}^{'-1}(a) = 0$. Hence $\iota'_{/\cB,S}$ is quasi-finite, and so is $\iota'_{/\cB}$. So $\dim (\iota'_{/\cB}\circ p'_{\cB})(\iota(X)) = \dim p'_{\cB}(\iota(X)) \ge \dim \iota(X) - g_{\cB} = d - g_{\cB} \ge \min(d,g) - g_{\cB}$. This contradicts \eqref{EqACZ} as $\dim (\iota'_{/\cB}\circ p'_{\cB})(\iota(X)) = \dim (\iota_{/\cB} \circ p_{\cB})(X)$.
\end{proof}

\begin{eg}\label{EgCounterexampleACZ}
The extra hypotheses (i) or (ii) in Theorem~\ref{ThmACZ} are necessary. We illustrate this with an example with $g = 4$. Consider $\mathfrak{A}_4 \rightarrow \A_4$. Now that $\mathfrak{A}_2 \times \mathfrak{A}_2 \rightarrow \A_2 \times \A_2$ is an abelian scheme of relative dimension $4$, it induces canonically the modular map
\[
\xymatrix{
\mathfrak{A}_2 \times \mathfrak{A}_2 \ar[r] \ar[d] \pullbackcorner & \mathfrak{A}_4 \ar[d] \\
\A_2 \times \A_2 \ar[r] & \A_4 .
}
\]
By abuse of notation we write $\mathfrak{A}_2 \times \mathfrak{A}_2$, resp. $\A_2 \times \A_2$, for the image of the morphism on the top, resp. on the bottom.

Let $S_1$ and $S_2$ be irreducible subvarieties of $\A_2$, then $S := S_1 \times S_2 \subseteq \A_2 \times \A_2$. The above diagram implies that $\mathfrak{A}_4|_S \rightarrow S$ is the product of the two abelian schemes $\mathfrak{A}_2|_{S_1} \rightarrow S_1$ and $\mathfrak{A}_2|_{S_2} \rightarrow S_2$.

Let $i \in \{1,2\}$. Let $\xi_i$ be a section of $\mathfrak{A}_2|_{S_i} \rightarrow S_i$ such that $\Z\xi_i$ is Zariski dense in $\mathfrak{A}_2|_{S_i}$ for $i \in \{1,2\}$. For the Betti map $b_{\Delta_i} \colon \mathfrak{A}_2|_{\Delta_i} \rightarrow \mathbb{T}^4$, where $\Delta_i \subseteq (S_i)^{\mathrm{an}}$ is an open subset, we have $\mathrm{rank}_{\R}(\mathrm{d}b_{\Delta_i}|_{\xi_i})_{s_i} \le 2\min(\dim S_i, 2)$ for all $s_i \in S_i(\C)$.

Now $\xi = (\xi_1,\xi_2)$ is a section of $\mathfrak{A}_4|_S \rightarrow S$ such that $\Z\xi$ is Zariski dense in $\mathfrak{A}_4|_S$. 

Take $S_1 = \A_2$ and $S_2$ to be a curve in $\A_2$ such that $\mathfrak{A}_2|_{S_2}$ has no fixed part (over any finite \'{e}tale covering of $S_2$), then $\mathfrak{A}_4|_S \rightarrow S$ has no fixed part (over any finite \'{e}tale covering of $S$) and $\dim S = 4 = g$. Let $\Delta = \Delta_1 \times \Delta_2$. Then the Betti map $b_{\Delta}$ equals $(b_{\Delta_1},b_{\Delta_2})$. So $\mathrm{rank}_{\R}(\mathrm{d}b_{\Delta}|_{\xi})_s \le \max_{s_1 \in \Delta_1}\mathrm{rank}_{\R}(\mathrm{d}b_{\Delta_1}|_{\xi_1})_{s_1} + \max_{s_2 \in \Delta_2} \mathrm{rank}_{\R}(\mathrm{d}b_{\Delta_2}|_{\xi_2})_{s_2} \le 2 (2 + \dim S_2) = 6 < 8 = 2\min(\dim \xi(S) ,g)$ for any $s \in \Delta$.
\end{eg}

\section{Application to fibered powers}\label{SectionAppFaltingsZhang}
The goal of this section is to prove (a generalized form of) Theorem~\ref{ThmFaltingsZhangBettiMapIntro}.

Recall the setting up. Let $S$ be an irreducible subvariety over $\C$ and let $\pi_S \colon \cA \rightarrow S$ be an abelian scheme of relative dimension $g_0 \ge 1$. We have the following modular map \eqref{EqEmbedAbSchIntoUnivAbVar}
\[
\xymatrix{
\cA \ar[r]^-{\iota} \ar[d]_{\pi_S} \pullbackcorner & \mathfrak{A}_{g_0} \ar[d] \\
S \ar[r]^-{\iota_S} & \A_{g_0}.
}
\]
Let $X$ be a closed irreducible subvariety of $\cA$ such that $\pi_S(X) = S$.

For any integer $m \ge 1$, set $\cA^{[m]} = \cA\times_S \ldots \times_S \cA$ ($m$-copies), $X^{[m]} = X \times_S \ldots \times_S X$ ($m$-copies)  
and by $b^{[m]}_{\Delta} = (b_{\Delta}, \ldots, b_{\Delta}) \colon \cA^{[m]}_{\Delta} \rightarrow \mathbb{T}^{2mg}$.

The modular map $\iota$ induces a morphism $\iota^{[m]} \colon \cA^{[m]} \rightarrow \mathfrak{A}_{mg_0}$. Set $X^{[m]}_{\iota} = \iota^{[m]}(X^{[m]})$. It is a closed subvariety of $\iota(X)\times_{\iota_S(S)} \ldots \times_{\iota_S(S)}\iota(X)$, which may be proper.

Let $\mathscr{D}_m^{\cA} \colon \cA^{[m+1]} \rightarrow \cA^{[m]}$ be the $m$th Faltings-Zhang map fiberwise defined by $(P_0, P_1, \ldots, P_m) \mapsto (P_1 - P_0, \ldots, P_m - P_0)$. In the particular case where $\cA \rightarrow S$ is $\mathfrak{A}_{g_0} \rightarrow \mathbb{A}_{g_0}$, abbreviate $\mathscr{D}_m = \mathscr{D}_m^{\mathfrak{A}_{g_0}}$. Then $\iota^{[m]} \circ \mathscr{D}_m^{\cA} = \mathscr{D}_m \circ \iota^{[m+1]}$.

\begin{thm}\label{ThmFaltingsZhang}
Assume that $X$ satisfies the following conditions: 
\begin{enumerate}
\item[(a)] we have $\dim X > \dim S$;
\item[(b)] for each $s \in S(\C)$, $X_s$ generates $\cA_s$;
\item[(c)] 
we have $X + \cA' \not\subseteq X$ for any non-isotrivial abelian subscheme $\cA'$ of $\cA \rightarrow S$.
\end{enumerate}
Then for each $t \ge 0$, we have the following:
\begin{enumerate}
\item[(i)] $\mathrm{rank}_{\R}(\mathrm{d}b_{\Delta}^{[m]}|_{X^{[m]}}) \ge 2(\dim X_{\iota}^{[m]} - t)$ for all $m \ge \dim S - t$.
\item[(ii)] $\mathrm{rank}_{\R}(\mathrm{d}b_{\Delta}^{[m]}|_{\mathscr{D}_m^{\cA}(X^{[m+1]})}) \ge 2(\dim \mathscr{D}_m(X_{\iota}^{[m+1]}) - t)$ for all $m \ge \dim X - t$.
\end{enumerate}
\end{thm}

Before proceeding to the proof, let us fix some notation. 
Set $d = \dim X - \dim S > 0$. 
 For any $m \ge 1$, we have $\dim X^{[m]} - \dim S = md$. 
 In particular, $\dim X^{[m]}_{\iota} \le \dim X^{[m]} = md + \dim S$.

For notation let $\cA_{X^{[m]}}$ be the translate of an abelian subscheme of $\cA^{[m]} \rightarrow S$ by a torsion section which contains $X^{[m]}$, minimal for this property. Hypothesis (b) then implies 
$\cA_{X^{[m]}} = \cA^{[m]}$ for all $m \ge 1$.


\begin{proof}[Proof of Theorem~\ref{ThmFaltingsZhang}.(i)]
Suppose $\mathrm{rank}_{\R}(\mathrm{d}b_{\Delta}^{[m]}|_{X^{[m]}}) < 2(\dim X_{\iota}^{[m]} - t)$. Then by Theorem~\ref{ThmCriterionDegIntro}, there exists an abelian subscheme $\cB$ of $\cA^{[m]} \rightarrow S$ (whose relative dimension we denote by $g_{\cB}$) such that 
\begin{equation}\label{EqXDegenerateAppCurve}
\dim (\iota_{/\cB}\circ p_{\cB})(X^{[m]}) < \dim X^{[m]}_{\iota} - t- g_{\cB}
\end{equation}
for the following diagram.
\begin{equation}\label{EqGeomMeaningQuotShimuraCritAppFZ}
\xymatrix{
\cA^{[m]} \ar[r]^-{p_{\cB}} \ar[d]_{\pi_S} & \cA^{[m]}/\cB \ar[d]_{\pi_{/\cB}} \ar[r]^-{\iota_{/\cB}} \pullbackcorner & \mathfrak{A}_{mg_0-g_{\cB}} \ar[d]^{\pi'} \\
S \ar[r]^-{\mathrm{id}_S} & S \ar[r]^-{\iota_{/\cB,S}} & \A_{mg_0-g_{\cB}}
}
\end{equation}

The following fact will be proved in Appendix~\ref{SectionAppendixLinearAlgebra}. There exists an isogeny
\begin{equation}\label{EqNewDecomposition}
\rho \colon \cA^{[m]} \rightarrow \cA^{[m]}
\end{equation}
such that $\rho(\cB) = \cB_1 \times_S \ldots \times_S \cB_m$ for some abelian subschemes $\cB_1,\ldots,\cB_m$ of $\cA$. 
Moreover we may assume that $\rho$ satisfies the following property: $q_1 \circ \rho = \rho_1 \circ q_1$ for some isogeny $\rho_1 \colon \cA\rightarrow \cA$, where $q_1 \colon \cA^{[m]} \rightarrow \cA$ is the projection to the first factor.


\begin{lemma}\label{LemmaNonIsotrivial}
$\cB_i \rightarrow S$ is non-isotrivial for each $i \in \{1,\ldots, m\}$.
\end{lemma}

We will see that Lemma~\ref{LemmaNonIsotrivial} implies the following naive lower bound for $\dim (\iota_{/\cB} \circ p_{\cB})(X^{[m]})$ 
\begin{equation}\label{EqpNXprimeLowerBound}
\dim (\iota_{/\cB} \circ p_{\cB})(X^{[m]}) \ge m(d+1)-g_{\cB}.
\end{equation}
Let us finish the proof by assuming Lemma~\ref{LemmaNonIsotrivial} and hence \eqref{EqpNXprimeLowerBound}. 
We have
\begin{align*}
& g_{\cB} - (\dim X^{[m]}_{\iota} - \dim (\iota_{/\cB} \circ p_{\cB})(X^{[m]})) \\
= & g_{\cB} + \dim (\iota_{/\cB} \circ p_{\cB})(X^{[m]}) - \dim X^{[m]}_{\iota} \\
\ge & m(d+1) - (md+\dim S) \\
\ge & m - \dim S.  
\end{align*}

This contradicts \eqref{EqXDegenerateAppCurve} for $m \ge \dim S - t$. So we are done.
\end{proof}

\begin{proof}[Proof of Lemma~\ref{LemmaNonIsotrivial}]
Suppose $\cB_{i_0} \rightarrow S$ is isotrivial for some $i_0$. 
Note that $(\cA^{[m]}/\cB)|_{\iota_{/\cB,S}^{-1}(a)} \rightarrow \iota_{/\cB,S}^{-1}(a)$ is an isotrivial abelian scheme for each $a$. Hence $\cA/\cB_{i_0}$, being an abelian subscheme of $\cA^{[m]}/\cB$ via the decomposition \eqref{EqNewDecomposition}, is isotrivial when restricted to ${\iota_{/\cB,S}^{-1}(a)}$. Thus $\cA|_{\iota_{/\cB,S}^{-1}(a)} \rightarrow \iota_{/\cB,S}^{-1}(a)$ is isotrivial. So by the discussion of the end of $\mathsection$\ref{SectionNotation}, $\dim \iota_S(\iota_{/\cB,S}^{-1}(a)) = 0$  where $\iota_S \colon S \rightarrow \mathbb{A}_{g_0}$ is the modular map for $\cA \rightarrow S$. 

Applying the modular map $\iota \colon \cA \rightarrow \mathfrak{A}_{g_0}$ (and $\iota_S \colon S \rightarrow \mathbb{A}_{g_0}$) to \eqref{EqGeomMeaningQuotShimuraCritAppFZ}, we obtain (denote by $\cA^{[m]}_{\iota} = \iota^{[m]}(\cA^{[m]})$) the following diagram.
\begin{equation}\label{EqGeomMeaningQuotShimuraCritAppFZMod}
\xymatrix{
\cA^{[m]}_{\iota} \ar[r]^-{p_{\iota^{[m]}(\cB)}} \ar[d]_{\pi} & \cA^{[m]}_{\iota} /\iota^{[m]}(\cB) \ar[d]_{\pi_{/\iota^{[m]}(\cB)}} \ar[r]^-{\iota'_{/\cB}} \pullbackcorner & \mathfrak{A}_{mg_0-g_{\cB}} \ar[d]^{\pi'} \\
\iota_S(S) \ar[r]^-{\mathrm{id}} & \iota_S(S) \ar[r]^-{\iota'_{/\cB,S}} & \A_{mg_0-g_{\cB}}
}
\end{equation}
The previous paragraph implies that $\dim (\iota'_{/\cB,S})^{-1}(a') =0$ for each $a'$. Thus $\iota'_{/\cB,S}$ is generically finite, and so is $\iota'_{/\cB}$.

It is not hard to check $(\iota_{/\cB}\circ p_{\cB})(X^{[m]}) = (\iota'_{/\cB} \circ p_{\iota^{[m]}(\cB)})(X^{[m]}_{\iota})$. So
\[
\dim (\iota_{/\cB}\circ p_{\cB})(X^{[m]}) = \dim (\iota'_{/\cB} \circ p_{\iota^{[m]}(\cB)})(X^{[m]}_{\iota}) = \dim p_{\iota^{[m]}(\cB)}(X^{[m]}_{\iota}) \ge \dim X^{[m]}_{\iota} - g_{\cB},
\]
where the last inequality holds because $\iota^{[m]}(\cB) \rightarrow \iota_S(S)$ has relative dimension $g_{\cB}$. This contradicts \eqref{EqXDegenerateAppCurve} since $t \ge 0$.
\end{proof}

\begin{proof}[Proof of \eqref{EqpNXprimeLowerBound}]
Use the notation in \eqref{EqGeomMeaningQuotShimuraCritAppFZ}. Take a point $s \in S$ such that $(\iota_{/\cB} \circ p_{\cB})(X^{[m]}) \rightarrow \pi'((\iota_{/\cB} \circ p_{\cB})(X^{[m]}))$ is flat over $s' := \iota_{/\cB,S}(s)$. Then $\dim (\iota_{/\cB} \circ p_{\cB})(X^{[m]}) = \dim \pi'((\iota_{/\cB} \circ p_{\cB})(X^{[m]})) + \dim ((\pi')^{-1}(s') \cap (\iota_{/\cB} \circ p_{\cB})(X^{[m]})) \ge \dim ((\pi')^{-1}(s') \cap (\iota_{/\cB} \circ p_{\cB})(X^{[m]}))$. So it suffices to prove the lower bound for $\dim ((\pi')^{-1}(s') \cap (\iota_{/\cB} \circ p_{\cB})(X^{[m]}))$.

On the other hand $(\iota_{/\cB} \circ p_{\cB})(\cA^{[m]}_s \cap X^{[m]}) \subseteq (\pi')^{-1}(s') \cap (\iota_{/\cB} \circ p_{\cB})(X^{[m]}))$ by the commutative diagram \eqref{EqGeomMeaningQuotShimuraCritAppFZ}. Hence it suffices to prove the lower bound for $\dim (\iota_{/\cB} \circ p_{\cB})(\cA^{[m]}_s \cap X^{[m]})$. But $p_{\cB}(\cA^{[m]}_s \cap X^{[m]})$ is contained in one fiber of $\pi_{/\cB}$. So $\iota_{\cB}|_{p_{\cB}(\cA^{[m]}_s \cap X^{[m]})}$ is an isomorphism. So
\[
\dim (\iota_{/\cB} \circ p_{\cB})(\cA^{[m]}_s \cap X^{[m]}) =  \dim p_{\cB} (\cA^{[m]}_s \cap X^{[m]}). 
\]
Thus it suffices to prove the lower bound for $\dim p_{\cB}(\cA^{[m]}_s \cap X^{[m]}) = \dim p_{\cB}(X_s^m)$.

The isogeny $\rho$ in \eqref{EqNewDecomposition} gives an isogeny $\cA_s^m \rightarrow \cA_s^m$, which we still denote by $\rho$ by abuse of notation. Set for simplicity $B = \cB_s$, $B_i = (\cB_i)_s$ for $i \in \{1,\ldots,m\}$. Denote by $p_i \colon \cA_s \rightarrow \cA_s/B_i$ the quotient. 
Then by the last paragraph, to prove \eqref{EqpNXprimeLowerBound} it suffices to prove
\begin{equation}\label{EqpNXprimeLowerBoundFiber}
\dim (p_1,\cdots,p_m)(\rho(X_s^m)) \ge m(d+1) - \sum_{i=1}^m \dim B_i.
\end{equation} 

By Hypothesis (c) on $X$, for a very general $s$ we have $X_s + \cA'_s \not\subseteq X_s$ for any non-isotrivial abelian subscheme $\cA'$ of $\cA$. In particular, $X_s + B_i \not\subseteq X_s$ by Lemma~\ref{LemmaNonIsotrivial}. Thus
\begin{equation}\label{EqDescreaseByQuotient}
\dim p_i(X_s) \ge d - (\dim B_i -1).
\end{equation}

Let us prove \eqref{EqpNXprimeLowerBoundFiber} for such an $s$ by induction on $m$. For $m = 1$, this is precisely \eqref{EqDescreaseByQuotient}. For general $m \ge 2$, denote by $q_1 \colon \cA_s^m \rightarrow \cA_s$ the projection to the first factor. Recall that $q_1 \circ \rho = \rho_1 \circ q_1$ for some isogeny $\rho_1 \colon \cA_s \rightarrow \cA_s$; see below \eqref{EqNewDecomposition}. For each $z \in \rho_1(X_s)$, we have
\begin{align*}
\rho(X_s^m) \cap q_1^{-1}(z) & = \rho \left(X_s^m \cap \rho^{-1}(q_1^{-1}(z)) \right)  = \rho \left(X_s^m \cap (q_1\circ \rho)^{-1}(z) \right) \\
& = \rho \left(X_s^m \cap (\rho_1\circ q_1)^{-1}(z) \right) = \rho \left(X_s^m \cap q_1^{-1}(\rho_1^{-1}(z)) \right) \\
& = \rho \left(X_s^m \cap \{\rho_1^{-1}(z)\} \times \cA_s^{m-1} \right) = \bigcup_{z' \in X_s \cap \rho_1^{-1}(z)} \rho(\{z'\} \times X_s^{m-1}).
\end{align*}
Thus each irreducible component of a non-empty fiber of $q_1|_{\rho(X_s^m)}$ is a translate of $\rho'_1(X_s^{m-1})$ for the isogeny $\rho'_1 \colon \cA_s^{m-1} \rightarrow \cA_s^{m-1}$ obtained from $\rho|_{\{0\}\times \cA_s^{m-1}}$.

We wish to apply the induction hypothesis to $\rho'_1(X_s^{m-1})$ and the quotient morphism $(p_2,\cdots,p_m)$. However we cannot directly do this because $\rho'_1$ may not satisfy the extra property that $q_1 \circ \rho'_1$ and $q_1$ differ from an isogeny $\cA_s \rightarrow \cA_s$.

To solve this problem, let $\cB_1^{\perp}$ be the neutral component of $\rho^{-1}(\epsilon_S\times_S \cB_2 \times_S \ldots \times_S \cB_m) \subseteq \rho^{-1}(\epsilon_S \times_S \cA^{[m-1]})$, where $\epsilon_S$ is the zero section of $\cA \rightarrow S$. Apply \eqref{EqNewDecomposition} to $\cB_1^\perp \subseteq \cA^{[m-1]}$. We thus obtain an isogeny
\[
\rho' \colon \cA^{[m-1]} \rightarrow \cA^{[m-1]}
\]
such that $\rho'(\cB_1^\perp) = \cB'_2 \times_S \ldots \times_S \cB'_m$ and that $q_1 \circ \rho' = \rho_2 \circ q_1$ for some isogeny $\rho_2 \colon \cA \rightarrow \cA$. As $\rho'_1 \colon \cA_s^{m-1} \rightarrow \cA_s^{m-1}$ is obtained from $\rho|_{\{0\}\times \cA_s^{m-1}}$, we have
\[
\dim(p_2,\cdots,p_m)(\rho_1'(X_s^{m-1})) = \dim(p_2',\cdots,p_m')(\rho'(X_s^{m-1})),
\]
where $p_i' \colon \cA_s \rightarrow \cA_s/(\cB'_i)_s$ is the quotient morphism.

Lemma~\ref{LemmaNonIsotrivial} applies to $\cB'_i$ since $\cA/\cB'_i$ is still an abelian subscheme of $\cA^{[m]}/\cB$. So $\cB'_i \rightarrow S$ is non-isotrivial. Hence we can apply the induction hypothesis to $\rho'(X_s^{m-1})$ and $(p_2',\cdots,p_m')$. So
\[
\dim(p_2',\cdots,p_m')(\rho'(X_s^{m-1})) \ge (m-1)(d+1) - \sum_{i=2}^m \dim (\cB'_i)_s = (m-1)(d+1) - \sum_{i=2}^m \dim B_i.
\]
Denote by $\bar{q}_1 \colon \cA_s/B_1 \times \cdots \cA_s/B_m \rightarrow \cA_s/B_1$ the projection to the first factor. 
Then we have $\bar{q}_1((p_1,\cdots,p_m)(\rho(X_s^m))) = p_1(q_1(\rho(X_s^m))) = p_1(q_1\circ \rho(X_s^m)) = p_1(\rho_1\circ q_1(X_s^m)) = p_1(\rho_1(X_s))$. But $\dim p_1(\rho_1(X_s)) \ge  d - (\dim B_1 - 1)$ by the same argument for \eqref{EqDescreaseByQuotient}. Hence by the fiber dimension theorem we have
\[
\dim (p_1,\cdots,p_m)(\rho(X_s^m)) \ge \left(d - (\dim B_1 - 1)\right) + \left((m-1)(d+1) - \sum_{i=2}^m \dim B_i\right) = m(d+1) - \sum_{i=1}^m\dim B_i.
\]
This finishes the proof.
\end{proof}

\begin{proof}[Proof of Theorem~\ref{ThmFaltingsZhang}.(ii)] 
The proof of part (ii) is similar. Let us sketch it. 

We reduce (ii) to the case where $X$ contains the zero section of $\pi_{\cA} \colon \cA \rightarrow S$. Indeed, up to replacing $S$ by a Zariski open subset (which does not change the generic Betti rank in question) we can take a section $\sigma \colon S \rightarrow X$ of $\pi_{\cA}|_X$. Then $\mathscr{D}_m^{\cA}(X^{[m+1]}) = \mathscr{D}_m^{\cA}((X-\sigma(S))^{[m+1]})$ for all $m \ge 1$. It suffices then to replace $X$ by $X-\sigma(S)$.

Set $Y := \mathscr{D}_m^{\cA}(X^{[m+1]})$. We have $Y \supseteq X^{[m]}$ because $X^{[m]} = \mathscr{D}_m^{\cA}(0_S \times_S X \times_S \cdots \times_S X) \subseteq \mathscr{D}_m^{\cA}(X^{[m+1]})$. 
Let $\cA_Y$ be the translate of an abelian subscheme of $\cA^{[m]} \rightarrow S$ by a torsion section which contains $Y$, minimal for this property. Then $\cA_Y = \cA^{[m]}$ since $\cA_{X^{[m]}} = \cA$.

Suppose $\mathrm{rank}_{\R}(\mathrm{d}b_{\Delta}^{[m]}|_Y) < 2(\dim \iota^{[m]}(Y) - t)$. Applying Theorem~\ref{ThmCriterionDegIntro}.(1) to $Y$, we get an abelian subscheme $\cB$ of $\cA^{[m]} \rightarrow S$ (whose relative dimension we denote by $g_{\cB}$) such that for the map $\iota_{/\cB} \circ p_{\cB}$ as in \eqref{EqGeomMeaningQuotShimuraCritAppFZ}, we have $\dim (\iota_{/\cB}\circ p_{\cB})(Y) < \dim \iota^{[m]}(Y) - t - g_{\cB}$. This $\cB$ is different from the one in the proof of part (i).

Under a possibly new decomposition (up to isogeny) 
$\cA^{[m]} \cong \cA \times_S \ldots \times_S \cA$ ($m$-copies),
we may write $\cB = \cB_1 \times_S \ldots \times_S \cB_m$ for some abelian subschemes $\cB_1,\ldots,\cB_m$ of $\cA$. One can show that $\cB_i \rightarrow S$ is non-isotrivial for each $i \in \{1,\ldots, m\}$: indeed it suffices to take a verbalized copy of the proof of Lemma~\ref{LemmaNonIsotrivial} with $X^{[m]}$ replaced by $Y$ and $X^{[m]}_{\iota}$ replace by $\iota^{[m]}(Y)$. 

Then \eqref{EqpNXprimeLowerBound} holds for this new $\cB$. Thus
\begin{align*}
 & g_{\cB} + t + \dim (\iota_{/\cB}\circ p_{\cB})(Y) - \dim Y \\
\ge & g_{\cB} + t + \dim (\iota_{/\cB}\circ p_{\cB})(Y) - ((m+1)d + \dim S) \\
\ge & g_{\cB} + t + \dim (\iota_{/\cB}\circ p_{\cB})(X^{[m]}) - (m+1)d - \dim S  \quad \text{ since $X^{[m]} \subseteq Y$} \\
\ge & t + m - d - \dim S  \quad \text{ by \eqref{EqpNXprimeLowerBound}} \\
= & t + m - \dim X 
\end{align*}
This contradicts $\dim (\iota_{/\cB}\circ p_{\cB})(Y) < \dim \iota^{[m]}(Y) - t - g_{\cB}$ when $m \ge \dim X - t$.

So $\mathrm{rank}_{\R}(\mathrm{d}b_{\Delta}^{[m]}|_Y) \ge 2(\dim \iota^{[m]}(Y) - t)$ when $m \ge \dim X - t$. So the conclusion follows because  $\mathrm{rank}_{\R}(\mathrm{d}b_{\Delta}^{[m]}|_{\mathscr{D}_m^{\cA}(X^{[m+1]})}) = \mathrm{rank}_{\R}(\mathrm{d}b_{\Delta}^{[m]}|_Y)$ and
$\dim \mathscr{D}_m(X^{[m+1]}_{\iota}) = \dim (\iota^{[m]} \circ \mathscr{D}_m^{\cA})(X^{[m+1]}) = \dim \iota^{[m]}(Y)$.
\end{proof}

\section{Link with relative Manin-Mumford}\label{SectionRelativeMM}
In this section, we discuss about the relative Manin-Mumford conjecture. It is closely related to $X^{\mathrm{deg}}(1)$.

Let $S$ be an irreducible variety over $\C$, and let $\pi_S \colon \cA \rightarrow S$ be an abelian scheme of relative dimension $g \ge 1$. 
Denote by $\cA_{\mathrm{tor}}$ the set of points $x \in \cA(\C)$ such that $[N]x$ lies in the zero section of $\cA \rightarrow S$ for some integer $N$. In other words $x$ is a torsion point in its fiber.

Let $X$ be a closed irreducible subvariety such that $\pi_S(X) = S$. Denote by $\cA_X$ the translate of an abelian subscheme of $\cA \rightarrow S$ by a torsion section which contains $X$, minimal for this property.

\begin{namedconj}[Relative Manin-Mumford]\label{ConjMM}
If $(X \cap \cA_{\mathrm{tor}})^{\Zar} = X$, then $\codim_{\cA_X}(X) \le \dim S$.
\end{namedconj}

\begin{conj}\label{ConjMMPre}
Assume $S$, $\pi_S$, and $X$ are defined over $\bar{\Q}$. If $(X \cap \cA_{\mathrm{tor}})^{\Zar} = X$, then $X^{\mathrm{deg}}(1)$ is Zariski dense in $X$.
\end{conj}

The goal is to reduce the Relative Manin-Mumford Conjecture to Conjecture~\ref{ConjMMPre}.

\begin{prop}\label{RelativeMMEquiv}
Conjecture~\ref{ConjMMPre} implies the Relative Manin-Mumford Conjecture. More precisely for $X \subseteq \cA \rightarrow S$ such that $(X \cap \cA_{\mathrm{tor}})^{\Zar} = X$, we have $\codim_{\cA_X}(X) \le \dim S$ if Conjecture~\ref{ConjMMPre} holds for any irreducible subvariety $X' \subseteq \mathfrak{A}_{g'}$ defined over $\bar{\Q}$ with $1 \le g' \le g$, $\dim X' \le \dim X$ and $\dim \pi'(X') \le \dim S$.\footnote{Here $\pi' \colon \mathfrak{A}_{g'} \rightarrow \A_{g'}$ is the universal abelian variety.}
\end{prop}

\begin{proof}
First let us note that by standard specialization argument, if relative Manin-Mumford holds for \textbf{\textit{all}} $S$, $\pi_S$ and $X$ defined over $\bar{\Q}$, then it holds for all $S$, $\pi_S$ and $X$ defined over $\C$.

Let $X \subseteq \cA \rightarrow S$ be as in the relative Manin-Mumford conjecture \textit{which is defined over $\bar{\Q}$}. Note that $\cA_X \rightarrow S$ itself is an abelian scheme. Replacing $\cA \rightarrow S$ by $\cA_X \rightarrow S$, we may assume $\cA_X = \cA$. It suffices to prove $\codim_{\cA}(X) \le \dim S$.

Let us prove it by induction on $(\dim S, g)$, upward on both parameters. When $\dim S = 0$ it follows from the classical Manin-Mumford conjecture (first proved by Raynaud \cite{RaynaudCourbes-sur-une} and then re-proved by many others). The conclusion for the case $g = 0$ is easily true.


In general, consider the modular map \eqref{EqEmbedAbSchIntoUnivAbVar}
\[
\xymatrix{
\cA \ar[r]^{\iota} \ar[d]_{\pi_S} \pullbackcorner & \mathfrak{A}_g \ar[d]^{\pi} \\
S \ar[r]^{\iota_S} & \A_g .
}
\]
Denote by $B = \iota_S(S)$. Then $\iota(\cA) = \pi^{-1}(B) =: \mathfrak{A}_g|_B$. Now $(X\cap \cA_{\mathrm{tor}})^{\Zar} = X$ implies $(\iota(X) \cap \iota(\cA)_{\mathrm{tor}})^{\Zar} = \iota(X)$. Applying Conjecture~\ref{ConjMMPre} to $\iota(X) \subseteq \iota(\cA)$,\footnote{Note that $\iota$ and $\iota_S$ are defined over $\bar{\Q}$, $\dim \iota(X) \le \dim X$ and $\dim \pi(\iota(X)) = \dim \iota_S(S) \le \dim S$.} we get that $\iota(X)^{\mathrm{deg}}(1)$ is Zariski dense in $\iota(X)$. Hence by Theorem~\ref{ThmDegeneracyLocusZarClosed} we have $\iota(X)^{\mathrm{deg}}(1) = \iota(X)$. Applying Theorem~\ref{ThmCriterionXDegenerateGeom} to $t = 1$ and $\iota(X)$, we get an abelian subscheme $\cB$ of $\mathfrak{A}_g|_B \rightarrow B$ (whose relative dimension we denote by $g_{\cB}$) such that for 
\[\xymatrix{
\mathfrak{A}_g|_B \ar[r]^-{p_{\cB}} \ar[d]_{\pi|_{\mathfrak{A}_g|_B}} & (\mathfrak{A}_g|_B)/\cB \ar[d]_{\pi_{/\cB}} \ar[r]^-{\iota_{/\cB}} \pullbackcorner & \mathfrak{A}_{g - g_{\cB}} \ar[d]^{\pi'} \\
B \ar[r]^-{\mathrm{id}_B} & B \ar[r]^-{\iota_{/\cB,G}} & \A_{g-g_{\cB}} ,
}
\]
we have  that $\iota_{/\cB} \circ p_{\cB}$ is not generically finite and 
\begin{equation}\label{EqRelativeMM}
\dim (\iota_{/\cB} \circ p_{\cB})(\iota(X)) < \dim \iota(X) - g_{\cB} + 1.
\end{equation}
Hence we have (denote by $\mathfrak{A}_{g - g_{\cB}}|_{\iota_{/\cB,G}(B)} = (\pi')^{-1}(\iota_{/\cB.G}(B))$)
\begin{align*}
& \scalebox{0.85}{ $(\codim_{\cA}(X) - \dim S) - \left(\codim_{(\mathfrak{A}_{g - g_{\cB}}|_{\iota_{/\cB,G}(B)})}(\iota_{/\cB}\circ p_{\cB})(\iota(X)) - \dim \iota_{/\cB,G}(B) \right)$ }  \\
= & \scalebox{0.85}{ $(\dim \cA - \dim S - \dim X) - \left(\dim (\mathfrak{A}_{g - g_{\cB}}|_{\iota_{/\cB,G}(B)}) - \dim \iota_{/\cB,G}(B) - \dim (\iota_{/\cB}\circ p_{\cB})(\iota(X)) \right)$ }  \\
= & g - \dim X - (g-g_{\cB}) + \dim (\iota_{/\cB}\circ p_{\cB})(\iota(X))  \\
= & - \dim X + \dim (\iota_{/\cB}\circ p_{\cB})(\iota(X)) + g_{\cB} \nonumber \\
\le & - \dim \iota(X) + \dim (\iota_{/\cB}\circ p_{\cB})(\iota(X)) + g_{\cB} \le 0,
\end{align*}
and so
\begin{equation}\label{EqRelativeMMLong}
\codim_{(\mathfrak{A}_{g - g_{\cB}}|_{\iota_{/\cB,G}(B)})}(\iota_{/\cB}\circ p_{\cB})(\iota(X)) \le \dim \iota_{/\cB,G}(B) \Longrightarrow \codim_{\cA}(X) \le \dim S.
\end{equation}

If $\dim \iota_{/\cB,G}(B) < \dim B$, then $\dim \iota_{/\cB,G}(B) < \dim S$ since $\dim B = \dim \iota_S(S) \le \dim S$. Now we can apply the induction hypothesis on $\dim S$ to get
\[
\codim_{(\mathfrak{A}_{g - g_{\cB}}|_{\iota_{/\cB,G}(B)})}(\iota_{/\cB}\circ p_{\cB})(\iota(X)) \le \dim \iota_{/\cB,G}(B).
\]
So $\codim_{\cA}(X) \le \dim S$ by \eqref{EqRelativeMMLong}.

If $\dim \iota_{/\cB,G}(B) = \dim B$, then $\iota_{/\cB,G}$ is generically finite. So $\iota_{/\cB}$ is generically finite. But $\iota_{/\cB} \circ p_{\cB}$ is not generically finite. So $g_{\cB} > 0$. Moreover
\[
\dim (\iota_{/\cB} \circ p_{\cB})(\iota(X)) = \dim p_{\cB}(\iota(X)) \ge \dim \iota(X) - g_{\cB},
\]
and hence $\dim p_{\cB}(\iota(X)) = \dim \iota(X) - g_{\cB}$ by \eqref{EqRelativeMM}. So $\iota(X) + \cB = \iota(X)$.

Now the assumption $(X\cap \cA_{\mathrm{tor}})^{\Zar} = X$ implies
\[
\big( p_{\cB}(\iota(X)) \cap ((\mathfrak{A}_g|_B)/\cB)_{\mathrm{tor}} \big)^{\Zar} = p_{\cB}(\iota(X)).
\]
As $(\mathfrak{A}_g|_B)/\cB \rightarrow B$ has relative dimension $g - g_{\cB} < g$,  we can apply the induction hypothesis on $g$ to get $\codim_{(\mathfrak{A}_g|_B)/\cB}(p_{\cB}(\iota(X))) \le \dim B$.

As both $\iota_{/\cB}$ and $\iota_{/\cB,G}$ are generically finite, we have $\dim (\mathfrak{A}_g|_B)/\cB = \dim \iota_{/\cB}((\mathfrak{A}_g|_B)/\cB) = \dim \mathfrak{A}_{g-g_{\cB}}|_{\iota_{\cB,G}(B)}$, $\dim p_{\cB}(\iota(X)) = \dim (\iota_{/\cB}\circ p_{\cB})(\iota(X))$ and $\dim B = \dim \iota_{/\cB,G}(B)$. So the left hand side of \eqref{EqRelativeMMLong} holds by the previous paragraph. Thus $\codim_{\cA}(X) \le \dim S$.
\end{proof}

\appendix
\renewcommand{\thesection}{\Alph{section}}
\setcounter{section}{0}

\section{Discussion when the base takes some simple form}\label{SectionAppendixParticularBase}
Let $S$ be an irreducible subvariety over $\C$ and let $\pi_S \colon \cA \rightarrow S$ be an abelian scheme of relative dimension $g \ge 1$. Let $X$ be a closed irreducible subvariety of $\cA$ with $\dim \pi_S(X) = S$.

\begin{defn}\label{DefnGenSp}
Denote by $\langle X \rangle_{\textrm{gen-sp}}$ the smallest subvariety of $\cA$ of the following form which contains $X$:
Up to taking a finite covering of $S$, we have $\langle X \rangle_{\textrm{gen-sp}} = \sigma + \cZ + \cB$, where $\cB$ is an abelian subscheme of $\cA \rightarrow S$, $\sigma$ is a torsion section of $\cA \rightarrow S$, and $\cZ = Z \times S$ where $C \times S$ is the largest constant abelian subscheme of $\cA \rightarrow S$ and $Z \subseteq C$.
\end{defn}

If we furthermore require $Z$ to be a point, then we obtain $\langle X \rangle_{\mathrm{sg}}$ (Definition~\ref{DefnSubvarOfsgType}).

Let $\cA_X$ be the translate of an abelian subscheme of $\cA \rightarrow S$ by a torsion section which contains $X$, minimal for this property. Then $\cA_X \rightarrow S$ itself is an abelian scheme, whose relative dimension we denote by $g_X$. It is easy to see $\langle X \rangle_{\mathrm{gen-sp}} \subseteq \langle X \rangle_{\mathrm{sg}} \subseteq \cA_X$.

The variety $\langle X \rangle_{\textrm{gen-sp}}$ is closely related to $\mathrm{rank}_{\R}(\mathrm{d}b_{\Delta}|_X)$. Let us start with the following proposition, saying that $\mathrm{rank}_{\R}(\mathrm{d}b_{\Delta}|_X)$ attains its minimal value if and only if $X = \langle X \rangle_{\textrm{gen-sp}}$.
\begin{prop}
$\mathrm{rank}_{\R}(\mathrm{d}b_{\Delta}|_X) = 2(\dim X - \dim S) \Longleftrightarrow X = \langle X \rangle_{\mathrm{gen-sp}}$.

In particular if $\dim S = 1$, then either $\mathrm{rank}_{\R}(\mathrm{d}b_{\Delta}|_X) = 2 \dim X$ or $X = \langle X \rangle_{\mathrm{gen-sp}}$.
\end{prop}
\begin{proof} The direction $\Leftarrow$ is not hard to check. Let us prove $\Rightarrow$ now.

Apply Theorem~\ref{ThmCriterionDegIntro}.(1) to $l = \dim X - \dim S + 1$. Then $\mathrm{rank}_{\R}(\mathrm{d}b_{\Delta}|_X) < 2(\dim X - \dim S + 1)$ if and only if there exists an abelian subscheme $\cB$ of $\cA_X \rightarrow S$ such that
\[
\dim (\iota_{/\cB} \circ p_{\cB})(X) \le \dim X - \dim S - g_{\cB}
\]
for
\[
\xymatrix{
\cA_X \ar[r]^-{p_{\cB}} \ar[d]_{\pi_S|_{\cA_X}} & \cA_X/\cB \ar[d]_{\pi_{/\cB}} \ar[r]^-{\iota_{/\cB}} \pullbackcorner & \mathfrak{A}_{g_X-g_{\cB}} \ar[d]^{\pi'} \\
S \ar[r]^-{\mathrm{id}_S} & S \ar[r]^-{\iota_{/\cB,S}} & \A_{g_X-g_{\cB}} .
}
\]
Thus 
$2(\dim X - \dim S) = \mathrm{rank}_{\R}(\mathrm{d}b_{\Delta}|_X) \le 2g_{\cB} + 2\dim (\iota_{/\cB} \circ p_{\cB})(X) \le 2(\dim X - \dim S)$. 
This implies $\dim p_{\cB}(X) = \dim X - g_{\cB}$, and $\dim (\iota_{/\cB}\circ p_{\cB})(X) = \dim X - \dim S - g_{\cB} = \dim p_{\cB}(X) - \dim S$. So $p_{\cB}|_X$ has relative dimension $g_{\cB}$, and $\iota_{/\cB}|_{p_{\cB}(X)}$ has relative dimension $\dim S$. So $X = p_{\cB}^{-1}(p_{\cB}(X))$ as $p_{\cB}$ has relative dimension $g_{\cB}$, and $p_{\cB}(X)= \iota_{/\cB}^{-1}(\iota_{/\cB}(p_{\cB}(X)))$ with $\iota_{/\cB}$ having relative dimension $\dim S$. So $\iota_{/\cB,S}$ has relative dimension $\dim S$, making $\iota_{/\cB,S}(S)$ a point.

Now $X = (\iota_{/\cB}\circ p_{\cB})^{-1}(\iota_{/\cB}\circ p_{\cB})(X)$ and $\iota_{/\cB,S}(S)$ is a point. So $X = \langle X \rangle_{\mathrm{gen-sp}}$.
\end{proof}

To further investigate the relation between $\mathrm{rank}_{\R}(\mathrm{d}b_{\Delta}|_X)$ with $\langle X \rangle_{\textrm{gen-sp}}$, let us first of all recall the modular map \eqref{EqEmbedAbSchIntoUnivAbVar}
\[
\xymatrix{
\cA \ar[r]^-{\iota} \ar[d]_{\pi_S} \pullbackcorner & \mathfrak{A}_g \ar[d] \\
S \ar[r]^-{\iota_S} & \A_g .
}
\]

\begin{prop}\label{PropDegSimpleBase}
We have
\begin{enumerate}
\item[(i)] $\mathrm{rank}_{\R}(\mathrm{d}b_{\Delta}|_X) \le 2\min(\dim \iota(X), \dim \langle X \rangle_{\mathrm{gen-sp}} - \dim S)$.
\item[(ii)] Assume $\dim \iota_S(S) = 1$ or the connected algebraic monodromy group of $\iota_S(S)$ is simple. Then $\mathrm{rank}_{\R}(\mathrm{d}b_{\Delta}|_X) \ge 2\min(\dim \iota(X), \dim \langle X \rangle_{\mathrm{gen-sp}} - \dim S)$.
\end{enumerate}
\end{prop}

Before proving this proposition, let us see its direct corollary on the Betti rank.
\begin{cor}\label{CorBettiRankSimpleBase}
Assume that $\iota_S(S)$ has dimension $1$ or has simple connected algebraic monodromy group. Then
\[
\mathrm{rank}_{\R}(\mathrm{d}b_{\Delta}|_X) = 2 \min(\dim \iota(X), \dim \langle X \rangle_{\mathrm{gen-sp}} - \dim S).
\]
\end{cor}
We point out that the hypotheses in Corollary~\ref{CorBettiRankSimpleBase} cannot be removed; see Example~\ref{EgCounterexampleACZ}.

\begin{rmk}\label{RmkDegSimpleBase}
Let $\mathfrak{C}_g$ be the universal curve embedded in $\mathfrak{A}_g$. Use notations in Theorem~\ref{ThmUnivCurveBettiRankIntro} and Theorem~\ref{ThmUnivCurveBettiRankIntroBis}. It is clearly true that $\langle \mathfrak{C}_S^{[m]} \rangle_{\mathrm{gen-sp}} = \mathfrak{A}_{mg} \times_{\A_{mg}} S$ for all $m \ge 1$.\footnote{Here $S$ is a subvariety of $\A_g$, and $\A_g$ is seen as a subvariety of $\A_{mg}$ via the diagonal embedding.} If $S$ has dimension $1$ or has simple connected algebraic monodromy group (for example when $S$ is the whole Torelli locus), then Corollary~\ref{CorBettiRankSimpleBase} has the following immediate corollaries: $\mathfrak{C}_S^{[m]}$ has maximal generic Betti rank for all $m \ge 3$ and $g \ge 2$, $\mathscr{D}_m(\mathfrak{C}_S^{[m+1]})$ has maximal generic Betti rank for all $m \ge 4$ and $g \ge 2$, and $\mathfrak{C}_S^{[m]} - \mathfrak{C}_S^{[m]}$ has maximal generic Betti rank for all $m \ge 4$ when $g \ge 5$, for all $m \ge 5$ when $g = 4$ and for all $m \ge 6$ for $g = 3$.
\end{rmk}

\begin{proof}[Proof of Proposition~\ref{PropDegSimpleBase}]
\begin{enumerate}
\item[(i)] We have seen $\mathrm{rank}_{\R}(\mathrm{d}b_{\Delta}|_X) \le 2\dim \iota(X)$ at the end of $\mathsection$\ref{SectionBettiRevisited}. On the other hand it is not hard to check $\mathrm{rank}_{\R}(\mathrm{d}b_{\Delta}|_{\langle X \rangle_{\mathrm{gen-sp}}}) = 2 (\dim \langle X \rangle_{\mathrm{gen-sp}} - \dim S)$ by the definitions of $b_{\Delta}$ and $\langle X \rangle_{\mathrm{gen-sp}}$. Hence we are done because $\mathrm{rank}_{\R}(\mathrm{d}b_{\Delta}|_X) \le \mathrm{rank}_{\R}(\mathrm{d}b_{\Delta}|_{\langle X \rangle_{\mathrm{gen-sp}}})$.

\item[(ii)] 
Set $l = \min(\dim\iota(X), \dim \langle X \rangle_{\mathrm{gen-sp}}- \dim S)$. Assume $\mathrm{rank}_{\R}(\mathrm{d}b_{\Delta}|_X) < 2l$. Then by Theorem~\ref{ThmCriterionDegIntro}, there exists an abelian subscheme $\cB$ of $\cA_X \rightarrow B$ (whose relative dimension we denote by $g_{\cB}$) such that 
\begin{equation}\label{EqDegInequaAAA}
\dim (\iota_{/\cB} \circ p_{\cB})(X) < \min(\dim\iota(X), \dim \langle X \rangle_{\mathrm{gen-sp}}- \dim S) - g_{\cB}
\end{equation}
for
\begin{equation}\label{EqDiagInequaAAA}
\xymatrix{
\cA_X \ar[r]^-{p_{\cB}} \ar[d]_{\pi_S|_{\cA_X}} & \cA_X/\cB \ar[d]_{\pi_{/\cB}} \ar[r]^-{\iota_{/\cB}} \pullbackcorner & \mathfrak{A}_{g_X-g_{\cB}} \ar[d]^{\pi'} \\
S \ar[r]^-{\mathrm{id}_S} & S \ar[r]^-{\iota_{/\cB,S}} & \A_{g_X-g_{\cB}} .
}
\end{equation}

Observe that $\langle X \rangle_{\mathrm{gen-sp}} \subseteq p_{\cB}^{-1}\left(\iota_{/\cB}^{-1}\left( \langle (\iota_{/\cB} \circ p_{\cB})(X) \rangle_{\mathrm{gen-sp}}\right)\right)$. So
\begin{align*}
\dim \langle X \rangle_{\mathrm{gen-sp}} - \dim S \le \dim(\cB\rightarrow S) + \dim \langle (\iota_{/\cB} \circ p_{\cB})(X) \rangle_{\mathrm{gen-sp}} = g_{\cB} + \dim \langle (\iota_{/\cB} \circ p_{\cB})(X) \rangle_{\mathrm{gen-sp}},
\end{align*}
and thus $(\iota_{/\cB} \circ p_{\cB})(X) \not= \langle (\iota_{/\cB} \circ p_{\cB})(X) \rangle_{\mathrm{gen-sp}}$ by \eqref{EqDegInequaAAA}. So $\iota_{/\cB,S}(S)$ is not a point.

Apply the modular map $\iota \colon \cA \rightarrow \mathfrak{A}_g$ (and $\iota_S \colon S \rightarrow \mathbb{A}_g$) to \eqref{EqDiagInequaAAA}. Denote by $B = \iota_S(S)$. We have
\[
\xymatrix{
\iota(\cA_X) \ar[r]^-{p'_{\cB}} \ar[d]_{\pi|_{\iota(\cA_X)}} & \iota(\cA_X)/\iota(\cB) \ar[d] \ar[r]^-{\iota'_{/\cB}} \pullbackcorner & \mathfrak{A}_{g_X-g_{\cB}} \ar[d]^{\pi'} \\
B \ar[r]^-{\mathrm{id}_B} & B \ar[r]^-{\iota'_{/\cB,S}} & \A_{g_X-g_{\cB}} .
}
\]

Let $G_B$ be the connected algebraic monodromy group of $B \subseteq \mathbb{A}_g$, and let $G_Q$ be the Mumford-Tate group of $B \subseteq \A_g$. Then $G_B < G_Q^{\mathrm{der}}$.

We have that $(\iota(\cA_X)/\iota(\cB))|_{\iota_{/\cB,S}^{'-1}(a)}$ is an isotrivial abelian scheme over $\iota_{/\cB,S}^{'-1}(a)$. Take $a$ to be Hodge generic in $\iota'_{/\cB,S}(B)$, then by a theorem of Deligne-Andr\'{e} \cite[$\mathsection$5, Theorem~1]{AndreMumford-Tate-gr}, the connected algebraic monodromy group $H$ of $\iota_{/\cB,S}^{'-1}(a)$ is a normal subgroup of $G_Q^{\mathrm{der}}$. Moreover $H<G_B$ since $\iota_{/\cB,S}^{'-1}(a) \subseteq B$. Thus $H\lhd G_B$.

Recall the assumption: either $\dim B = 1$ or $H$ is simple. In the first case, either $\iota'_{/\cB,S}(B)$ is a point or $\iota'_{/\cB,S}$ is quasi-finite. In the second case, either $H = G_B$ or $H = 1$ by the previous paragraph. If $H = G_B$, then the fact that $(\iota(\cA_X)/\iota(\cB))|_{\iota_{/\cB,S}^{'-1}(a)} \rightarrow \iota_{/\cB,S}^{'-1}(a)$ is isotrivial implies that $\iota(\cA_X)/\iota(\cB) \rightarrow B$ is isotrivial. Hence $\iota'_{/\cB,S}(B)$ is a point.  If $H = 1$, then $\dim \iota_{/\cB,S}^{'-1}(a) = 0$, and hence $\iota'_{/\cB,S}$ is generically quasi-finite. 

As $\iota'_{/\cB,S}(B) = \iota_{/\cB,S}(S)$ is not a point, we have that $\iota'_{/\cB,S}$ is generically quasi-finite. So $\iota'_{\cB}$ is generically quasi-finite, and so 
$\dim (\iota_{\cB}\circ p_{\cB})(X) = \dim (\iota'_{/\cB} \circ p'_{\cB})(\iota(X)) = \dim p'_{\cB}(\iota(X)) \ge \dim \iota(X) - g_{\cB}$, 
contradicting \eqref{EqDegInequaAAA}. \qedhere
\end{enumerate}
\end{proof}

\section{Decomposition of abelian subschemes}\label{SectionAppendixLinearAlgebra}
The goal of this appendix is to prove \eqref{EqNewDecomposition}.

\subsection{Linear Algebra}
Let $G$ be a reductive group and let $V$ be a finite dimensional $G$-representation, both defined over $\mathbb{Q}$.

Let $m \ge 1$ be an integer and let $W$ be a $G$-submodule of $V^{\oplus m}$. The goal of this section is to prove the following proposition.
\begin{prop}\label{PropLinearAlgebra}
There exists a $G$-linear isomorphism
\[
\tau \colon V^{\oplus m} \cong V^{\oplus m}
\]
such that
\[
\tau(W) = W_1 \oplus \cdots \oplus W_m
\]
for some $G$-submodules $W_1,\ldots,W_m$ of $V$. Moreover $\tau$ can be chosen such that $q_1\circ \tau = q_1$ where $q_1 \colon V^{\oplus m} \rightarrow V$ is the projection to the first factor.
\end{prop}

The key to prove this proposition is the following lemma.
\begin{lemma}\label{LemmaLinearAlgebra}
Let $q \colon V^{\oplus m} \rightarrow V^{\oplus m'}$ be a linear projection. Then there exists a $G$-homomorphism $i \colon V^{\oplus m'} \rightarrow V^{\oplus m}$ such that
\begin{enumerate}
\item[(i)] $q\circ i = 1_{V^{\oplus m'}}$;
\item[(ii)] $i(q(W)) = i(V^{\oplus m'}) \cap W$;
\item[(iii)] $i|_{q(W)}$ is injective.
\end{enumerate}
\end{lemma}
\begin{proof}
As $G$ is a reductive group and $\mathrm{char}\mathbb{Q}=0$, there exists a $G$-submodule $W_0$ of $W$ such that $W = (W\cap \ker q) \oplus W_0$. In particular $W_0 \cap \ker q = W_0 \cap W \cap \ker q = \{0\}$. Thus $W_0 + \ker q = W_0 \oplus \ker q$. Note that $q(W) = q(W_0)$, and $q|_{W_0}$ is injective.

Now $W_0 \oplus \ker q$ is a $G$-submodule of $V^{\oplus m}$. Again as $G$ is a reductive group and $\mathrm{char}\mathbb{Q}=0$, there exists a $G$-submodule $W_0'$ of $V^{\oplus m}$ such that $V^{\oplus m} = (W_0 \oplus \ker q) \oplus W_0'$. Moreover we claim that $W_0' \cap W = 0$. Indeed, $W = (W\cap \ker q) \oplus W_0 \subseteq \ker q \oplus W_0$, and $W_0' \cap (\ker q \oplus W_0) = 0$. Thus $(W_0 \oplus W_0') \cap W = W_0$.

As $V^{\oplus m} = (W_0 \oplus \ker q) \oplus W_0'$, we have that $q|_{W_0 \oplus W_0'} \colon W_0 \oplus W_0' \rightarrow V^{\oplus m'}$ is injective and both sides have the same dimension. So $q|_{W_0 \oplus W_0'}$ is an isomorphism.

Let $i \colon V^{\oplus m'} \rightarrow V^{\oplus m}$ be the composite of the inclusion $W_0 \oplus W_0' \subseteq V^{\oplus m}$ with $(q|_{W_0 \oplus W_0'})^{-1}$. Then $i(V^{\oplus m'}) = W_0\oplus W_0'$.

Let us show that $i$ is the desired map. Property (i) clearly holds. To see (ii) and (iii), recall that $q(W) = q(W_0)$, and hence $i|_{q(W)} = i|_{q(W_0)} = (q|_{W_0 \oplus W_0'})^{-1}|_{q(W_0)} \colon q(W_0) \cong W_0 = (W_0 \oplus W_0') \cap W  = i(V^{\oplus m'}) \cap W$.
\end{proof}

\begin{proof}[Proof of Proposition~\ref{PropLinearAlgebra}]
We prove the proposition by induction on $m$. When $m = 1$, the proposition trivially holds true.

Now for a general $m \ge 1$, suppose the proposition is proved for $1,\ldots, m-1$. 

Apply Lemma~\ref{LemmaLinearAlgebra} to $q_1 \colon V^{\oplus m} \rightarrow V$ the projection to the first component. We thus obtain $i_1 \colon V \rightarrow V^{\oplus m}$ with the three properties. In particular $V^{\oplus m} = i_1(V) \oplus \ker q_1$.


Let $q'_1 \colon V^{\oplus m} \rightarrow V^{\oplus (m-1)}$ be the quotient by $i_1(V)$. Then $\ker q'_1 \cap W = i_1(V) \cap W = i_1(q_1(W))$ by property (ii) of $i_1$. Let $i'_1 \colon V^{\oplus (m-1)} \rightarrow V^{\oplus m}$, $v \mapsto (0,v)$. Then $i'_1(V^{\oplus (m-1)}) = \ker q_1$ and from the last paragraph we have $V^{\oplus m} = i_1(V) \oplus i_1'(V^{\oplus (m-1)})$. Define
\begin{equation}\label{EqTau1}
\tau_1 \colon V^{\oplus m} \xrightarrow{(q_1, q'_1)} V \oplus V^{\oplus (m-1)} \xrightarrow{i_1 + i_1'} V^{\oplus m}.
\end{equation}
Then $\tau_1$ is a $G$-isomorphism, and $\tau_1(W) \subseteq i_1\circ q_1(W) \oplus i'_1 \circ q'_1(W)$. But $\dim i'_1 \circ q'_1(W) = \dim q'_1(W) = \dim W - \dim (\ker q'_1 \cap W) = \dim W - \dim i_1(q_1(W))$. Hence $\dim W = \dim \tau_1(W) \le \dim i_1\circ q_1(W) + \dim i'_1 \circ q'_1(W) = \dim i_1\circ q_1(W) + \dim W - \dim i_1\circ q_1(W) = \dim W$. Hence $\tau_1(W) = i_1\circ q_1(W) \oplus i'_1 \circ q'_1(W)$.

Set $W_1 = i_1\circ q_1(W)$ and $W_1^{\perp} = i'_1 \circ q'_1(W)$.

By induction hypothesis applied to $W_1^{\perp} \subseteq V^{\oplus (m-1)}$, there exists a $G$-linear $\tau_1' \colon V^{\oplus (m-1)} \cong V^{\oplus (m-1)}$ 
such that $\tau_1'(W_1^{\perp}) = W_2 \oplus \cdots \oplus W_m$ for some $G$-submodules $W_2,\ldots,W_m$ of $V$.

Let $\tau$ be the composite
\[
\tau \colon V^{\oplus m} \xrightarrow{\tau_1} V^{\oplus m} = V \oplus V^{\oplus (m-1)} \xrightarrow{(1_V,\tau_1')} V \oplus V^{\oplus (m-1)} = V^{\oplus m}.
\]
Then $\tau(W) = W_1 \oplus W_2 \oplus \cdots \oplus W_m$. Hence we are done for the construction of $\tau$.

Let us prove the ``Moreover'' part. By our construction of $\tau$, we have $q_1\circ \tau = q_1\circ (1_V, \tau_1') \circ \tau_1 =q_1\circ \tau_1 = q_1\circ (i_1 \circ q_1 + i'_1 \circ q'_1) = q_1$ as $q_1 \circ i_1 = 1_V$ and $q_1 \circ i'_1 = 0$.
\end{proof}

\subsection{Proof of \eqref{EqNewDecomposition}}
Now we are ready to prove \eqref{EqNewDecomposition}. Let $\pi_S \colon \cA \rightarrow S$ be an abelian scheme, $\cA^{[m]}$ be the $m$-fold fibered power, and $\cB$ be an abelian subscheme of $\pi_S^{[m]} \colon \cA^{[m]} \rightarrow S$.

The variation of Hodge structures (V$\mathbb{Q}$HS) $(R^1 (\pi_S)_*\mathbb{Q})^\vee$ over $S$ is polarizable of type $(-1,0)+(0,-1)$. Its generic Mumford--Tate group $G$ is a reductive group defined over $\mathbb{Q}$. 

As $\mathrm{End}_{\text{V}\mathbb{Q}-\text{HS}}((R^1 (\pi_S)_*\mathbb{Q})^\vee)$ is a variation of Hodge structures of $S$ of weight $0$, there exists $s \in S(\mathbb{C})$ such that each element in $\mathrm{End}_{\mathbb{Q}-\text{HS}}((R^1 (\pi_S)_*\mathbb{Q})^\vee_s)$ extends to an element in $\mathrm{End}_{\text{V}\mathbb{Q}\text{HS}}((R^1 (\pi_S)_*\mathbb{Q})^\vee)$, up to replacing $S$ by a finite covering. See \cite[proof of Theorem~10.20]{PetersMixed-Hodge-Str}.

Set $V=(R^1 (\pi_S)_*\mathbb{Q})^\vee_s = H_1(\cA_s,\mathbb{Q})$. 
Apply Proposition~\ref{PropLinearAlgebra} to $W = H_1(\cB_s,\mathbb{Q})$. We obtain a $G$-isomorphism $\tau \colon V^{\oplus m} \cong V^{\oplus m}$ such that $
\tau(W) = W_1 \oplus \cdots \oplus W_m$
for some $G$-submodules $W_1,\ldots,W_m$ of $V$. Moreover $q_1 \circ \tau = q_1$ with $q_1 \colon V^{\oplus m} \rightarrow V$ the projection to the first factor.

For each $i \in \{1,\ldots,m\}$, $W_i = \ker \alpha_i$ for some $\alpha_i \in \mathrm{End}_G(V)$. But $\mathrm{End}_G(V) \subseteq \mathrm{End}_{\mathbb{Q}-\text{HS}}(V)$, so by the discussion above each $\alpha_i$ extends to an element in $\mathrm{End}_{\text{V}\mathbb{Q}\text{HS}}((R^1 (\pi_S)_*\mathbb{Q})^\vee)$, which by abuse of notation is still denoted by $\alpha_i$. Similarly $\tau$ extends to some $\rho \in \mathrm{End}_{\text{V}\mathbb{Q}\text{HS}}((R^1 (\pi_S^{[m]})_*\mathbb{Q})^\vee)$. We then have $\rho\left((R^1 (\pi_S^{[m]}|_{\cB})_*\mathbb{Q})^\vee\right) = \ker(\alpha_1) \times_S \ldots \times_S \ker(\alpha_m)$.

By \cite[Rappel~4.4.3]{DeligneHodgeII}, $\rho$ gives rise to an isogeny $\cA^{[m]} \rightarrow \cA^{[m]}$ and each $\alpha_i$ gives rise to an isogeny $\cA \rightarrow \cA$. By abuse of notation we still use $\rho$, $\alpha_i$ to denote them. Then $\rho(\cB) = \ker(\alpha_i) \times_S \ldots \times_S \ker(\alpha_m)$. It suffices to take $\cB_i = \ker(\alpha_i)$.

The ``Moreover'' part of \eqref{EqNewDecomposition} follows from the same argument and the equality $q_1 \circ \tau = q_1$ above.

\bibliographystyle{alpha}
\bibliography{bibliography}


\end{document}